\documentclass[11pt]{amsart}

\usepackage{graphicx}
\usepackage{mathptmx}
\usepackage{amsmath}
\usepackage{amssymb}
\usepackage{amsthm}
\usepackage{tikz}
\usepackage{subcaption}
\usepackage{algpseudocode}
\usepackage{mdwlist}

\theoremstyle{plain}
\newtheorem{theorem}{Theorem}[section]
\newtheorem{proposition}[theorem]{Proposition}
\newtheorem{lemma}[theorem]{Lemma}
\newtheorem{corollary}[theorem]{Corollary}

\theoremstyle{definition}
\newtheorem{definition}[theorem]{Definition}
\newtheorem{example}[theorem]{Example}
\newtheorem{remark}[theorem]{Remark}

\DeclareMathOperator{\lk}{lk}
\DeclareMathOperator{\dl}{\text{dl}}

\DeclareMathOperator{\DEG}{\text{Deg}}

\newcommand{\bone}{rib}

\newcommand{\longmapsfrom}{\mathrel{\reflectbox{\ensuremath{\longmapsto}}}}

\newenvironment{enum*}
{\parskip=-4pt
\begin{enumerate*}

}
{\end{enumerate*}}

\begin{document}

\title{Generalized Macaulay representations and the flag $f$-vectors of generalized colored complexes}
\date{}
\author{Kai Fong Ernest Chong}
\address{Department of Mathematics\\
      Cornell University\\
      Ithaca, NY 14853-4201, USA}
\email{kc343@cornell.edu}


\begin{abstract}
A colored complex of type ${\bf a} = (a_1, \dots, a_n)$ is a simplicial complex $\Delta$ on a vertex set $V$, together with an ordered partition $(V_1, \dots, V_n)$ of $V$, such that every face $F$ of $\Delta$ satisfies $|F \cap V_i| \leq a_i$. For each ${\bf b} = (b_1, \dots, b_n) \leq {\bf a}$, let $f_{\bf b}$ be the number of faces $F$ of $\Delta$ such that $|F \cap V_i| = b_i$. The array of integers $\{f_{\bf b}\}_{{\bf b} \leq {\bf a}}$ is called the fine $f$-vector of $\Delta$, and it is a refinement of the $f$-vector of $\Delta$. In this paper, we generalize the notion of Macaulay representations and give a numerical characterization of the fine $f$-vectors of colored complexes of arbitrary type, in terms of these generalized Macaulay representations. As part of the proof, we introduce the property of ${\bf a}$-Macaulay decomposability for simplicial complexes, which implies vertex-decomposability, and we show that every pure color-shifted balanced complex $\Delta$ of type ${\bf a}$ is ${\bf a}$-Macaulay decomposable. Combined with previously known results, we also obtain a numerical characterization of the flag $f$-vectors of completely balanced Cohen-Macaulay complexes.
\end{abstract}

\maketitle

\section{Introduction and Overview}\label{sec:Intro}
The Kruskal-Katona theorem~\cite{Katona1968,Kruskal1963,Schutzenberger1959} is a fundamental result in geometric combinatorics that gives a numerical characterization of the $f$-vectors of simplicial complexes. Since its proof in the 1960s, there has been much work on finding analogous numerical characterizations for various classes of simplicial complexes, multicomplexes, and polytopal complexes; see \cite{Macaulay1927,Nevo2006:GeneralizedWegnerTheorem,Stanley1975:UBCandCohenMacaulayRings,Stanley1977:CohenMacaulayComplexes,book:StanleyGreenBook,Wegner1984}. A common theme in many of these results is to prove that a list of several classes of complexes share the same $f$-vectors or $h$-vectors, and use the fact that an explicit numerical characterization is known for one of these classes. This strategy, albeit fruitful, is inherently limited by the numerical characterizations that are already known. Such characterizations are usually expressed in terms of Macaulay representations and suitably defined differentials on these representations, yet one fundamental obstacle remains: Macaulay representations are generally not well-suited for characterizing complexes with additional combinatorial structure.

In this paper, we introduce a generalized notion of Macaulay representations. Our motivation stems from the fact that colored complexes and completely balanced Cohen-Macaulay complexes, which include Coxeter complexes and order complexes of posets as important subclasses, are well-studied classes of complexes for which complete numerical characterizations of their fine $f$-vectors were previously unknown. Our main goal is to give numerical characterizations for these two classes in terms of generalized Macaulay representations.

Throughout this paper, a {\it colored complex} of type ${\bf a} = (a_1, \ldots, a_n)$ is a pair $(\Delta, \pi)$, where $\Delta$ is a simplicial complex with vertex set $V$, and $\pi = (V_1, \dots, V_n)$ is an ordered partition of $V$, such that every face $F$ of $\Delta$ satisfies $|F \cap V_i| \leq a_i$. Note that the usual notion of colored complexes many authors use is equivalent to our definition of colored complexes of type $(1, \dots, 1)$. A $(d-1)$-dimensional {\it balanced complex} of type ${\bf a}$ is a colored complex $\Delta$ of type ${\bf a}$ satisfying $\sum a_i = d$, while a $(d-1)$-dimensional {\it completely balanced complex} is a balanced complex of type ${\bf 1}_d := (1,\dots, 1)$. For brevity, we write `CM', `${\bf a}$-balanced' and `${\bf a}$-colored' to mean `Cohen-Macaulay', `balanced of type ${\bf a}$' and `colored of type ${\bf a}$' respectively. For each ${\bf b} = (b_1, \dots, b_n) \leq {\bf a}$, let $f_{\bf b}$ be the number of faces $F$ of $\Delta$ such that $|F \cap V_i| = b_i$. The array of integers $\{f_{\bf b}\}_{{\bf b} \leq {\bf a}}$ is called the {\it fine $f$-vector} of $\Delta$, and it is a refinement of the $f$-vector of $\Delta$. Although some authors call this array the flag $f$-vector, we instead reserve the notion of `flag $f$-vector' to mean something else that is closely related. All other relevant terminology that we use in the rest of this section will be defined in later sections.

Around the 1980s, Bj\"{o}rner, Frankl and Stanley~\cite{BjornerFranklStanley1987:CohenMacaulayComplexes,Stanley1979:BalancedCohenMacaulayComplexes} proved the following combinatorial characterization:

\begin{theorem}\label{thm:BFSthm}
Let ${\bf a} \in \mathbb{P}^n$, and let $f = \{f_{\bf b}\}_{{\bf b} \leq {\bf a}}$ be an array of integers. The following are equivalent:
\begin{enum*}
\item $f$ is the fine $h$-vector of an ${\bf a}$-balanced CM complex.
\item $f$ is the fine $h$-vector of an ${\bf a}$-balanced pure shellable complex.
\item $f$ is the fine $f$-vector of an ${\bf a}$-colored multicomplex.
\item $f$ is the fine $f$-vector of a color-compressed ${\bf a}$-colored multicomplex.
\end{enum*}
\end{theorem}

Subsequently, Babson and Novik~\cite{BabsonNovik2006:FaceNumbersNongenericInitialIdeals} proved that the fine $h$-vector of an ${\bf a}$-balanced CM complex is the fine $h$-vector of a color-shifted ${\bf a}$-balanced CM complex, and that a color-shifted balanced complex is CM if and only if it is pure. Also, Biermann and Van Tuyl~\cite{BiermannVanTuyl2013:BalancedVDSimplicialComplexesAndHVectors} recently proved that the $h$-vector of a completely balanced CM complex is the $h$-vector of a pure completely balanced vertex-decomposable complex. In fact, we prove the stronger statement that the fine $h$-vector of an ${\bf a}$-balanced CM complex is the fine $h$-vector of a pure ${\bf a}$-balanced vertex-decomposable complex (Corollary \ref{cor:BalancedCMListOfEquiv}). Thus, finding a numerical characterization for any of these classes of complexes would yield important enumerative information for all of them.

The Frankl-F\"{u}redi-Kalai theorem~\cite{FFK1988} extends the Kruskal-Katona theorem and gives a numerical characterization of the $f$-vectors of $(1, \dots, 1)$-colored complexes, or equivalently, the $f$-vectors of $(1, \dots, 1)$-colored multicomplexes. However, since its publication in 1988, very little progress has been made towards a numerical characterization of the more refined fine $f$-vectors. As pointed out in \cite{BjornerFranklStanley1987:CohenMacaulayComplexes}, part of the difficulty lies in the non-uniqueness of color-compressed multicomplexes with a given fine $f$-vector. Frohmader~\cite{Frohmader2012:FlagFVectorsOfColoredComplexes} showed it is not possible to make further progress towards a numerical characterization of the fine $f$-vectors of color-shifted $(1, \dots, 1)$-colored complexes through stronger restrictions on the color-selected subcomplexes. As for specific small cases, Walker~\cite{Walker2007:MulticoverInequalitiesOnColoredComplexes} gave a numerical characterization of the fine $f$-vectors of $(1,1)$-colored complexes, while Frohmader~\cite{Frohmader2012:FlagFVectorsOfThreeColoredComplexes} gave a numerical characterization of the fine $f$-vectors of $(1,1,1)$-colored complexes that involves a complicated brute-force check of many cases.

A large portion of this paper is devoted to addressing these difficulties. The central ideal that makes generalized Macaulay representations possible is the notion of `Macaulay decomposability', which we introduce in Section \ref{sec:MacaulayDecomposability}. An ${\bf a}$-Macaulay decomposable simplicial complex is vertex-decomposable (Proposition \ref{prop:Macaulay-decomposable=>vertex-decomposable}), and we show every pure color-shifted ${\bf a}$-balanced complex is ${\bf a}$-Macaulay decomposable and hence vertex-decomposable (Theorem \ref{thm:color-shifted=>Macaulay-decomposable}). This allows a geometric interpretation of `decomposing' ${\bf a}$-balanced complexes into pieces we can better understand.

The experienced reader would notice that Macaulay decomposability is a slight extension of the notion of vertex-decomposability. Nevertheless, it turns out that Macaulay decomposability is the ``correct'' notion in the context of (generalized) Macaulay representations. To illustrate what we mean, consider the following example: Let $\Delta$ be the (unique) pure $2$-dimensional compressed simplicial complex with $6$ facets and linearly ordered vertices $x_1 < \dots < x_5$. Using Macaulay decomposability, we can decompose $\Delta$ into (iterations of) the deletions and links of $\Delta$ at certain distinguished vertices that we call Macaulay shedding vertices. Note that the number $f_k(\Delta)$ of $k$-dimensional faces of $\Delta$ equals $f_k(\dl_{\Delta}(x)) + f_{k-1}(\lk_{\Delta}(x))$, where $\dl_{\Delta}(x)$ and $\lk_{\Delta}(x)$ denote the deletion and link of $\Delta$ respectively at any given vertex $x$ of $\Delta$. The (usual) $3$rd Macaulay representation of $f_2(\Delta) = 6$ is $6 = \bigl(\begin{smallmatrix} 4\\ 3 \end{smallmatrix} \bigr) + \bigl(\begin{smallmatrix} 2\\ 2 \end{smallmatrix} \bigr) + \bigl(\begin{smallmatrix} 1\\ 1 \end{smallmatrix} \bigr)$, and notice that the three summands $\bigl(\begin{smallmatrix} 4\\ 3 \end{smallmatrix} \bigr), \bigl(\begin{smallmatrix} 2\\ 2 \end{smallmatrix} \bigr), \bigl(\begin{smallmatrix} 1\\ 1 \end{smallmatrix} \bigr)$ equal the number of $2$-dimensional faces in $\dl_{\Delta}(x_5)$, the number of edges in $\dl_{\lk_{\Delta}(x_5)}(x_3)$, and the number of vertices in $\lk_{\lk_{\Delta}(x_5)}(x_3)$ respectively. The vertices $x_3, x_5$ are precisely the Macaulay shedding vertices of $\Delta$ that yield this ``decomposition'' of $\Delta$ into three subcomplexes corresponding to the summands. As we will later realize, all the abovementioned information (including the Macaulay shedding vertices and the $3$rd Macaulay representation of $6$) is already encoded in the labeled binary tree illustrated in Figure \ref{Fig:example-intro}. It is helpful to have this example in mind when delving into the details in subsequent sections, which can get rather technical.

\begin{figure}[h!t]
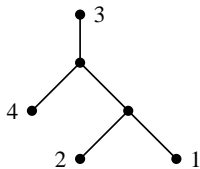

\centering
\scalebox{0.8}{\tikz[thick,scale=0.8]{
    \draw {
    (2,3) node[circle, draw, fill=black, inner sep=0pt, minimum width=4pt, label=right:{$3$}] {} -- (2,2) node[circle, draw, fill=black, inner sep=0pt, minimum width=4pt] {} -- (3,1) node[circle, draw, fill=black, inner sep=0pt, minimum width=4pt] {} -- (4,0) node[circle, draw, fill=black, inner sep=0pt, minimum width=4pt, label=right:{$1$}] {}
    (2,2) -- (1,1) node[circle, draw, fill=black, inner sep=0pt, minimum width=4pt, label=left:{$4$}] {}
    (3,1) -- (2,0) node[circle, draw, fill=black, inner sep=0pt, minimum width=4pt, label=left:{$2$}] {}
    };
}}
\caption{The generalized $3$-Macaulay representation of $6$ (with labels on interior vertices omitted).}
\label{Fig:example-intro}
\end{figure}

In general, we think of generalized Macaulay representations as labeled binary trees that succinctly encode the relevant combinatorial information when we keep track of (iterations of) the deletions and links of suitably defined complexes. The left and right branches of these trees correspond to deletions and links respectively. It is not {\it a priori} obvious that such generalized Macaulay representations are well-defined and can be defined without relying on the existence of certain complexes, so a significant part of this paper (see Sections \ref{subsec:ConstructionOfSheddingTrees}--\ref{subsec:GenMacReps}) involves constructing such representations, finding the ``correct'' definition, and proving that these representations are well-defined purely numerical notions that have associated well-defined differentials.


The rest of the paper is organized as follows. Section \ref{sec:Preliminaries} covers the preliminaries. In Section \ref{sec:ColorCompressionColoredComplexes}, we show that color compression preserves the fine $f$-vectors of colored complexes (Corollary \ref{cor:ColorCompressionPreservesFineFVector}). Section \ref{sec:MacaulayDecomposability} introduces ${\bf a}$-Macaulay decomposability and related results, while in Section \ref{sec:GeneralizedMacaulayRepresentations}, we define generalized ${\bf a}$-Macaulay representations using ${\bf a}$-Macaulay decomposability and relate them to the fine $f$-vectors of ${\bf a}$-colored complexes and completely balanced CM complexes.

\section{Preliminaries}\label{sec:Preliminaries}
Let $\mathbb{N}$ and $\mathbb{P}$ denote the non-negative integers and positive integers respectively, and for convenience, let $\overline{\mathbb{N}} = \mathbb{N} \cup \{\infty\}, \overline{\mathbb{P}} = \mathbb{P} \cup \{\infty\}$. For $n\in \mathbb{P}$ and ${\bf a} = (a_1, \ldots, a_n), {\bf b} = (b_1, \ldots, b_n)$ in $\mathbb{Z}^n$, let $[n]$ be the set $\{1, \ldots, n\}$, let $|{\bf a}| = a_1 + \ldots + a_n$, and write ${\bf a} \leq {\bf b}$ if $a_i \leq b_i$ for all $i\in [n]$. Set $[0] := \emptyset$. By convention, ${\bf a} < {\bf b}$ if ${\bf a} \leq {\bf b}$ and ${\bf a} \neq {\bf b}$. Also, write ${\bf a} \lessdot {\bf b}$ if ${\bf a} < {\bf b}$ and $|{\bf b}| = |{\bf a}| + 1$. For brevity, we use ${\bf 0}_n$ and ${\bf 1}_n$ to mean the $n$-tuples $(0,\dots, 0)$ and $(1, \dots, 1)$ respectively. Denote the Kronecker delta function by $\delta(i,j)$, i.e. $\delta(i,j) = 1$ if $i=j$, $\delta(i,j) = 0$ if $i\neq j$, and define $\boldsymbol{\delta}_{i,n} = (\delta(i,1), \delta(i,2), \dots, \delta(i,n)) \in \mathbb{N}^n$. Given $N,k\in \mathbb{P}$, it is easy to show there exists a unique expansion
\begin{equation*}
N = \binom{N_k}{k} + \binom{N_{k-1}}{k-1} + \dots + \binom{N_j}{j},
\end{equation*}
such that $N_k > N_{k-1} > \dots > N_j \geq j \geq 1$. (See, e.g. \cite[Section 8]{Greene-Kleitman}, for a proof.) Such an expansion is called the {\it $k$-th Macaulay representation of $N$}.

Suppose $\pi = (X_1, \dots, X_n)$ is a sequence of (possibly empty) subsets of a set $X$. For any $W \subseteq X$, define $\pi \cap W := (X_1 \cap W, \dots, X_n \cap W)$. If the subsets $X_1, \dots, X_n$ are not all empty, and if $1\leq i_1 < \dots < i_m \leq n$ are all the indices such that $X_{i_t} \neq \emptyset$ for each $t\in [m]$, then define $\overline{\pi} := (X_{i_1}, \dots, X_{i_m})$. Analogously, given ${\bf x} = (x_1, \dots, x_n) \in \mathbb{N}^n\backslash \{{\bf 0}_n\}$, if $1\leq i_1 < \dots < i_m \leq n$ are all the indices such that $x_{i_t} > 0$ for each $t\in [m]$, then define $\overline{{\bf x}} := (x_{i_1}, \dots, x_{i_m}) \in \mathbb{P}^m$.

An {\it ordered partition} of a non-empty set $S$ is a finite sequence $(S_1, \dots, S_n)$ of non-empty, pairwise disjoint subsets of $S$ satisfying $S = S_1 \cup \dots \cup S_n$. For each $k\in \mathbb{N}$, let $\binom{S}{k}$ denote the set of all $k$-subsets of $S$, which by default is empty if $|S| < k$. If $S$ is linearly ordered, then the {\it colex} (colexicographic) order $\leq_{c\ell}$ on $\binom{S}{k}$ is defined by $A <_{c\ell} B$ if and only if $\max\{i:i\in A-B\} < \max\{j:j\in B-A\}$. This is sometimes known as the reverse lexicographic order or the squashed order. Let $X = \{x_1, x_2, \dots\}$ be a set of variables linearly ordered by $x_1 < x_2 < \dots$, and let $\mathcal{M}_X^d$ be the collection of all monomials in variables $X$ of degree $d\in \mathbb{N}$. The {\it lex} (lexicographic) order $\leq_{\ell ex}$ on $\mathcal{M}_X^d$ induced by this linear order on $X$ is defined by $x_1^{\alpha_1}x_2^{\alpha_2}x_3^{\alpha_3}\cdots  <_{\ell ex} x_1^{\beta_1}x_2^{\beta_2}x_3^{\beta_3}\cdots$ if and only if $\alpha_i < \beta_i$ for the largest $i\in \mathbb{P}$ such that $\alpha_i \neq \beta_i$. The {\it rev-lex} (reverse lexicographic) order $\leq_{r\ell}$ on $\mathcal{M}_X^d$ induced by this linear order on $X$ is defined by $x_1^{\alpha_1}x_2^{\alpha_2}x_3^{\alpha_3}\cdots  <_{r\ell} x_1^{\beta_1}x_2^{\beta_2}x_3^{\beta_3}\cdots$ if and only if $\alpha_i > \beta_i$ for the smallest $i\in \mathbb{P}$ such that $\alpha_i \neq \beta_i$. By identifying each squarefree monomial $x_{i_1}\cdots x_{i_d} \in \mathcal{M}_X^d$ with $\{i_1, \dots, i_d\} \in \binom{\mathbb{P}}{d}$, the induced lex order on the subcollection of all squarefree monomials in $\mathcal{M}_X^d$ is equivalent to the colex order defined on $\binom{\mathbb{P}}{d}$.

\subsection{Colored Complexes and Balanced Complexes}
A {\it simplicial complex} $\Delta$ on a vertex set $V$ is a collection of subsets of $V$ such that $\{v\} \in \Delta$ for all $v\in V$ and $\Delta$ is closed under set inclusion (i.e. $F \in \Delta, F^{\prime} \subseteq F \Rightarrow F^{\prime} \in \Delta$). In this paper, we always assume $V$ is finite and non-empty. Elements of $\Delta$ are called {\it faces}. The {\it dimension} of each $F \in \Delta$ is $\dim F := |F| - 1$, and the {\it dimension} of $\Delta$, denoted by $\dim \Delta$, is the maximum dimension of its faces. Maximal faces are called {\it facets}, $0$-dimensional faces are called {\it vertices}, and $1$-dimensional faces are called {\it edges}. The {\it $f$-vector} of $\Delta$ is $(f_0, \dots, f_{\dim \Delta})$, where each $f_i$ is the number of $i$-dimensional faces of $\Delta$. By default, set $f_{-1} = 1$, which corresponds to the empty face $\emptyset \in \Delta$. If all facets of $\Delta$ have the same dimension, then we say $\Delta$ is {\it pure}. Given another simplicial complex $\Delta^{\prime}$ with vertex set $V^{\prime}$, we say $\Delta$ and $\Delta^{\prime}$ are {\it isomorphic} if there exists a bijection $\nu: V \to V^{\prime}$ such that $F \in \Delta$ if and only if $\nu(F) \in \Delta^{\prime}$. Given an arbitrary collection $\mathcal{F} = \{F_1, \dots, F_m\}$ of non-empty subsets of $V$, there is a (unique) smallest simplicial complex, denoted by $\langle \mathcal{F} \rangle$, which contains all $F_i$. This simplicial complex is said to be {\it generated} by $\mathcal{F}$, and a simplicial complex generated by one face, i.e. $|\mathcal{F}| = 1$, is called a {\it simplex}.

A {\it subcomplex} of $\Delta$ is a subcollection of $\Delta$ that is also a simplicial complex. The {\it $k$-skeleton} of $\Delta$ is the subcomplex $\{F \in \Delta: \dim F \leq k\}$. The {\it deletion} of a face $\sigma$ from $\Delta$ is the subcomplex $\dl_{\Delta}(\sigma) := \{\tau \in \Delta: \sigma \not\subseteq \tau\}$. The {\it link} of a face $\sigma$ in $\Delta$ is the subcomplex $\lk_{\Delta}(\sigma) := \{\tau \in \Delta: \tau \cap \sigma = \emptyset \text{ and }\tau \cup \sigma \in \Delta\}$. The {\it join} of two simplicial complexes $\Delta_1$, $\Delta_2$ with disjoint vertex sets is the simplicial complex $\Delta_1 * \Delta_2 := \{\sigma_1 \cup \sigma_2: \sigma_1 \in \Delta_1, \sigma_2 \in \Delta_2\}$.

Let $R = \mathbb{F}[x_1, \dots, x_n]$ be a polynomial on $n$ variables over a field $\mathbb{F}$, and let $\Delta$ be a simplicial complex with vertex set $V = \{v_1, \dots, v_n\}$. The {\it Stanley-Reisner ring} of $\Delta$ (over $\mathbb{F}$) is the ring $\mathbb{F}[\Delta] := R/I_{\Delta}$, where $I_{\Delta}$ is the {\it Stanley-Reisner ideal} of $\Delta$ given by $I_{\Delta} := \big\langle \big\{x_{i_1}\dots x_{i_m}: \{v_{i_1}, \dots, v_{i_m}\} \not\in \Delta\big\}\big\rangle$. We say $\Delta$ is {\it Cohen-Macaulay} (CM) over $\mathbb{F}$ if $\mathbb{F}[\Delta]$ is a Cohen-Macaulay ring. For a detailed treatment of CM complexes and their significance, see \cite{book:StanleyGreenBook}.

Balanced complexes (defined in Section \ref{sec:Intro}) are not necessarily pure, in contrast to the original definition introduced by Stanley~\cite{Stanley1979:BalancedCohenMacaulayComplexes}. Note also that some authors use the notion ``balanced'' to mean what we call completely balanced. We can produce examples of balanced complexes of arbitrary type ${\bf a} \in \mathbb{P}^n$ from a completely balanced complex $(\Delta, \pi)$ satisfying $\dim \Delta = |{\bf a}|-1$ as follows: If $d = |{\bf a}|$ and $\pi = (V_1, \dots, V_d)$ is the ordered partition of some vertex set $V$, then any surjective map $\phi: [d] \to [n]$ induces the ordered partition $\pi^{\prime} = (V_1^{\prime}, \dots, V_n^{\prime})$ of $V$ such that $v\in V_i$ implies $v \in V_{\phi(i)}^{\prime}$, and we check that $(\Delta, \pi^{\prime})$ is an ${\bf a}$-balanced complex. However, as pointed out by Swartz~\cite[Section 3]{Swartz2006:FaceEnumerationSpheresToManifolds}, not all balanced complexes arise from completely balanced complexes in this manner.

Given ${\bf a} \in \mathbb{P}^n$, let $(\Delta, \pi)$ be a $(d-1)$-dimensional ${\bf a}$-colored complex with fine $f$-vector $\{f_{{\bf b}}\}_{{\bf b} \leq {\bf a}}$. The {\it fine $h$-vector} of $(\Delta, \pi)$ is the array of integers $\{h_{{\bf b}}\}_{{\bf b} \leq {\bf a}}$ defined by 
\begin{equation}\label{eqn:fine-h-vector-defn}
h_{{\bf b}} := \sum_{{\bf c} \leq {\bf b}} f_{{\bf c}} \prod_{i=1}^n (-1)^{b_i - c_i} \binom{a_i - c_i}{b_i - c_i}, \ \ \ \ {\bf 0}_n \leq {\bf b} \leq {\bf a},
\end{equation}
and the vector $(h_0, \dots, h_d)$, where $h_i = \sum_{|{\bf b}| = i} h_{{\bf b}}$ for each $i$, is called the {\it $h$-vector} of $\Delta$. If $(\Delta, \pi)$ is completely balanced, i.e. of type ${\bf 1}_d$, we can identify each $d$-tuple ${\bf b} = (b_1, \dots, b_d) \in \mathbb{N}^d$ satisfying ${\bf b} \leq {\bf 1}_d$ with the subset $\{i\in [d]: b_i = 1\}$ of $[d]$. Under this identification, the corresponding arrays $\{f_S\}_{S \subseteq [d]}$ and $\{h_S\}_{S \subseteq [d]}$ are called the {\it flag $f$-vector} and {\it flag $h$-vector} of $(\Delta, \pi)$ respectively.

\subsection{Colored Multicomplexes}
A {\it multicomplex} $M$ on a set of variables $X$ is a collection of monomials in these variables that is closed under divisibility (i.e. $m\in M, m^{\prime}|m \Rightarrow m^{\prime} \in M$). In this paper, we always assume $M$ is non-empty, but we do not require every $x\in X$ to be in $M$. A {\it subcomplex} of $M$ is a subcollection of $M$ that is also a multicomplex. Given $Y \subseteq X$ and any monomial $m = \prod_{x\in X} x^{c(x)}$ in $M$, let $m_Y := \prod_{x\in Y} x^{c(x)}$, let $\deg(m) := \sum_{x\in X} c(x)$ be the degree of $m$, and define the subcomplex $M_Y := \{m_Y: m\in M\}$. By default, $1 \in M_Y$. For each $d\in \mathbb{N}$, let $M^d$ be the collection of monomials in $M$ of degree $d$. The {\it $f$-vector} of $M$ is $(f_0, f_1, \dots)$, where $f_i = |M^i|$ for each $i\in \mathbb{N}$.

Given a set of variables $X = \{x_1, x_2, \dots\}$ and any map $\phi: X \to \overline{\mathbb{P}}$, let ${\bf e} = (e_1, e_2, \dots) := (\phi(x_1), \phi(x_2), \dots)$, and let $\mathcal{M}_X({\bf e})$ be the set of all monomials $x_1^{c_1}x_2^{c_2}\cdots$ such that $0\leq c_i \leq e_i$ for each $i$. For every $d\in \mathbb{N}$, let $\mathcal{M}_X^d({\bf e})$ be the subset of monomials in $\mathcal{M}_X({\bf e})$ of degree $d$. In particular, $\mathcal{M}_X({\bf 1}_n)$ is the set of all squarefree monomials on variables $x_1,\dots, x_n$. Note that every simplicial complex $\Delta$ on a vertex set $V \subseteq X$ corresponds bijectively to a finite multicomplex $M \subseteq \mathcal{M}_X({\bf 1}_n)$ via $\{x_{i_1}, \dots, x_{i_t}\} \in \Delta \Leftrightarrow {x_{i_1}}\cdots {x_{i_t}} \in M$, so multicomplexes can be considered as generalizations of simplicial complexes. Note however that the $f$-vector of $\Delta$ and the $f$-vector of $M$ differ by a shift in the indexing.

Let ${\bf a} = (a_1, \ldots, a_n) \in \mathbb{P}^n$. A {\it colored multicomplex of type ${\bf a}$} is a pair $(M, \pi)$, where $M$ is a multicomplex on set $X$, and $\pi = (X_1, \dots, X_n)$ is an ordered partition of $X$, such that $\deg(m_{X_i}) \leq a_i$ for all $m\in M$ and $i\in [n]$. Given a monomial $m\in M$, the vector $\DEG(m) := (\deg(m_{X_1}), \dots, \deg(m_{X_n})) \in \mathbb{N}^n$ is called the {\it multidegree} of $m$. For each ${\bf b} \in \mathbb{N}^n$, let $f_{{\bf b}}$ be the number of monomials $m$ in $M$ such that $\DEG(m) = {\bf b}$. The array of integers $\{f_{{\bf b}}\}_{{\bf b} \leq {\bf a}}$ is called the {\it fine $f$-vector} of $(M, \pi)$, and it is a refinement of the $f$-vector $(f_0, f_1, \dots)$ of $M$ in the sense that $f_i = \sum_{|{\bf b}| = i} f_{{\bf b}}$ for each $i \in \mathbb{N}$. For type ${\bf a} = {\bf 1}_n$, define the {\it flag $f$-vector} of $(M, \pi)$ analogously as in the case of simplicial complexes.

\subsection{Color Shifting}
Color shifting is a colored analogue of shifting, and it was (to the best of our knowledge) first considered by Babson and Novik~\cite{BabsonNovik2006:FaceNumbersNongenericInitialIdeals} in 2006 in the context of ${\bf a}$-colored complexes. Although usually defined only for colored complexes, color shifting can naturally be extended to colored multicomplexes.

Fix a set of variables $X$, and let $\pi = (X_1, \dots, X_n)$ be an ordered partition of $X$ such that each $X_i = \{x_{i,1}, x_{i,2}, \dots\}$ is linearly ordered by $x_{i,1} < x_{i,2} < \dots$. For each $i\in [n]$, let $\phi_i: X_i \to \overline{\mathbb{P}}$ be any map, and let ${\bf e}_i = (\phi_i(x_{i,1}), \phi_i(x_{i,2}), \dots)$. Let $\mathcal{M}_{\pi}({\bf e}_1, \dots, {\bf e}_n)$ denote the collection of all monomials $m$ in variables $X$ such that $m_{X_i} \in \mathcal{M}_{X_i}({\bf e}_i)$ for every $i\in [n]$. When ${\bf e}_i = (\infty, \infty, \dots)$ for every $i\in [n]$, we simply write $\mathcal{M}_{\pi}({\bf e}_1, \dots, {\bf e}_n)$ as $\mathcal{M}_{\pi}$. If a colored multicomplex $(M, \pi)$ satisfies $M \subseteq \mathcal{M}_{\pi}({\bf e}_1, \dots, {\bf e}_n)$, then by abuse of notation, we say $(M, \pi)$ is in $\mathcal{M}_{\pi}({\bf e}_1, \dots, {\bf e}_n)$. Note that it is implicitly assumed every colored multicomplex in $\mathcal{M}_{\pi}({\bf e}_1, \dots, {\bf e}_n)$ has $\pi$ as its corresponding ordered partition, and the length of each sequence ${\bf e}_i$ is the cardinality of $X_i$.

Let $Y = \{y_1, y_2, \dots\}$ be a set of variables linearly ordered by $y_1 < y_2 < \dots$, let $\phi: Y \to \overline{\mathbb{P}}$ be an arbitrary map, and define ${\bf e} = (e_1, e_2, \dots) := (\phi(y_1), \phi(y_2), \dots)$. A multicomplex $M^{\prime} \in \mathcal{M}_Y({\bf e})$ is called {\it shifted} (in $\mathcal{M}_Y({\bf e})$) if every $m^{\prime}\in M^{\prime}$ satisfies the property: $y_j$ divides $m^{\prime}$ and ${y_i}^{e_i}$ does not divide $m^{\prime}$ for integers $1\leq i < j \Longrightarrow m^{\prime}y_i/y_j \in M^{\prime}$. If $M^{\prime}$ is finite and ${\bf e} = (1, 1, \dots)$, then $M^{\prime}$ can be identified with a simplicial complex, and this notion of `shifted' coincides with the usual notion of a shifted complex; see, e.g., \cite{Kalai:AlgebraicShiftingNotes}. Given ${\bf a}\in \mathbb{P}^n$, an ${\bf a}$-colored multicomplex $(M, \pi)$ in $\mathcal{M}_{\pi}({\bf e}_1, \dots, {\bf e}_n)$ is called {\it color-shifted} if for every $i\in [n]$ and $m\in M_{X-X_i}$, the multicomplex $\{m^{\prime} \in M_{X_i}: mm^{\prime} \in M, \deg(m^{\prime}) = k\}$ is shifted in $\mathcal{M}_{X_i}({\bf e}_i)$ for all $k\geq 0$. A color-shifted ${\bf a}$-colored complex can be analogously defined as follows.

\begin{definition}
Let $(\Delta, \rho)$ be an ${\bf a}$-colored complex for some ${\bf a}\in \mathbb{P}^n$, and let $\rho = (V_1, \dots, V_n)$. For each $i\in [n]$, assume $V_i = \{v_{i,1}, \dots, v_{i,\lambda_i}\}$ has $\lambda_i$ elements linearly ordered by $v_{i,1} < \dots < v_{i,\lambda_i}$. Then we say $(\Delta, \rho)$ is {\it color-shifted} if every $F\in \Delta$ and $i\in [n]$ satisfy the property: $v_{i,j} \in F$ and $v_{i,j^{\prime}} \not\in F$ for some integers $1\leq j^{\prime} < j \leq \lambda_i \Rightarrow (F\backslash\{v_{i,j}\})\cup\{v_{i,j^{\prime}}\} \in \Delta$.
\end{definition}

Babson and Novik~\cite{BabsonNovik2006:FaceNumbersNongenericInitialIdeals} developed a remarkable theory of colored algebraic shifting, which is a colored analogue of symmetric algebraic shifting proposed by Kalai~\cite{Kalai1991:SymmetricAlgebraicShifting}. Recall that $X_i$ is linearly ordered for each $i \in [n]$, thus $X$ as a poset is a disjoint union of $n$ chains. Let $\prec$ be any linear extension of this partial order on $X$. The main idea is that given $\prec$, ${\bf a} \in \mathbb{P}^n$, and any ${\bf a}$-colored complex $(\Gamma, \pi \cap V)$ (where $V \subseteq X$ is the vertex set of $\Gamma$), we can construct another simplicial complex $\widetilde{\Delta}_{\prec}(\Gamma)$ with vertex set $V^{\prime} \subseteq X$, such that $(\widetilde{\Delta}_{\prec}(\Gamma), \pi \cap V^{\prime})$ is a color-shifted ${\bf a}$-colored complex with the same fine $f$-vector as $(\Gamma, \pi \cap V)$. This map $(\Gamma, \pi \cap V) \mapsto (\widetilde{\Delta}_{\prec}(\Gamma), \pi \cap V^{\prime})$ is called {\it colored algebraic shifting}, and the ${\bf a}$-colored complex $(\widetilde{\Delta}_{\prec}(\Gamma), \pi \cap V^{\prime})$ is called the {\it colored algebraic shifting} of $(\Gamma, \pi \cap V)$ (with respect to $\prec$). See \cite{BabsonNovik2006:FaceNumbersNongenericInitialIdeals} for details.

\section{Color Compressions of Colored Complexes}\label{sec:ColorCompressionColoredComplexes}
Color compression, a colored analogue of compression, was first introduced by Bj{\"o}rner, Frankl and Stanley~\cite{BjornerFranklStanley1987:CohenMacaulayComplexes} in 1987, although under the same original name of `compression'. The main goal of this section is to prove that color compression preserves the fine $f$-vectors of colored multicomplexes in $\mathcal{M}_{\pi}({\bf e}_1, \dots, {\bf e}_n)$ (see Theorem \ref{thm:ColorCompressionPreservesFineFVector}). The ideas involved are not new, and our proof follows from a slight modification of the proof of \cite[Theorem 1]{BjornerFranklStanley1987:CohenMacaulayComplexes}, which considered colored multicomplexes in $\mathcal{M}_{\pi}$. Nevertheless, Theorem \ref{thm:ColorCompressionPreservesFineFVector} is analogous to how the Clements-Lindstr\"{o}m theorem~\cite{ClementsLindstrom1969} extends both the Kruskal-Katona theorem~\cite{Katona1968,Kruskal1963,Schutzenberger1959} and the Macaulay theorem~\cite{Macaulay1927}.

Throughout this section, let $X$ be a set of variables, and let $\pi = (X_1, \dots, X_n)$ be an ordered partition of $X$ such that each $X_i = \{x_{i,1}, x_{i,2}, \dots\}$ is linearly ordered by $x_{i,1} < x_{i,2} < \dots$. For every $i\in [n]$, let $\phi_i: X_i \to \overline{\mathbb{P}}$ be an arbitrary map, and define ${\bf e}_i = (e_{i,1}, e_{i,2}, \dots) := (\phi_i(x_{i,1}), \phi_i(x_{i,2}), \dots)$. Recall that an {\it order ideal} of a poset $(P, \leq)$ is a subset $I \subseteq P$ such that if $x\in P$ and $y\leq x$, then $y\in I$. If $\leq$ is a well-ordering on $P$, then the order ideal $I$ is called an {\it initial segment} of $P$ (with respect to $\leq$). Observe that for every $i\in [n]$, $d\in \mathbb{N}$, we can order the monomials in $\mathcal{M}_{X_i}^d({\bf e}_i)$ by treating $\mathcal{M}_{X_i}^d({\bf e}_i)$ as a subposet of $\mathcal{M}_{X_i}^d$ with the induced lex order. This induced lex order is clearly a well-ordering on $\mathcal{M}_{X_i}^d({\bf e}_i)$, and we write $\mathcal{I}_{X_i}^{d; {\bf e}_i}(N)$ to denote the lex initial segment of $\mathcal{M}_{X_i}^d({\bf e}_i)$ of size $N$.

Let $(M, \pi)$ be a colored multicomplex of type ${\bf a} \in \mathbb{P}^n$ in $\mathcal{M}_{\pi}({\bf e}_1, \dots, {\bf e}_n)$, and fix some $t\in [n]$. Note that every monomial $m\in M$ can be factorized as $m = m_{X_t}m_{X - X_t}$, hence we can write $M$ as the disjoint union
\begin{equation}
M = \bigsqcup_{\widetilde{m} \in M_{X - X_t}} \Big(\bigsqcup_{d\in \mathbb{N}} \Big\{m\in M: m_{X - X_t} = \widetilde{m}, \deg(m/\widetilde{m}) = d \Big\}\Big).
\end{equation}
For each $d\in \mathbb{N}$, $\widetilde{m} \in M_{X - X_t}$, let $f_d(\widetilde{m})$ be the number of $m\in M$ such that $m_{X - X_t} = \widetilde{m}$ and $\deg(m/\widetilde{m}) = d$. Define the operation $\mathcal{C}_t$ on $M$ by
\begin{equation}\label{eqn-compression-defn}
\mathcal{C}_t(M) := \bigsqcup_{\widetilde{m} \in M_{X - X_t}} \Big(\bigsqcup_{d\in \mathbb{N}} \Big\{\widetilde{m}m^* : m^* \in \mathcal{I}_{X_t}^{d; {\bf e}_t}\big(f_d(\widetilde{m})\big) \Big\}\Big).
\end{equation}
Note that when $M_{X_t} = M$, this specializes to the usual notion of compression for multicomplexes. The Clements-Lindstr\"{o}m theorem~\cite{ClementsLindstrom1969} is equivalent to the statement that if $e_{t,1} \geq e_{t,2} \geq \dots$, and $M_{X_t} = M$, then $\mathcal{C}_t(M)$ is a multicomplex. Using this version of Clements-Lindstr\"{o}m theorem, we prove the following important lemma.

\begin{lemma}\label{lemma:compression}
If $e_{t,1} \geq e_{t,2} \geq \dots$, then $\mathcal{C}_t(M)$ is a multicomplex.
\end{lemma}

\begin{proof}
Let $p$ be an arbitrary monomial in $\mathcal{M}_{\pi}({\bf e}_1, \dots, {\bf e}_n)$, and suppose there is some $x\in X$ such that $p^{\prime} := px$ is in $\mathcal{C}_t(M)$. To show $\mathcal{C}_t(M)$ is a multicomplex, we have to show that $p \in \mathcal{C}_t(M)$.

Suppose $x\in X-X_t$. Let $q := p_{X - X_t}$, $q^{\prime} := p^{\prime}_{X - X_t}$, $d := \deg(p/q)$, and note that $q^{\prime} = qx$, $\deg(p^{\prime}/q^{\prime}) = d$. Every $m\in M$ satisfying $m_{X-X_t} = q^{\prime}$ must also satisfy $m/x \in M$ and $(m/x)_{X - X_t} = q$, hence $f_d(q^{\prime}) \leq f_d(q)$. This means $\mathcal{I}_{X_t}^{d; {\bf e}_t}(f_d(q^{\prime})) \subseteq \mathcal{I}_{X_t}^{d; {\bf e}_t}(f_d(q))$, therefore $p\in \mathcal{C}_t(M)$.

Suppose instead $x\in X_t$. Again let $q := p_{X - X_t}$, and observe that $p^{\prime}_{X - X_t} = q$. For any subcollection $M^{\prime} \subseteq \mathcal{M}_{\pi}$, define $M^{\prime}/q := \{m/q: m\in M^{\prime}, m_{X-X_t} = q\}$. We check that $M/q$ is a subcomplex of $M$ satisfying $(M/q)_{X_t} = M/q$. Since $\mathcal{C}_t(M/q) = \mathcal{C}_t(M)/q$, it follows from $p^{\prime} \in \mathcal{C}_t(M)$ and $p^{\prime}_{X-X_t} = q$ that $p^{\prime}/q \in \mathcal{C}_t(M/q)$. Now $e_{t,1} \geq e_{t,2} \geq \dots$, so the Clements-Lindstr\"{o}m theorem says $\mathcal{C}_t(M/q)$ is a multicomplex, thus $p/q \in \mathcal{C}_t(M/q)$, which implies $p \in \mathcal{C}_t(M)$.
\end{proof}

Clearly, $(\mathcal{C}_t(M), \pi)$ is an ${\bf a}$-colored multicomplex in $\mathcal{M}_{\pi}({\bf e}_1, \dots, {\bf e}_n)$ with the same fine $f$-vector as $(M, \pi)$. If $\mathcal{C}_t(M) = M$ for all $t\in [n]$, then we say $(M, \pi)$ is {\it color-compressed}. Equivalently, $(M, \pi)$ is {\it color-compressed} if for each $t\in [n]$ and $m\in M_{X-X_t}$, the set $\{m^{\prime} \in M_{X_t}: mm^{\prime}\in M, \deg(m^{\prime}) = k\}$ is a lex initial segment of $\mathcal{M}_{X_t}^k({\bf e}_t)$ for all $k\geq 0$.

\begin{remark}\label{remark:color-shifted=color-compressed-in-1n-case}
It follows from definition that if $(M, \pi)$ is color-compressed, then $(M, \pi)$ is color-shifted. In general, the converse is not true. However, if ${\bf a} = {\bf 1}_n$, then the notions `color-shifted' and `color-compressed' are equivalent, and we leave this easy exercise to the reader.
\end{remark}

\begin{theorem}\label{thm:ColorCompressionPreservesFineFVector}
Let ${\bf a} \in \mathbb{P}^n$, and let $f = \{f_{{\bf b}}\}_{{\bf b} \leq {\bf a}}$ be an array of integers. If $e_{i,1} \geq e_{i,2} \geq \dots$ for every $i\in [n]$, then the following are equivalent:
\begin{enum*}
\item $f$ is the fine $f$-vector of an ${\bf a}$-colored multicomplex in $\mathcal{M}_{\pi}({\bf e}_1, \dots, {\bf e}_n)$.\label{item1:ColorCompressionPreservesFineFVector}
\item $f$ is the fine $f$-vector of a color-compressed ${\bf a}$-colored multicomplex in $\mathcal{M}_{\pi}({\bf e}_1, \dots, {\bf e}_n)$.\label{item2:ColorCompressionPreservesFineFVector}
\end{enum*}
\end{theorem}

\begin{proof}
Let $(M, \pi)$ be an ${\bf a}$-colored multicomplex in $\mathcal{M}_{\pi}({\bf e}_1, \dots, {\bf e}_n)$ that is not color-compressed. Starting with $M_0 = M$, iteratively construct a sequence $M_0, M_1, M_2, \dots$ of multicomplexes such that $M_i = \mathcal{C}_{t_i}(M_{i-1})$ for each $i\geq 1$, where $t_i \in [n]$ is chosen so that $M_i \neq M_{i-1}$. Given $t\in [n]$ and any ${\bf a}$-colored multicomplex $(M^{\prime}, \pi)$ in $\mathcal{M}_{\pi}({\bf e}_1, \dots, {\bf e}_n)$, we get
\begin{equation}\label{eqn-score-of-multicomplexes}
\sum_{p\in M^{\prime}} \sum_{i\in [n]} |\{q \in \mathcal{M}_{X_i}({\bf e}_i): q \leq_{\ell ex} p_{X_i}\}| \geq \sum_{p\in \mathcal{C}_t(M^{\prime})} \sum_{i\in [n]} |\{q \in \mathcal{M}_{X_i}({\bf e}_i): q \leq_{\ell ex} p_{X_i}\}|,
\end{equation}
with equality holding if and only if $\mathcal{C}_t(M^{\prime}) = M^{\prime}$. Thus, the sequence $M_0, M_1, M_2, \dots$ must terminate, say with last term $M_k$, and $(M_k, \pi)$ is a color-compressed ${\bf a}$-colored multicomplex in $\mathcal{M}_{\pi}({\bf e}_1, \dots, {\bf e}_n)$ with the same fine $f$-vector as $(M, \pi)$, which proves \ref{item1:ColorCompressionPreservesFineFVector} $\Rightarrow$ \ref{item2:ColorCompressionPreservesFineFVector}. The converse \ref{item2:ColorCompressionPreservesFineFVector} $\Rightarrow$ \ref{item1:ColorCompressionPreservesFineFVector} is trivial.
\end{proof}

Let $(\Delta, \rho)$ be an ${\bf a}$-colored complex for some ${\bf a}\in \mathbb{P}^n$, where $\rho = (V_1, \dots, V_n)$ is an ordered partition of the vertex set $V$ of $\Delta$, and assume $V_i$ is linearly ordered for each $i\in [n]$. We say $(\Delta, \rho)$ is {\it color-compressed} if for all $F\in \Delta$ and $t\in [n]$, the set $\big\{ F^{\prime} \cap V_t: F^{\prime} \in \Delta, |F^{\prime} \cap V_t| = |F \cap V_t|, F^{\prime} \cap (V\backslash V_t) = F \cap (V\backslash V_t)\big\}$ is a colex initial segment of $\binom{V_t}{|F \cap V_t|}$. By setting ${\bf e}_i = (1,1, \dots)$ for each $i\in [n]$, Theorem \ref{thm:ColorCompressionPreservesFineFVector} yields the following:

\begin{corollary}\label{cor:ColorCompressionPreservesFineFVector}
Let ${\bf a} \in \mathbb{P}^n$, and let $f = \{f_{{\bf b}}\}_{{\bf b} \leq {\bf a}}$ be an array of integers. The following are equivalent:
\begin{enum*}
\item $f$ is the fine $f$-vector of an ${\bf a}$-colored complex.
\item $f$ is the fine $f$-vector of a color-compressed ${\bf a}$-colored complex.
\end{enum*}
\end{corollary}

In the proof of Theorem \ref{thm:ColorCompressionPreservesFineFVector}, we started with an ${\bf a}$-colored multicomplex $(M, \pi)$ in $\mathcal{M}_{\pi}({\bf e}_1, \dots, {\bf e}_n)$ that is not color-compressed, and we constructed a finite sequence of multicomplexes in $\mathcal{M}_{\pi}({\bf e}_1, \dots, {\bf e}_n)$:
\begin{equation}\label{sequence-of-multicomplexes}
M, \ \mathcal{C}_{t_1}(M), \ \mathcal{C}_{t_2}(\mathcal{C}_{t_1}(M)), \ \dots \ , \ \mathcal{C}_{t_k}(\mathcal{C}_{t_{k-1}}(\cdots \mathcal{C}_{t_1}(M))).
\end{equation}
All the terms in this sequence are distinct, and the last term corresponds to a color-compressed colored multicomplex. Clearly if $n=1$ (i.e. in the ``uncolored'' case), the last term is uniquely determined by $M$, known as the {\it compression} of $M$. However, if $n>1$, then in general, different choices of $t_1, \dots, t_k$ determine different color-compressed colored multicomplexes, as the following example shows.

\begin{example}
Let $X = \{x_1, x_2, x_3\}$ and $Y = \{y_1, y_2, y_3\}$ be sets of variables, each linearly ordered by $x_1 < x_2 < x_3$ and $y_1 < y_2 < y_3$ respectively, and set $\pi = (X,Y)$. Consider the colored multicomplex $(M, \pi)$ of type $(1,1)$ in $\mathcal{M}_{\pi}({\bf 1}_3, {\bf 1}_3)$, given by $M = \{x_1y_1, x_1y_2, x_1y_3, x_2y_1, x_2y_2, x_3y_1, x_3y_3\}$. Note that $(\mathcal{C}_1(M), \pi)$ and $(\mathcal{C}_2(M), \pi)$ are both color-compressed, yet $\mathcal{C}_1(M) \neq \mathcal{C}_2(M)$.
\end{example}

Let $(M,\pi)$ and $(M^{\prime}, \pi)$ be colored multicomplexes in $\mathcal{M}_{\pi}({\bf e}_1, \dots, {\bf e}_n)$, and suppose $(M^{\prime}, \pi)$ is color-compressed. If $M^{\prime} = \mathcal{C}_{t_k}(\mathcal{C}_{t_{k-1}}(\cdots \mathcal{C}_{t_1}(M)))$ for some finite sequence $t_1, \dots, t_k$ of integers in $[n]$, then $(M^{\prime}, \pi)$ is called a {\it color compression} of $(M, \pi)$. It is easy to see that the integers $t_1, \dots, t_k$ in \eqref{sequence-of-multicomplexes} are distinct, hence every colored multicomplex in $\mathcal{M}_{\pi}({\bf e}_1, \dots, {\bf e}_n)$ has at most $n!$ color compressions. We leave the reader to verify that if each $X_i$ is infinite, then in fact almost every colored multicomplex (in the probabilistic sense) in $\mathcal{M}_{\pi}({\bf e}_1, \dots, {\bf e}_n)$ has $n!$ color compressions.

For $n\neq 1$, the non-uniqueness of color compressions of colored multicomplexes and the non-uniqueness of color-compressed multicomplexes with a given fine $f$-vector~\cite{BjornerFranklStanley1987:CohenMacaulayComplexes} suggest the usual definition of the $i$-th Macaulay representation of a positive integer $N$ (which is uniquely determined given $N$ and $i$) is inadequate for the task of numerically characterizing the fine $f$-vectors of color-compressed ${\bf a}$-colored multicomplexes in $\mathcal{M}_{\pi}({\bf e}_1, \dots, {\bf e}_n)$. Later in Section \ref{sec:GeneralizedMacaulayRepresentations}, we will introduce generalized Macaulay representations to address these non-uniqueness issues.

\section{Macaulay Decomposability}\label{sec:MacaulayDecomposability}
In this section, we introduce the notion of ${\bf a}$-Macaulay decomposability for simplicial complexes and show that ${\bf a}$-Macaulay decomposable simplicial complexes are vertex-decomposable. As the main result of this section, we prove that pure color-shifted ${\bf a}$-balanced complexes are ${\bf a}$-Macaulay decomposable and hence vertex-decomposable.

Note that vertex-decomposability was originally defined for pure complexes by Provan and Billera~\cite{ProvanBillera1980:Decompositions} and subsequently generalized to non-pure complexes by Bj\"{o}rner and Wachs~\cite[Section 11]{BjornerWachs1997:NonpureShellableII} as follows. 
\begin{definition}
A simplicial complex $\Delta$ on a vertex set $V$ is {\it vertex-decomposable} if
\begin{enum*}
\item $\Delta$ is a simplex or $\Delta = \{\emptyset\}$; or
\item there exists a vertex $x \in V$, called a {\it shedding vertex} of $\Delta$, such that
\begin{enum*}
\item $\dl_{\Delta}(x)$ and $\lk_{\Delta}(x)$ are vertex-decomposable, and
\item no facet of $\lk_{\Delta}(x)$ is a facet of $\dl_{\Delta}(x)$.
\end{enum*}
\end{enum*}
\end{definition}

In this paper, we always use this generalized definition for vertex-decomposability. If $\Delta$ is pure, then this definition specializes to the original definition given by Provan and Billera~\cite{ProvanBillera1980:Decompositions}. 

\begin{definition}
Let ${\bf a} = (a_1, \dots, a_n) \in \mathbb{N}^n$, and let $\Delta$ be a simplicial complex on a vertex set $V$. A subcomplex $\Delta^{\prime}$ of $\Delta$ is called an {\it ${\bf a}$-\bone} of $\Delta$ if there exists an ordered partition $(V_1, \dots, V_n)$ of $V$ such that
\[\Delta^{\prime} = \Big\langle\Big\{ F_1 \cup \dots \cup F_n \in \Delta: F_i \in \binom{V_i}{a_i} \text{ for each }i\in [n]\Big\}\Big\rangle.\]
\end{definition}

For example, if $p,q\in \mathbb{P}$, then the complete bipartite graph $K_{p,q}$ is a $(1,1)$-{\bone} of a simplex with $p+q$ vertices. If $\Delta$ is a simplex, then an ${\bf a}$-{\bone} of $\Delta$ is the subcomplex $\langle \binom{V_1}{a_1} \rangle * \dots * \langle \binom{V_n}{a_n} \rangle$ for some ordered partition $(V_1, \dots, V_n)$. By definition, an ${\bf a}$-{\bone} of any simplicial complex $\Delta$ is always pure of dimension $|{\bf a}|-1$, and if $\Delta$ is pure, then for a fixed ${\bf a} \in \mathbb{P}^n$, the union of all ${\bf a}$-{\bone}s of $\Delta$ (over all ordered partitions of $V$) forms the $(|{\bf a}| - 1)$-skeleton of $\Delta$. For the case $n=1$, the ${\bf a}$-{\bone} of every $\Delta$ is unique.

\begin{lemma}\label{lemma:a-bone-vertex-decomposable}
Any ${\bf a}$-{\bone} of a simplex is vertex-decomposable.
\end{lemma}

\begin{proof}
Let $\Delta$ be a simplex on a vertex set $V$, and let $\Gamma$ be the ${\bf a}$-{\bone} of $\Delta$ corresponding to the ordered partition $(V_1, \dots, V_n)$ of $V$. Given simplicial complexes $\Sigma_1, \Sigma_2$ with disjoint vertex sets, \cite[Proposition 2.4]{ProvanBillera1980:Decompositions} says $\Sigma_1 * \Sigma_2$ is vertex-decomposable if and only if $\Sigma_1$ and $\Sigma_2$ are both vertex-decomposable.
Since $\Gamma = \langle \binom{V_1}{a_1} \rangle * \dots * \langle \binom{V_n}{a_n} \rangle$, it then suffices to show that the $k$-skeleton (any $k\in \mathbb{N}$) of a simplex is vertex-decomposable, which is straightforward.
\end{proof}

For ${\bf x}, {\bf y} \in \mathbb{Z}^n$, recall that ${\bf x} < {\bf y}$ if ${\bf x} \leq {\bf y}$ and ${\bf x} \neq {\bf y}$. Also, recall that ${\bf x} \lessdot {\bf y}$ if ${\bf x} < {\bf y}$ and $|{\bf y}| = |{\bf x}| + 1$.

\begin{definition}\label{defn:MacDecomp}
Let ${\bf a} \in \mathbb{N}^n$. A simplicial complex $\Delta$ on a vertex set $V$ is called {\it ${\bf a}$-Macaulay decomposable} if
\begin{enum*}
\item $\Delta$ is an ${\bf a}$-{\bone} of a simplex; or
\item\label{defn:MacDecompCondii} there exists a vertex $x \in V$ and some ${\bf a}^{\prime} \in \mathbb{N}^n$, ${\bf a}^{\prime} < {\bf a}$, such that\label{shedding-Mac}
\begin{enum*}
\item $\dl_{\Delta}(x)$ is ${\bf a}$-Macaulay decomposable,
\item $\lk_{\Delta}(x)$ is ${\bf a}^{\prime}$-Macaulay decomposable, and
\item\label{defn:MacDecompCondiic} no facet of $\lk_{\Delta}(x)$ is a facet of $\dl_{\Delta}(x)$.
\end{enum*}
\end{enum*}
Such a vertex $x$ in \ref{shedding-Mac} is called a {\it Macaulay shedding vertex} of $\Delta$. As the following proposition shows, a Macaulay shedding vertex is a shedding vertex of a vertex-decomposable simplicial complex.
\end{definition}

\begin{proposition}\label{prop:Macaulay-decomposable=>vertex-decomposable}
Let ${\bf a} \in \mathbb{N}^n$, and let $\Delta$ be a simplicial complex. If $\Delta$ is ${\bf a}$-Macaulay decomposable, then $\Delta$ is vertex-decomposable, and a Macaulay shedding vertex of $\Delta$ is a shedding vertex of $\Delta$.
\end{proposition}

\begin{proof} We prove by double induction on $|{\bf a}|$ and $|V|$, where $V$ denotes the vertex set of $\Delta$. The cases $|{\bf a}| = 0$ and $|V| \leq 2$ are trivial, and in view of Lemma \ref{lemma:a-bone-vertex-decomposable}, we assume $\Delta$ is not an ${\bf a}$-{\bone} of any simplex. Let $x$ be a Macaulay shedding vertex of $\Delta$. Note that $\dl_{\Delta}(x)$ is an ${\bf a}$-Macaulay decomposable simplicial complex on a vertex set contained in $V\backslash \{x\}$, while $\lk_{\Delta}(x)$ is ${\bf a}^{\prime}$-Macaulay decomposable for some ${\bf a}^{\prime} < {\bf a}$, hence $\dl_{\Delta}(x)$ and $\lk_{\Delta}(x)$ are both vertex-decomposable by induction hypothesis. Also, no facet of $\lk_{\Delta}(x)$ is a facet of $\dl_{\Delta}(x)$ by definition, hence $\Delta$ is vertex-decomposable, and we easily conclude that $x$ is also a shedding vertex.
\end{proof}

\begin{corollary}\label{cor:pure-vertex-decomposable-iff-d-Macaulay-decomp}
Let $\Delta$ be a $(d-1)$-dimensional simplicial complex. Then $\Delta$ is vertex-decomposable if and only if $\Delta$ is $(d)$-Macaulay decomposable.
\end{corollary}

\begin{proof}
Let $\Delta$ be vertex-decomposable with vertex set $V$. In view of Proposition \ref{prop:Macaulay-decomposable=>vertex-decomposable}, we only need to show that $\Delta$ is $(d)$-Macaulay decomposable, which easily follows by induction on $|V|$.
\end{proof}

A simplicial complex is {\it shellable} if its facets can be arranged in a linear order $F_1, \dots, F_t$ so that the subcomplex $\langle \{F_{k+1}\} \rangle \cap \langle \{F_1, \dots, F_k\} \rangle$ is pure of dimension $\dim F_{k+1} - 1$ for all $k \in [t-1]$. Similar to vertex-decomposability, the notion of shellability was originally defined for pure complexes. Here, we use the generalized notion of shellability introduced by Bj\"{o}rner and Wachs~\cite{BjornerWachs1996:NonpureShellableI}, whereby shellable simplicial complexes are not necessarily pure. In this generalized setting, vertex-decomposable simplicial complexes are shellable (\cite{ProvanBillera1980:Decompositions}~\cite[Theorem 11.3]{BjornerWachs1997:NonpureShellableII}), thus Proposition \ref{prop:Macaulay-decomposable=>vertex-decomposable} yields the following implications.

\begin{corollary}\label{cor:MacDecomp=>vertx-decomp=>shellable}
Let ${\bf a} \in \mathbb{N}^n$. Then ${\bf a}$-Macaulay decomposable $\Longrightarrow$ vertex-decomposable $\Longrightarrow$ shellable.
\end{corollary}

Note that the second implication is strict, even for pure complexes~\cite[Remark 3.4.5]{ProvanBillera1980:Decompositions}. Corollary \ref{cor:pure-vertex-decomposable-iff-d-Macaulay-decomp} gives a partial converse to the first implication when $n=1$. However for $n>1$, the first implication is strict; see Proposition \ref{prop:MacDecompSeriesOfImplications}.

\begin{proposition}\label{prop:join-Mac-decomp}
Let $n,m\in \mathbb{P}$, let ${\bf a} \in \mathbb{N}^n$, ${\bf b} \in \mathbb{N}^m$, and let $\Delta$, $\Delta^{\prime}$ be simplicial complexes with disjoint vertex sets $V, V^{\prime}$ respectively. If $\Delta$ is ${\bf a}$-Macaulay decomposable and $\Delta^{\prime}$ is ${\bf b}$-Macaulay decomposable, then $\Delta * \Delta^{\prime}$ is $({\bf a}, {\bf b})$-Macaulay decomposable.
\end{proposition}

\begin{proof}
We prove by double induction on $|{\bf a}| + |{\bf b}|$ and $|V| + |V^{\prime}|$, where the cases $|{\bf a}| + |{\bf b}| \leq 2$ or $|V| + |V^{\prime}| \leq 2$ are trivial. Suppose $\Delta$ is not an ${\bf a}$-{\bone} of a simplex, and let $x$ be a Macaulay shedding vertex of $\Delta$. Since $\dl_{\Delta}(x)$ is ${\bf a}$-Macaulay decomposable with its vertex set contained in $V\backslash \{x\}$, while $\lk_{\Delta}(x)$ is ${\bf a}^{\prime}$-Macaulay decomposable for some ${\bf a}^{\prime} < {\bf a}$, the induction hypothesis yields $\dl_{\Delta}(x) * \Delta^{\prime}$ is $({\bf a}, {\bf b})$-Macaulay decomposable and $\lk_{\Delta}(x) * \Delta^{\prime}$ is $({\bf a}^{\prime}, {\bf b})$-Macaulay decomposable. It is easy to show that $\dl_{\Delta}(x) * \Delta^{\prime} = \dl_{\Delta * \Delta^{\prime}}(x)$ and $\lk_{\Delta}(x) * \Delta^{\prime} = \lk_{\Delta * \Delta^{\prime}}(x)$. Also, if a facet $F$ of $\lk_{\Delta * \Delta^{\prime}}(x)$ is a facet of $\dl_{\Delta * \Delta^{\prime}}(x)$, then $F \cap V$ is a facet of both $\dl_{\Delta}(x)$ and $\lk_{\Delta}(x)$, which is a contradiction. Consequently, $\Delta * \Delta^{\prime}$ is $({\bf a}, {\bf b})$-Macaulay decomposable by definition in this case. The case when $\Delta^{\prime}$ is not a ${\bf b}$-{\bone} of a simplex is similar.

Finally, if $\Delta$ is an ${\bf a}$-{\bone} of a simplex $\Sigma$ with vertex set $X$, and $\Delta^{\prime}$ is a ${\bf b}$-{\bone} of a simplex $\Sigma^{\prime}$ with vertex set $X^{\prime}$, then $\Delta * \Delta^{\prime} = \langle\binom{X_1}{a_1} \rangle * \dots * \langle\binom{X_n}{a_n} \rangle * \langle\binom{X_1^{\prime}}{b_1} \rangle * \dots * \langle\binom{X_m^{\prime}}{b_m} \rangle$, where ${\bf a} = (a_1, \dots, a_n)$, ${\bf b} = (b_1, \dots, b_m)$, and $(X_1, \dots, X_n)$, $(X_1^{\prime}, \dots, X_m^{\prime})$ are the corresponding ordered partitions of $X$, $X^{\prime}$ respectively. Thus, $\Delta * \Delta^{\prime}$ is an $({\bf a}, {\bf b})$-{\bone} of a simplex with vertex set $X \cup X^{\prime}$ and hence $({\bf a}, {\bf b})$-Macaulay decomposable.
\end{proof}

Let ${\bf a} = (a_1, \dots, a_n) \in \mathbb{N}^n$ and ${\bf b} = (b_1, \dots, b_m) \in \mathbb{N}^m$, where $m\leq n$ are positive integers. If there are integers $0 = i_0 < i_1 < \dots < i_{m-1} < i_m = n$ such that $b_t = a_{i_{t-1}+1} + \dots + a_{i_t}$ for all $t\in [m]$, then we say ${\bf a}$ is an {\it ordered refinement} of ${\bf b}$. In particular, ${\bf a}$ is an ordered refinement of ${\bf b}$ implies $|{\bf a}| = |{\bf b}|$. If there exists a permutation $\phi$ on $[n]$ such that $(a_{\phi(1)}, \dots, a_{\phi(n)})$ is an ordered refinement of ${\bf b}$, then we say ${\bf a}$ is a {\it permuted refinement} of ${\bf b}$, and we denote this by ${\bf a} \leq_{pr} {\bf b}$. Let $\Sigma$ be the simplex with vertex set $V = \{(i,j) \in \mathbb{P}^2: j\in [n], i\in [a_j+1]\}$. For each $t\in [n]$, define $V_t := \{(i,t) \in \mathbb{P}^2: i\in [a_t+1]\} \subseteq V$. Then the unique ${\bf a}$-{\bone} of $\Sigma$ corresponding to the ordered partition $(V_1, \dots, V_n)$ of $V$ is called the {\it ${\bf a}$-ribcage}. For example, the $(1,1)$-ribcage is isomorphic to the complete bipartite graph $K_{2,2}$, while the $(2)$-ribcage is isomorphic to the empty triangle $\big\langle \big\{\{1,2\}, \{2,3\}, \{3,1\} \big\}\big\rangle$.

\begin{lemma}\label{lemma:ribcage}
Let ${\bf b} \in \mathbb{P}^m$ and let $\Delta$ be the ${\bf b}$-ribcage. If ${\bf a} \in \mathbb{P}^n$ is an ordered refinement of ${\bf b}$ such that $\Delta$ is ${\bf a}$-Macaulay decomposable, then ${\bf a} = {\bf b}$.
\end{lemma}

\begin{proof}
Write ${\bf a} = (a_1, \dots, a_n)$ and suppose $\Delta$ is ${\bf a}$-Macaulay decomposable corresponding to the ordered partition $(V_1^{\prime}, \dots, V_n^{\prime})$ of its vertex set $V$. Note that $|V| = |b| + m = |a| + m$, so if $n>m$, then $|V_j^{\prime}| = a_j$ for some $j\in [n]$, which implies $V_j^{\prime}$ is a non-empty set of cone-points of $\Delta$. (Recall: $v \in V$ is a {\it cone-point} of $\Delta$ if $v\in F$ for all facets $F$ of $\Delta$.) However, by the definition of the ${\bf b}$-ribcage, $\Delta$ has no cone-points, thus $n=m$, which then forces ${\bf a} = {\bf b}$.
\end{proof}

\begin{proposition}\label{prop:MacDecompSeriesOfImplications}
Let $m,n \in \mathbb{P}$, let ${\bf a} \in \mathbb{N}^n$, ${\bf b} \in \mathbb{N}^m$, and let $\Delta$ be a simplicial complex. If ${\bf a} \leq_{pr} {\bf b}$, then $\Delta$ is ${\bf a}$-Macaulay decomposable implies $\Delta$ is ${\bf b}$-Macaulay decomposable. Furthermore, if ${\bf a} \in \mathbb{P}^n$, ${\bf b} \in \mathbb{P}^m$ and $n \neq m$, then this implication is strict.
\end{proposition}

\begin{proof}
Observe that if we can show every ${\bf a}$-{\bone} of a simplex is ${\bf b}$-Macaulay decomposable for all ${\bf b} \geq_{pr} {\bf a}$, then by the definition of Macaulay decomposability, we can prove the first assertion using double induction on $|{\bf a}|$ and the number of vertices of $\Delta$. Write ${\bf a} = (a_1, \dots, a_n)$, let $\Delta^{\prime}$ be a simplex with vertex set $V$, and let $\Gamma$ be an ${\bf a}$-{\bone} of $\Delta^{\prime}$ corresponding to the ordered partition $(V_1, \dots, V_n)$ of $V$. Choose some ${\bf b} = (b_1, \dots, b_m) \geq_{pr} {\bf a}$. Without loss of generality, assume ${\bf a}$ is an ordered refinement of ${\bf b}$, and let $0 = i_0 < i_1 < \dots < i_{m-1} < i_m = n$ such that $b_t = a_{i_{t-1}+1} + \dots + a_{i_t}$ for all $t\in [m]$. For each $t\in [m]$, define $\Gamma_t := \big\langle \bigl(\begin{smallmatrix} V_{i_{t-1}+1}\\ a_{i_{t-1}+1} \end{smallmatrix} \bigr) \big\rangle * \dots * \big\langle
\bigl(\begin{smallmatrix} V_{i_t}\\ a_{i_t} \end{smallmatrix} \bigr) \big\rangle$, and note that $\Gamma_t$ is an $(a_{i_{t-1}+1}, \dots, a_{i_t})$-{\bone} of the simplex with vertex set $V_{i_{t-1}+1} \cup \dots \cup V_{i_t}$, so $\Gamma_t$ is vertex-decomposable by Lemma \ref{lemma:a-bone-vertex-decomposable} and hence $(b_t)$-Macaulay decomposable by Corollary \ref{cor:pure-vertex-decomposable-iff-d-Macaulay-decomp}. Now since $\Gamma = \Gamma_1 * \dots * \Gamma_m$, Proposition \ref{prop:join-Mac-decomp} yields $\Gamma$ is ${\bf b}$-Macaulay decomposable, thus proving the first assertion. Finally, the second assertion follows immediately from Lemma \ref{lemma:ribcage}.
\end{proof}

\begin{proposition}\label{prop:dimMacDecomp}
Let ${\bf a} \in \mathbb{N}^n$. If $\Delta$ is an ${\bf a}$-Macaulay decomposable simplicial complex, then $\dim \Delta = |{\bf a}| - 1$.
\end{proposition}

\begin{proof}
The assertion follows by double induction on $|{\bf a}|$ and $|V|$, where $V$ denotes the vertex set of $\Delta$.
\end{proof}

An ${\bf a}$-Macaulay decomposable simplicial complex is not necessarily pure. For example, the non-pure simplicial complex $\big\langle\big\{ \{1,2,3\}, \{1,2,4\}, \{1,5\}, \{2,5\} \big\}\big\rangle$ is $(2,1)$-Macaulay decomposable with $5$ as a Macaulay shedding vertex. If the simplicial complex $\Delta$ in Definition \ref{defn:MacDecomp} is pure, then condition \ref{defn:MacDecompCondiic} of the definition can be omitted. Furthermore, if condition \ref{defn:MacDecompCondii} holds, then ${\bf a}^{\prime} \lessdot {\bf a}$, and both $\dl_{\Delta}(x)$ and $\lk_{\Delta}(x)$ are pure. Conversely, if there exists some vertex $x^{\prime}$ of a simplicial complex $\Delta^{\prime}$ such that $\dl_{\Delta^{\prime}}(x^{\prime})$ is pure and ${\bf a}$-Macaulay decomposable; and $\lk_{\Delta^{\prime}}(x^{\prime})$ is pure and ${\bf a}^{\prime}$-Macaulay decomposable for some ${\bf a}^{\prime} \in \mathbb{N}^n$, ${\bf a}^{\prime} \lessdot {\bf a}$, then $\Delta^{\prime}$ is pure and ${\bf a}$-Macaulay decomposable.

\begin{theorem}\label{thm:color-shifted=>Macaulay-decomposable}
Let ${\bf a} \in \mathbb{P}^n$ and let $(\Delta, \pi)$ be a pure ${\bf a}$-balanced complex. If $(\Delta, \pi)$ is color-shifted, then $\Delta$ is ${\bf a}$-Macaulay decomposable.
\end{theorem}

\begin{proof}
Write ${\bf a} = (a_1, \dots, a_n)$, $|{\bf a}| = k$, and let $\pi = (V_1, \dots, V_n)$ be the ordered partition of the vertex set $V$ of $\Delta$. For each $i\in [n]$, assume $V_i = \{v_{i,1}, \dots, v_{i,\lambda_i}\}$ has $\lambda_i$ elements linearly ordered by $v_{i,1} < \dots < v_{i,\lambda_i}$. Without loss of generality, assume $\Delta$ is not an ${\bf a}$-{\bone} of a simplex, which implies $\lambda_t > a_t$ for some $t\in [n]$. Define ${\bf a}^{\prime} := {\bf a} - \boldsymbol{\delta}_{t,n}$, and let $W$, $W^{\prime}$ be the vertex sets of $\dl_{\Delta}(v_{t,\lambda_t})$, $\lk_{\Delta}(v_{t,\lambda_t})$ respectively. Since $(\Delta, \pi)$ is color-shifted, it follows that $(\dl_{\Delta}(v_{t,\lambda_t}), \pi \cap W)$ is a color-shifted ${\bf a}$-colored complex, while $(\lk_{\Delta}(v_{t,\lambda_t}), \overline{\pi \cap W^{\prime}})$ is a color-shifted $\overline{{\bf a}^{\prime}}$-colored complex. We now prove this theorem by double induction on $k$ and $|V|$.

Suppose there is a facet $F$ of $\dl_{\Delta}(v_{t,\lambda_t})$ such that $\dim F < k-1$. Since $\Delta$ is pure of dimension $k-1$, we have $F \cup \{v_{t,\lambda_t}\} \in \Delta$, so $(\Delta, \pi)$ is color-shifted implies $F \cup \{v_{t,j}\} \in \Delta$ for every $j\in [\lambda_t-1]$ satisfying $v_{t,j} \not\in F$. Such a $j$ exists since $\lambda_t > a_t$, and in particular, $F \cup \{v_{t,j}\} \in \dl_{\Delta}(v_{t,\lambda_t})$, which contradicts the assumption that $F$ is a facet of $\dl_{\Delta}(v_{t,\lambda_t})$. This means $(\dl_{\Delta}(v_{t,\lambda_t}), \pi \cap W)$ is a pure color-shifted ${\bf a}$-balanced complex, hence $\dl_{\Delta}(v_{t,\lambda_t})$ is ${\bf a}$-Macaulay decomposable by induction hypothesis. A similar argument shows $\lk_{\Delta}(v_{t,\lambda_t})$ is pure of dimension $k-2$ and hence ${\bf a}^{\prime}$-Macaulay decomposable by induction hypothesis and Proposition \ref{prop:MacDecompSeriesOfImplications}, therefore $\Delta$ is ${\bf a}$-Macaulay decomposable.
\end{proof}

\begin{remark}
Both the hypotheses in Theorem \ref{thm:color-shifted=>Macaulay-decomposable}, i.e. that $(\Delta, \pi)$ is pure and ${\bf a}$-balanced (instead of just ${\bf a}$-colored), are necessary. Murai~\cite{Murai2008:BettiNumbersSquarefreeStronglyColorStable} gave an example of a non-pure ${\bf 1}_5$-balanced color-shifted complex that is not shellable, so since all ${\bf 1}_5$-Macaulay decomposable simplicial complexes are shellable (Corollary \ref{cor:MacDecomp=>vertx-decomp=>shellable}), this same example fails to be ${\bf 1}_5$-Macaulay decomposable. Also, the pure simplicial complex $\big\langle\big\{ \{1,2\}, \{3,4\} \big\}\big\rangle$, together with the ordered partition $(\{1\}, \{2\}, \{3\}, \{4\})$ of its vertex set $[4]$, is clearly color-shifted and ${\bf 1}_4$-colored, although it is not ${\bf 1}_4$-balanced. However, it fails to be ${\bf 1}_4$-Macaulay decomposable. In general, Proposition \ref{prop:dimMacDecomp} says an ${\bf a}$-colored complex that is not ${\bf a}$-balanced can never be ${\bf a}$-Macaulay decomposable.
\end{remark}

Theorem \ref{thm:color-shifted=>Macaulay-decomposable} generalizes Murai's result~\cite[Proposition 4.2]{Murai2008:BettiNumbersSquarefreeStronglyColorStable}, which states that a pure ${\bf a}$-balanced color-shifted complex is shellable. From Theorem \ref{thm:color-shifted=>Macaulay-decomposable} and its proof, we get the following useful corollaries.
\begin{corollary}\label{cor:Macaulay-shedding-vertices}
Let ${\bf a} = (a_1, \dots, a_n) \in \mathbb{P}^n$, let $(\Delta, \pi)$ be a pure color-shifted ${\bf a}$-balanced complex, and let $\pi = (V_1, \dots, V_n)$ be the corresponding ordered partition of its vertex set. If $\Delta$ is not an ${\bf a}$-{\bone} of a simplex, then the maximal element of $V_i$ is a Macaulay shedding vertex of $\Delta$ for every $i\in [n]$ satisfying $|V_i| > a_i$.
\end{corollary}

\begin{proof}
See the proof of Theorem \ref{thm:color-shifted=>Macaulay-decomposable}.
\end{proof}

\begin{corollary}\label{cor:color-shifted-balanced-equiv-to-pure}
Let ${\bf a}\in \mathbb{P}^n$, and let $(\Delta, \pi)$ be a color-shifted ${\bf a}$-balanced complex. Then the following are equivalent:
\begin{enum*}
\item $\Delta$ is pure. \label{cor-color-shifted-pure-equivalences:pure}
\item $\Delta$ is pure ${\bf a}$-Macaulay decomposable. \label{cor-color-shifted-pure-equivalences:pure-a-MacDecomp}
\item $\Delta$ is pure vertex-decomposable. \label{cor-color-shifted-pure-equivalences:pure-vertex-decomp}
\item $\Delta$ is pure shellable. \label{cor-color-shifted-pure-equivalences:pure-shellable}
\item $\Delta$ is Cohen-Macaulay (over any field). \label{cor-color-shifted-pure-equivalences:CM}
\end{enum*}
\end{corollary}

\begin{proof} Theorem \ref{thm:color-shifted=>Macaulay-decomposable} yields \ref{cor-color-shifted-pure-equivalences:pure} $\Rightarrow$ \ref{cor-color-shifted-pure-equivalences:pure-a-MacDecomp}, Proposition \ref{prop:Macaulay-decomposable=>vertex-decomposable} yields \ref{cor-color-shifted-pure-equivalences:pure-a-MacDecomp} $\Rightarrow$ \ref{cor-color-shifted-pure-equivalences:pure-vertex-decomp}, while the implications \ref{cor-color-shifted-pure-equivalences:pure-vertex-decomp} $\Rightarrow$ \ref{cor-color-shifted-pure-equivalences:pure-shellable} $\Rightarrow$ \ref{cor-color-shifted-pure-equivalences:CM} $\Rightarrow$ \ref{cor-color-shifted-pure-equivalences:pure} are well-known (see \cite{book:StanleyGreenBook}).
\end{proof}

\begin{corollary}\label{cor:BalancedCMListOfEquiv}
Let ${\bf a} \in \mathbb{P}^n$, and let $f = \{f_{{\bf b}}\}_{{\bf b} \leq {\bf a}}$ be an array of integers. Then the following are equivalent:
\begin{enum*}
\item $f$ is the fine $f$-vector of an ${\bf a}$-balanced CM complex.\label{cond-equiv-CM}
\item $f$ is the fine $f$-vector of an ${\bf a}$-balanced pure vertex-decomposable complex.\label{cond-equiv-vertex-decomp}
\item $f$ is the fine $f$-vector of an ${\bf a}$-balanced pure color-shifted complex.\label{cond-equiv-pure-CS}
\item $f$ is the fine $f$-vector of an ${\bf a}$-balanced color-shifted CM complex.\label{cond-equiv-CS-CM}
\end{enum*}
\end{corollary}

\begin{proof}
Let $(\Gamma, \pi)$ be an ${\bf a}$-balanced CM complex, and assume each $V_i$ in the ordered partition $\pi = (V_1, \dots, V_n)$ is linearly ordered, so $V := \bigcup_{i\in [n]} V_i$ as a poset is a disjoint union of $n$ chains. By \cite[Theorem 6.5]{BabsonNovik2006:FaceNumbersNongenericInitialIdeals}, there is a linear extension $\prec$ of this partial order on $V$ such that the colored algebraic shifting $(\widetilde{\Delta}_{\prec}(\Gamma), \pi^{\prime})$ is a color-shifted ${\bf a}$-balanced CM complex with the same fine $f$-vector as $(\Gamma, \pi)$, hence \ref{cond-equiv-CM} $\Rightarrow$ \ref{cond-equiv-CS-CM}. The implications \ref{cond-equiv-CS-CM} $\Rightarrow$ \ref{cond-equiv-pure-CS} $\Rightarrow$ \ref{cond-equiv-vertex-decomp} follow from Corollary \ref{cor:color-shifted-balanced-equiv-to-pure}, Theorem \ref{thm:color-shifted=>Macaulay-decomposable} and Proposition \ref{prop:Macaulay-decomposable=>vertex-decomposable}, while \ref{cond-equiv-vertex-decomp} $\Rightarrow$ \ref{cond-equiv-CM} is trivial.
\end{proof}

\begin{remark}
Corollary \ref{cor:color-shifted-balanced-equiv-to-pure} generalizes a result by Babson and Novik~\cite[Theorem 6.2]{BabsonNovik2006:FaceNumbersNongenericInitialIdeals}, which states that given $(\Delta, \pi)$ a color-shifted ${\bf a}$-balanced complex, $\Delta$ is Cohen-Macaulay (over any field) if and only if $\Delta$ is pure. Also, Corollary \ref{cor:BalancedCMListOfEquiv} generalizes a result by Biermann and Van Tuyl~\cite{BiermannVanTuyl2013:BalancedVDSimplicialComplexesAndHVectors}, which deals with $f$-vectors instead of the more refined fine $f$-vectors, and is equivalent to the statement that the $f$-vector of a completely balanced CM complex is the $f$-vector of a completely balanced pure vertex-decomposable complex. The proof by Biermann and Van Tuyl uses a generalization of the ``whiskering'' construction of graphs introduced by Villarreal~\cite{Villarreal1990:CMGraphs}, and Cook and Nagel~\cite{CookNagel2012:CMGraphsFlagComplexes}.
\end{remark}

\section{Generalized Macaulay Representations}\label{sec:GeneralizedMacaulayRepresentations}
In this section, we introduce the notion of generalized Macaulay representations and use them to obtain numerical characterizations of the fine $f$-vectors of ${\bf a}$-colored complexes and the flag $f$-vectors of completely balanced CM complexes. Because graph theory terminology varies widely in the literature, we first describe in Section \ref{subsec:GraphTheoryTermi} the particular terminology that we use, which will be pertinent when we define generalized Macaulay representations. Sections \ref{subsec:ConstructionOfSheddingTrees}--\ref{subsec:GenMacReps} contain a unified proof of our two characterizations, divided into four main steps.

In Section \ref{subsec:ConstructionOfSheddingTrees}, we start with a pure color-shifted ${\bf a}$-balanced complex $(\Delta, \pi)$ and construct a labeled trivalent planted binary tree $(T, \phi)$, which we call the shedding tree of $(\Delta, \pi)$. Next, in Section \ref{subsec:MacaulayTrees}, we define a purely numerical notion of an ${\bf a}$-Macaulay tree of $N$ for any $n,N\in \mathbb{P}$, ${\bf a} \in \mathbb{N}^n\backslash\{{\bf 0}_n\}$. If ${\bf a}\in \mathbb{P}^n$, then every pure color-shifted ${\bf a}$-balanced complex with $N$ facets corresponds to an ${\bf a}$-Macaulay tree of $N$, while conversely, every ${\bf a}$-Macaulay tree of $N$ corresponds to a pure ${\bf a}$-balanced complex, although not necessarily color-shifted.

In Section \ref{subsec:CompressedLikeCompatibleMacTrees}, we define what it means for an ${\bf a}$-Macaulay tree to be compressed-like and compatible, and characterize pure color-compressed ${\bf a}$-balanced complexes in terms of certain compressed-like compatible ${\bf a}$-Macaulay trees. Finally, in Section \ref{subsec:GenMacReps}, we define the notion of a generalized ${\bf a}$-Macaulay representation of $N$ for any $n\in \mathbb{P}$, $N\in \mathbb{N}$, ${\bf a} \in \mathbb{N}^n\backslash\{{\bf 0}_n\}$, and we get our two desired numerical characterizations.

Throughout Sections \ref{subsec:ConstructionOfSheddingTrees}--\ref{subsec:GenMacReps} unless otherwise stated, let ${\bf a} = (a_1, \dots, a_n)$, let $V$ be the vertex set of $\Delta$, and let $\pi = (V_1, \dots, V_n)$ be the corresponding ordered partition of $V$. For each $i\in [n]$, write $\lambda_i = |V_i|$, and let $V_i = \{v_{i,1}, v_{i,2}, \dots, v_{i,\lambda_i}\}$ be linearly ordered by $v_{i,1} < v_{i,2} < \dots < v_{i,\lambda_i}$. Set $\boldsymbol{\lambda} = (\lambda_1, \dots, \lambda_n)$. For every ${\bf b} = (b_1, \dots, b_n) \in \mathbb{Z}^n$ and every subcomplex $\Delta^{\prime}$ of $\Delta$ with vertex set $V^{\prime}$, let $\mathcal{F}_{\bf b}(\Delta^{\prime})$ be the set of all faces $F\in \Delta^{\prime}$ such that $|F \cap V_i| = b_i$ for all $i\in [n]$, and define $f_{\bf b}(\Delta^{\prime}) = |\mathcal{F}_{\bf b}(\Delta^{\prime})|$. By default, set $f_{{\bf b}}(\Delta) = 0$ if ${\bf b} \not\in \mathbb{N}^n$. Observe that $\{f_{\bf b}(\Delta^{\prime})\}_{{\bf b} \leq {\bf a}}$ is the fine $f$-vector of $(\Delta^{\prime}, \pi \cap V^{\prime})$. For ${\bf x} = (x_1, \dots, x_n), {\bf y} = (y_1, \dots, y_n)$ in $\mathbb{Z}^n$, define $\binom{{\bf x}}{{\bf y}} := \binom{x_1}{y_1}\binom{x_2}{y_2}\cdots\binom{x_n}{y_n}$, and recall that ${\bf x} < {\bf y}$ if ${\bf x} \leq {\bf y}$ and ${\bf x} \neq {\bf y}$. By convention, $\binom{x_i}{y_i}$ equals $0$ if $y_i<0$. Note that ${\bf x} \lessdot {\bf y}$ if and only if ${\bf x} = {\bf y} - \boldsymbol{\delta}_{i,n}$ for some $i\in [n]$. If $\varphi$ is a function whose image is contained in $\mathbb{Z}^n$ (for some $n\in \mathbb{P}$), then $\varphi$ has $n$ component functions, and for each $t\in [n]$, let $\varphi^{[t]}$ be the $t$-th component function of $\varphi$.

\subsection{Graph Theory Terminology}\label{subsec:GraphTheoryTermi}
A {\it graph} is a pair $G = (V,E)$ such that $V$ is a finite set called the {\it vertex set}, and $E \subseteq \binom{V}{2}$. Elements of $V$ and $E$ are called {\it vertices} and {\it edges} respectively, and a vertex $v \in V$ is {\it incident} to an edge $e \in E$ if $e = \{u,v\}$ for some vertex $u\neq v$. The {\it degree} of $v\in V$, denoted by $\deg(v)$, is the number of edges incident to $v$. We say $u,v \in V$ are {\it adjacent} if $\{u,v\} \in E$. A {\it labeling} of a subset $U\subseteq V$ is a map $\phi: U \to L$, where $L$ is an arbitrary set whose elements are called {\it $\phi$-labels}, or simply {\it labels} if the context is clear. A {\it vertex labeling} of $G $ is a labeling of $V$. A {\it subgraph} of $G$ is a graph $G^{\prime} = (V^{\prime}, E^{\prime})$ such that $V^{\prime} \subseteq V$ and $E^{\prime} \subseteq E$. Given a subset $S \subseteq V$, the subgraph $G^{\prime} = (S, E \cap \binom{S}{2})$ is called the {\it subgraph of $G$ induced by $S$}. Given an edge $e = \{u,v\}$ of $G$, the graph $G\backslash e := (V, E\backslash\{e\})$ is said to be obtained from $G$ by {\it deleting} the edge $e$, and the graph constructed from $G$ by identifying the vertices $u$ and $v$ and then deleting edge $e$, which we denote by $G/e$, is said to be obtained from $G$ by {\it contracting} $e$. If $\phi$ is a vertex labeling of $G$ such that $\phi(u) = \phi(v)$, then $\phi$ induces a vertex labeling of $G/e$ in the obvious way, and we denote this vertex labeling of $G/e$ by $\phi/e$.

Given $u,v\in V$, a {\it path} from $u$ to $v$ is a sequence $v_0, v_1, \dots, v_n$ of (not necessarily distinct) vertices such that $v_0 = u$, $v_n = v$, and $\{v_0, v_1\}, \dots, \{v_{n-1}, v_n\}$ are $n$ distinct edges in $E$. The {\it length} of this walk is $n$. If $n\geq 3$ and $u=v$, then we call this path a {\it cycle}. We say $G$ is {\it connected} if there is a path from $x$ to $y$ for every $x,y\in V$. A connected graph with no cycles is called a {\it tree}, and for any vertices $x,y$ of a tree, there exists a unique path from $x$ to $y$.

Given a tree $G = (V,E)$, a vertex $v\in V$ is called a {\it leaf} if $\deg(v) = 1$. If $\deg(u) = 3$ for every non-leaf $u\in V$, then $G$ is called {\it trivalent}. We say $G$ is {\it rooted} if it has a distinguished vertex $r_0$, which we call the {\it root}. Given this root $r_0$ and any vertices $x,y\in V$ (possibly $x=y$), if $x$ is contained in the path from $r_0$ to $y$, then we say $x$ is an {\it ancestor} of $y$, and $y$ is a {\it descendant} of $x$. If in addition $\{x,y\} \in E$, then we say $x$ is a {\it parent} of $y$, and $y$ is a {\it child} of $x$. Note that in a rooted tree, the root has no parents, and every other vertex has a unique parent. A {\it planted} tree is a rooted tree whose root is a leaf, i.e. the root has a unique child. A rooted tree is called {\it binary} if every vertex has $\leq 2$ children. In this paper, we assume all rooted trees are {\it planar}, i.e. we specify an ordering of the children (if any) of every vertex in rooted trees.

Let $G = (V, E)$ be a trivalent planted binary tree with root $r_0$. Let $Y$ be the set of all trivalent vertices in $V$, and let $L$ be the set of all leaves in $V$ distinct from $r_0$, whose elements are called {\it terminal vertices}. For convenience, define the {\it $(\text{root},1,3)$-triple} of $G$ as the triple $(r_0, L, Y)$, and for any vertex $x$ of $G$, let $\mathcal{D}_G(x)$ denote the set of all descendants of $x$ in $G$. By definition, every $y\in Y$ has exactly $2$ children, which we call the {\it left child} and {\it right child}. The distinction between the left and right children of every $y\in Y$ is well-defined by the assumption of planarity, and we always attach subscripts `left' and 'right' to $y$ (i.e. $y_{\text{left}}$ and $y_{\text{right}}$) to denote the left and right children respectively of $y$. The {\it depth-first order} on $V$ is the linear order $\leq$ on $V$ such that $r_0$ is the unique minimal element, and $y < y_{\text{left}} < y_{\text{right}}$ for every $y\in Y$. An enumeration $x_1, \dots, x_k$ of a subset $X \subseteq V$ of size $k$ is said to be {\it arranged in depth-first order} if $x_1 < \dots < x_k$ (with respect to the depth-first order). For every $v\in Y \cup L$, it is easy to show there is a unique sequence of vertices $v_1, \dots, v_m$ in $V$ ($m\in \mathbb{P}$) such that $v_1 = v$, $v_m \in L$, and $v_i$ is the left child of $v_{i-1}$ whenever $i\in [m]$ and $i>1$. We call this sequence $v_1, \dots, v_m$ the {\it left relative sequence} of $v$ in $G$, and we say $v_m$ is the {\it left-most relative} of $v$ in $G$. Define {\it right relative sequence} and {\it right-most relative} analogously.

Suppose $T = (V, E)$ and $T^{\prime} = (V^{\prime}, E^{\prime})$ are two trivalent planted binary trees with $(\text{root},1,3)$-triples $(r_0, L, Y)$ and $(r_0^{\prime}, L^{\prime}, Y^{\prime})$ respectively. Let $r_1$ be the unique child of $r_0$ in $T$, and let $r_1^{\prime}$ be the unique child of $r_0^{\prime}$ in $T^{\prime}$. Then, we say $T, T^{\prime}$ are {\it isomorphic as rooted trees} if there exists a bijection $\beta: V \to V^{\prime}$ such that $\beta(r_0) = r_0^{\prime}$, $\beta(r_1) = r_1^{\prime}$, and every $y\in Y$ satisfies $\beta(y) \in Y^{\prime}$, $\beta(y_{\text{left}}) = \beta(y)_{\text{left}}$ and $\beta(y_{\text{right}}) = \beta(y)_{\text{right}}$.

\subsection{Construction of Shedding Trees}\label{subsec:ConstructionOfSheddingTrees}
Let ${\bf a} \in \mathbb{P}^n$, and let $(\Delta, \pi)$ be a pure color-shifted ${\bf a}$-balanced complex. If $\Delta$ is an ${\bf a}$-{\bone} of a simplex, then the fine $f$-vector of $(\Delta, \pi)$ is easy to compute: Since $\Delta = \langle \binom{V_1}{a_1} \rangle * \dots * \langle \binom{V_n}{a_n} \rangle$, we get $f_{{\bf b}}(\Delta) = \binom{\boldsymbol{\lambda}}{{\bf b}}$ for every ${\bf b} \in \mathbb{N}^n$ satisfying ${\bf b} \leq {\bf a}$. If $\Delta$ is not an ${\bf a}$-{\bone} of a simplex, then by Corollary \ref{cor:Macaulay-shedding-vertices}, there is some $i\in [n]$ such that $\lambda_i > a_i$, and $v_{i,\lambda_i}$ is a Macaulay shedding vertex of $\Delta$. Note that $\Delta = \dl_{\Delta}(v) \sqcup \big\{F \cup \{v\}: F \in \lk_{\Delta}(v)\big\}$ for every vertex $v$ of $\Delta$, hence
\begin{equation}
\mathcal{F}_{{\bf b}}(\Delta) = \mathcal{F}_{{\bf b}}\big(\dl_{\Delta}(v_{i,\lambda_i})\big) \sqcup \Big\{F \cup \{v_{i,\lambda_i}\}: F \in \mathcal{F}_{{\bf b} - \boldsymbol{\delta}_{i,n}} \big(\lk_{\Delta}(v_{i,\lambda_i})\big)\Big\}
\end{equation}
for all ${\bf b} \in \mathbb{N}^n$ satisfying ${\bf b} \leq {\bf a}$, which yields
\begin{equation}\label{eqn:del-link-split}
f_{{\bf b}}(\Delta) = f_{{\bf b}}\big(\dl_{\Delta}(v_{i,\lambda_i})\big) + f_{{\bf b} - \boldsymbol{\delta}_{i,n}}\big(\lk_{\Delta}(v_{i,\lambda_i})\big)
\end{equation}
for all ${\bf b} \in \mathbb{N}^n$ satisfying ${\bf b} \leq {\bf a}$. Furthermore, if we let $W$, $W^{\prime}$ be the vertex sets of $\dl_{\Delta}(v_{i,\lambda_i})$, $\lk_{\Delta}(v_{i,\lambda_i})$ respectively, then by Corollary \ref{cor:Macaulay-shedding-vertices}, $(\dl_{\Delta}(v_{i,\lambda_i}), \pi \cap W)$ is a pure color-shifted ${\bf a}$-balanced complex, while $(\lk_{\Delta}(v_{i,\lambda_i}), \overline{\pi \cap W^{\prime}})$ is a pure color-shifted $(\overline{{\bf a} - \boldsymbol{\delta}_{i,n}})$-balanced complex, so these two balanced subcomplexes are ${\bf a}$-Macaulay decomposable and $({\bf a} - \boldsymbol{\delta}_{i,n})$-Macaulay decomposable respectively by Theorem \ref{thm:color-shifted=>Macaulay-decomposable} and Proposition \ref{prop:MacDecompSeriesOfImplications}. 

The main idea of this subsection is to start with the pure color-shifted ${\bf a}$-balanced complex $(\Delta, \pi)$, and for every pure color-shifted ${\bf a}^{\prime}$-balanced subcomplex $(\Delta^{\prime}, \pi^{\prime})$ of $(\Delta, \pi)$ that is not an ${\bf a}^{\prime}$-{\bone} of a simplex (where $\pi^{\prime} = (V_1^{\prime}, \dots, V_m^{\prime})$, ${\bf a}^{\prime} = (a_1^{\prime}, \dots, a_m^{\prime}) \in \mathbb{P}^m$ and $m\leq n$), choose some $t\in [m]$ such that $|V_t^{\prime}| > a_t^{\prime}$, replace $(\Delta^{\prime}, \pi^{\prime})$ with two pure color-shifted balanced subcomplexes $(\dl_{\Delta^{\prime}}(v_{t,\max}^{\prime}), \pi^{\prime} \cap W)$ and $(\lk_{\Delta^{\prime}}(v_{t,\max}^{\prime}), \overline{\pi^{\prime} \cap W^{\prime}})$ (where $v_{t,\max}^{\prime}$ is the maximal element of $V_t^{\prime}$, and $W, W^{\prime}$ are the vertex sets of $\dl_{\Delta^{\prime}}(v_{t,\max}^{\prime})$, $\lk_{\Delta^{\prime}}(v_{t,\max}^{\prime})$ respectively), and repeat the process until every remaining pure color-shifted balanced subcomplex, say of type ${\bf a}^{\prime\prime}$, is an ${\bf a}^{\prime\prime}$-{\bone} of a simplex. By repeatedly using \eqref{eqn:del-link-split} on these subcomplexes of $\Delta$, we can then compute $f_{\bf b}(\Delta)$ for each ${\bf b} \in \mathbb{N}^n$ satisfying ${\bf b} \leq {\bf a}$. We now make this idea rigorous.

\begin{definition}
Let $S = (S_1, \dots, S_n)$ be an $n$-tuple of sets, let ${\bf x} = (x_1, \dots, x_n) \in \mathbb{N}^n$, and let $t\in [n]$. A simplicial complex $\Sigma$ is called {\it $t$-factorizable} with respect to $(S, {\bf x})$ if $\Sigma = \big\langle \binom{S_t}{x_t} \big\rangle * \Sigma^{\prime}$ for some (non-empty) simplicial complex $\Sigma^{\prime}$. If ${\bf x} \in \mathbb{P}^n$, then let $\kappa(S, {\bf x})$ be the $n$-tuple whose $i$-th entry is $|S_i|$ if $S_i$ is non-empty, and $x_i$ otherwise.
\end{definition}

Let $T_0$ be the planted tree with $2$ vertices, where $r_0$ is the root, and $r_1$ is the unique child of $r_0$. Define the vertex labeling $\phi_0$ of $T_0$ by $\phi_0(r_0) = \phi_0(r_1) = ({\bf a}, \Delta, \pi, \boldsymbol{\lambda})$, i.e. the $\phi_0$-labels are $4$-tuples, where the first entry is in $\mathbb{N}^n$, the second entry is a simplicial complex, the third entry is an $n$-tuple of subsets of $V$, and the last entry is in $\mathbb{P}^n$. Starting with the pair $(T_0, \phi_0)$, we shall construct a sequence `seq' of pairs $(T_0, \phi_0), (T_1, \phi_1), (T_2, \phi_2), \dots$ via the following algorithm:
\begin{center}
\begin{algorithmic}
\small
\State $\text{seq} \gets (T_0, \phi_0)$.
\State $t \gets n$.
\While {$t>0$}
    \State $(T^{\prime}, \phi^{\prime}) \gets \text{last pair in seq}$.
    \State $\mathcal{L} \gets \text{set of terminal vertices of }T^{\prime}$.
    \If {$\exists v \in \mathcal{L}$ with $\phi^{\prime}(v) = ({\bf a}^{\prime}, \Delta^{\prime}, \pi^{\prime}, \boldsymbol{\lambda}^{\prime})$ such that $\Delta^{\prime}$ is not $t$-factorizable with respect to $(\pi^{\prime}, {\bf a}^{\prime})$}
        \State Choose any such $v$ with $\phi^{\prime}(v) = ({\bf a}^{\prime}, \Delta^{\prime}, \pi^{\prime}, \boldsymbol{\lambda}^{\prime})$.
        \State Write ${\bf a}^{\prime} = (a_1^{\prime}, \dots, a_n^{\prime})$, $\pi^{\prime} = (V_1^{\prime}, \dots, V_n^{\prime})$, and let $v_{t,\max}^{\prime}$ be the maximal element of $V_t^{\prime}$.
        \State Let $W, W^{\prime}$ denote the vertex sets of $\dl_{\Delta^{\prime}}(v_{t,\max}^{\prime})$, $\lk_{\Delta^{\prime}}(v_{t,\max}^{\prime})$ respectively.
        \State Construct a binary tree $T$ from $T^{\prime}$ by adding two children $v_{\text{left}}$ and $v_{\text{right}}$ to $v$.
        \State Treat $T^{\prime}$ as a subgraph of $T$. Let $V^{\prime}$ denote the vertex set of $T^{\prime}$.
        \State Define the vertex labeling $\phi$ of $T$ by
        \[\phi(u) = \begin{cases} \phi^{\prime}(u), & \text{ if }u\in V^{\prime}\backslash\{v\};\\
                                    t, & \text{ if }u=v;\\
                                    ({\bf a}^{\prime}, \dl_{\Delta^{\prime}}(v_{t,\max}^{\prime}), \pi^{\prime} \cap W, \kappa(\pi^{\prime} \cap W, \boldsymbol{\lambda}^{\prime} - \boldsymbol{\delta}_{t,n})), & \text{ if }u=v_{\text{left}};\\
                                    ({\bf a}^{\prime} - \boldsymbol{\delta}_{t,n}, \lk_{\Delta^{\prime}}(v_{t,\max}^{\prime}), \pi^{\prime} \cap W^{\prime}, \kappa(\pi^{\prime} \cap W^{\prime}, \boldsymbol{\lambda}^{\prime} - \boldsymbol{\delta}_{t,n})), & \text{ if }u=v_{\text{right}}.\end{cases}\]
        \State $\text{seq} \gets (\text{seq}, (T, \phi))$ (i.e. seq with the term $(T, \phi)$ appended).
    \Else
        \State $t \gets t-1$.
    \EndIf
\EndWhile
\end{algorithmic}
\end{center}

Clearly $|V_t^{\prime}| = a_t^{\prime}$ or $a_t^{\prime} = 0$ implies $\Delta^{\prime}$ is $t$-factorizable with respect to $(\pi^{\prime}, {\bf a}^{\prime})$, so if $\Delta^{\prime}$ is not $t$-factorizable with respect to $(\pi^{\prime}, {\bf a}^{\prime})$, then $|V_t^{\prime}| > a_t^{\prime} > 0$, and Corollary \ref{cor:Macaulay-shedding-vertices} implies $v_{t,\max}^{\prime}$ is a Macaulay shedding vertex of $\Delta^{\prime}$. Since $\Delta$ is ${\bf a}$-Macaulay decomposable, this algorithm must eventually terminate, independent of the choices of $v$ made. Let $(T_0, \phi_0), (T_1, \phi_1), \dots, (T_k, \phi_k)$ be such a sequence constructed from $(\Delta, \pi)$. It is easy to see that the last pair $(T_k, \phi_k)$ (possibly $(T_0, \phi_0)$) is uniquely determined by $(\Delta, \pi)$, independent of the choices of $v$ made. Note however that the other pairs in the sequence do depend on the choices of $v$ at each iteration in the algorithm. We call any such sequence a {\it shedding sequence} of $(\Delta, \pi)$, and we say the last pair $(T_k, \phi_k)$ is the {\it shedding tree} of $(\Delta, \pi)$.

By construction, each $T_i$ is a trivalent planted binary tree with $T_{i-1}$ as a subgraph, and each $\phi_i$ is a vertex labeling of $T_i$, thus each pair $(T_i, \phi_i)$ is a labeled trivalent planted binary tree. The trivalent vertices of $T_i$ are labeled with integers in $[n]$, such that $\phi_i(y) \geq \phi_i(y^{\prime})$ for every pair $y, y^{\prime}$ of trivalent vertices in $T_i$ satisfying $y^{\prime} \in \mathcal{D}_{T_i}(y)$, i.e. $y^{\prime}$ is a descendant of $y$. Also, every terminal vertex of $T_i$ is labeled with a $4$-tuple $({\bf a}^{\prime}, \Delta^{\prime}, \pi^{\prime}, \boldsymbol{\lambda}^{\prime})$, where ${\bf a}^{\prime} = (a_1^{\prime}, \dots, a_n^{\prime}) \in \mathbb{N}^n$, $\Delta^{\prime}$ is a (non-empty) simplicial complex, $\pi^{\prime} = (V_1^{\prime}, \dots, V_n^{\prime})$ is an $n$-tuple of subsets of $V$, and $\boldsymbol{\lambda}^{\prime} \in \mathbb{P}^n$. We can check that $(\Delta^{\prime}, \overline{{\pi}^{\prime}})$ is a pure $\overline{{\bf a}^{\prime}}$-balanced color-shifted complex, and $\Delta^{\prime}$ is ${\bf a}^{\prime}$-Macaulay decomposable. Furthermore, if $(T_i, \phi_i) = (T_k, \phi_k)$ is the shedding tree of $(\Delta, \pi)$, then $\Delta^{\prime}$ is $t$-factorizable with respect to $(\pi^{\prime}, {\bf a}^{\prime})$ for every $t\in [n]$, hence $\Delta^{\prime}$ is the $\overline{{\bf a}^{\prime}}$-{\bone} of a simplex corresponding to the ordered partition $\overline{\pi^{\prime}}$ in this case.

\begin{figure}[h!t]
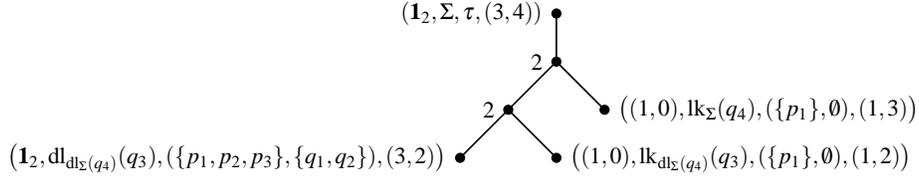

\centering
\scalebox{0.8}{\tikz[thick,scale=0.8]{
    \draw {
    (2,3) node[circle, draw, fill=black, inner sep=0pt, minimum width=4pt, label=left:{$({\bf 1}_2, \Sigma, \tau, (3,4))$}] {} -- (2,2) node[circle, draw, fill=black, inner sep=0pt, minimum width=4pt, label=left:{$2$}] {} -- (1,1) node[circle, draw, fill=black, inner sep=0pt, minimum width=4pt, label=left:{$2$}] {} -- (0,0) node[circle, draw, fill=black, inner sep=0pt, minimum width=4pt, label=left:{$\big({\bf 1}_2, \dl_{\dl_{\Sigma}(q_4)}(q_3), (\{p_1, p_2, p_3\}, \{q_1, q_2\}), (3,2)\big)$}] {}
    (2,2) -- (3,1) node[circle, draw, fill=black, inner sep=0pt, minimum width=4pt, label=right:{$\big((1,0), \lk_{\Sigma}(q_4), (\{p_1\}, \emptyset), (1,3)\big)$}] {}
    (1,1) -- (2,0) node[circle, draw, fill=black, inner sep=0pt, minimum width=4pt, label=right:{$\big((1,0), \lk_{\dl_{\Sigma}(q_4)}(q_3), (\{p_1\}, \emptyset), (1,2)\big)$}] {}
    };
}}
\caption{The shedding tree of the ${\bf 1}_2$-balanced complex $(\Sigma, \tau)$ given in Example \ref{example:shedding-tree}.}
\label{Fig:example-recursion}
\end{figure}

\begin{example}\label{example:shedding-tree}
Let $\tau = (P, Q)$ such that $P = \{p_1, p_2, p_3\}$ and $Q = \{q_1, q_2, q_3, q_4\}$ satisfy $p_1 < p_2 < p_3$ and $q_1 < q_2 < q_3 < q_4$. Consider the pure color-shifted $(1,1)$-balanced complex $(\Sigma, \tau)$, where
\[\Sigma = \Big\langle \Big\{\{p_1, q_1\}, \{p_2, q_1\}, \{p_3, q_1\}, \{p_1, q_2\}, \{p_2, q_2\}, \{p_3, q_2\}, \{p_1, q_3\}, \{p_1, q_4\}\Big\}\Big\rangle.\]
Then its shedding tree is given in Figure \ref{Fig:example-recursion}.
\end{example}

Given any term $(T_i, \phi_i)$ of a shedding sequence of $(\Delta, \pi)$, the restriction of $\phi_i$ to the terminal vertices of $T_i$ has four component functions, which we denote by $\phi_i^{[\mathcal{L}, {\bf a}]}$, $\phi_i^{[\mathcal{L}, \Delta]}$, $\phi_i^{[\mathcal{L}, \pi]}$, $\phi_i^{[\mathcal{L}, \boldsymbol{\lambda}]}$ respectively. Suppose $(T, \phi)$ is the shedding tree of $(\Delta, \pi)$, and let $(r_0, L, Y)$ be the $(\text{root}, 1, 3)$-triple of $T$. Note that $\phi^{[\mathcal{L}, {\bf a}]}(u) \in \mathbb{N}^n$ and $\phi^{[\mathcal{L}, \boldsymbol{\lambda}]}(u) \in \mathbb{P}^n$ for all $u\in L$. By the construction of $(T, \phi)$, it is easy to show that ${\bf 0}_n < \phi^{[\mathcal{L}, {\bf a}]}(u) \leq \phi^{[\mathcal{L}, \boldsymbol{\lambda}]}(u)$ for all $u\in L$. Also, by the repeated use of \eqref{eqn:del-link-split}, we have the following.

\begin{proposition}\label{prop:shedding-tree-fine-f-vector} For each ${\bf b} \in \mathbb{N}^n$ satisfying ${\bf b} \leq {\bf a}$,
\begin{equation}
f_{{\bf b}}(\Delta) = \sum_{u\in L} \binom{\phi^{[\mathcal{L}, \boldsymbol{\lambda}]}(u)}{\phi^{[\mathcal{L}, {\bf a}]}(u) + {\bf b} - {\bf a}}.
\end{equation}
\end{proposition}
For example, the balanced complex $(\Sigma, \tau)$ in Example \ref{example:shedding-tree} satisfies $f_{(1,1)}(\Sigma) = 8$, $f_{(1,0)}(\Sigma) = 3$, $f_{(0,1)}(\Sigma) = 4$, and we check that indeed $8 = \binom{(3,2)}{(1,1)} + \binom{(1,2)}{(1,0)} + \binom{(1,3)}{(1,0)}$, $3 = \binom{(3,2)}{(1,0)} + \binom{(1,2)}{(1,-1)} + \binom{(1,3)}{(1,-1)}$, and $4 = \binom{(3,2)}{(0,1)} + \binom{(1,2)}{(0,0)} + \binom{(1,3)}{(0,0)}$.

\begin{remark}
The entire Section \ref{subsec:ConstructionOfSheddingTrees} is still true if we replace every instance of ``color-shifted'' with ``color-compressed''.
\end{remark}

\subsection{Macaulay Trees}\label{subsec:MacaulayTrees}
Let ${\bf a} \in \mathbb{N}^n\backslash\{{\bf 0}_n\}$. Given any pure color-shifted $\overline{{\bf a}}$-balanced complex $(\Delta, \pi)$, it is clear from Proposition \ref{prop:shedding-tree-fine-f-vector} that the shedding tree $(T, \phi)$ of $(\Delta, \pi)$ encodes superfluous information for computing the fine $f$-vector of $\Delta$. Motivated by the desire to retain only the necessary numerical information needed to compute the fine $f$-vector of $\Delta$, we shall define the notion of an ${\bf a}$-Macaulay tree of $N$ for any $N\in \mathbb{P}$, which does not depend on the existence of any $\overline{{\bf a}}$-balanced complexes. First, we introduce a series of related definitions.

\begin{definition}
Let ${\bf a} \in \mathbb{N}^n$. An {\it ${\bf a}$-splitting tree} is a triple $(T, \mu, \nu)$ such that
\begin{enum*}
\item $T = (V, E)$ is a trivalent planted binary tree with $(\text{root},1,3)$-triple $(r_0, L, Y)$;
\item $\mu: Y \to [n]$ is a labeling of the trivalent vertices of $T$; and
\item $\nu: V \to \mathbb{Z}^n$ is a vertex labeling of $T$ recursively defined as follows.
\begin{enum*}
\item Define $\nu(r_0) = \nu(r_1) = {\bf a}$, where $r_1$ is the unique child of the root $r_0$.
\item For every $y\in Y$, if $\nu(y)$ was already defined, while $\nu(y_{\text{left}})$, $\nu(y_{\text{right}})$ are both not yet defined, then define $\nu(y_{\text{left}}) = \nu(y)$, $\nu(y_{\text{right}}) = \nu(y) - \boldsymbol{\delta}_{\mu(y),n}$.
\end{enum*}
\end{enum*}
Note that $\nu: V \to \mathbb{Z}^n$ is completely determined by $T, \mu, {\bf a}$, and we say $(T, \mu, \nu)$ is the ${\bf a}$-splitting tree {\it induced} by $(T, \mu)$.
\end{definition}

\begin{definition}
Let $T$ be a trivalent planted binary tree with $(\text{root},1,3)$-triple $(r_0, L, Y)$, and let $\varphi$ be any vertex labeling of $T$ such that $\varphi(Y) \subseteq [n]$ and $\varphi(L) \subseteq \mathbb{P}^n$. The {\it left-weight labeling} of $(T, \varphi)$ is the map $\omega: Y \cup L \to \mathbb{P}^n$ defined by
\begin{equation*}
\omega(v) := \Big(\varphi(v_m) + \sum_{i=1}^{m-1} \boldsymbol{\delta}_{\varphi(v_i), n}\Big),
\end{equation*}
where $v_1, \dots, v_m$ is the left relative sequence of $v$ in $T$, and we say $\omega(v)$ is the {\it left-weight} of $v$ in $(T, \varphi)$. In particular, if $v\in L$, then $\omega(v) = \varphi(v)$. We say $\omega$ is {\it proper} if $\omega(y_{\text{left}}) \geq \omega(y_{\text{right}})$ for every $y\in Y$.
\end{definition}

\begin{definition}
Let $n, N\in \mathbb{P}$, ${\bf a} \in \mathbb{N}^n\backslash\{{\bf 0}_n\}$. An {\it ${\bf a}$-Macaulay tree of $N$} is any pair $(T, \varphi)$ such that
\begin{enum*}
\item $T$ is a trivalent planted binary tree with $(\text{root},1,3)$-triple $(r_0, L, Y)$; and
\item $\varphi$ is a vertex labeling of $T$ satisfying all of the following conditions.
\begin{enum*}
\item $\varphi(r_0) = {\bf a}$, $\varphi(Y) \subseteq [n]$ and $\varphi(L) \subseteq \mathbb{P}^n$.
\item If $y, y^{\prime} \in Y$ satisfies $y^{\prime} \in \mathcal{D}_T(y)$, i.e. $y^{\prime}$ is a descendant of $y$, then $\varphi(y) \geq \varphi(y^{\prime})$.
\item If $y\in Y$ satisfies $\varphi(y) < n$, then $\varphi^{[t]}(u) = \varphi^{[t]}(u^{\prime})$ for all $u, u^{\prime} \in L \cap \mathcal{D}_T(y)$ and all $t\in [n]\backslash [\varphi(y)]$.
\item The left-weight labeling of $(T, \varphi)$ is proper.
\item The ${\bf a}$-splitting tree $(T, \varphi|_Y, \nu)$ induced by $(T, \varphi|_Y)$ satisfies $\varphi(u) \geq \nu(u) > {\bf 0}_n$ for all $u\in L$.
\item If $y\in Y$ satisfies $\nu^{[\varphi(y)]}(y) = 1$, then $\omega^{[\varphi(y)]}(x) = \omega^{[\varphi(y)]}(y) - 1$ for all $x\in \mathcal{D}_T(y_{\text{right}})$.
\item $N$ can be written as the sum $N = \sum_{u\in L} \binom{\varphi(u)}{\nu(u)}$.
\end{enum*}
\end{enum*}
\end{definition}

An {\it ${\bf a}$-Macaulay tree} is an ${\bf a}$-Macaulay tree of $N$ for some $N\in \mathbb{P}$, and a {\it Macaulay tree} is an ${\bf a}$-Macaulay tree for some ${\bf a} \in \mathbb{N}^n\backslash\{{\bf 0}_n\}$. Suppose $(T, \varphi)$ is a Macaulay tree, and let $(r_0, L, Y)$ be the $(\text{root},1,3)$-triple of $T$. The {\it weight} of $(T, \varphi)$ is the left-weight of the unique child of $r_0$. Given any $t\in \{0,1,\dots, n\}$, we say $x\in Y \cup L$ is a {\it $t$-leading vertex} of $(T, \varphi)$ if it satisfies the following two conditions (i): If $x\in Y$, then $\varphi(x) \leq t$; and (ii): If $\widetilde{x}$ is the parent of $x$ and $\widetilde{x} \neq r_0$, then $\varphi(\widetilde{x}) > t$. Let $\mathcal{L}_{(T, \varphi)}(t)$ be the set of all $t$-leading vertices of $(T, \varphi)$. Note that a $t$-leading vertex could possibly be a $t^{\prime}$-leading vertex for distinct $t, t^{\prime} \in \{0,1,\dots, n\}$, while not every vertex in $Y\cup L$ is necessarily a $t$-leading vertex for some $t\in \{0,1,\dots, n\}$. It is easy to show that for every $t\in \{0,1,\dots, n\}$ and every path $v_0, v_1, \dots, v_{\ell}$ from $v_0 = r_0$ to some terminal vertex $v_{\ell} \in L$, there exists a unique $i_t \in [\ell]$ such that $v_{i_t}$ is a $t$-leading vertex. In particular, the unique child of $r_0$ is always an $n$-leading vertex, and it is the only $n$-leading vertex in $Y\cup L$, while every terminal vertex is a $0$-leading vertex, i.e. $\mathcal{L}_{(T, \varphi)}(0) = L$.

\begin{definition}
Let $(T, \varphi)$ be a Macaulay tree, and let $(r_0, L, Y)$ be the $(\text{root}, 1,3)$-triple of $T$. A trivalent vertex $y\in Y$ is called a {\it cloning vertex} of $(T, \varphi)$ if it satisfies the following two conditions:
\begin{enum*}
\item The two subgraphs of $T$ induced by $\mathcal{D}_T(y_{\text{left}})$ and $\mathcal{D}_T(y_{\text{right}})$ are isomorphic as rooted trees, and they have the same corresponding $\varphi$-labels, i.e. if $\beta: \mathcal{D}_T(y_{\text{left}}) \to \mathcal{D}_T(y_{\text{right}})$ is the corresponding isomorphism, then $\varphi(\beta(x)) = \varphi(x)$ for all $x\in \mathcal{D}_T(y_{\text{left}})$.
\item Both $y_{\text{left}}$ and $y_{\text{right}}$ are $(\varphi(y)-1)$-leading vertices of $(T, \varphi)$.
\end{enum*}
If $(T, \varphi)$ has no cloning vertices, then we say $(T, \varphi)$ is {\it condensed}. In particular, $(T, \varphi)$ is condensed implies $\varphi(y_{\text{left}}) > \varphi(y_{\text{right}})$ for every $y\in Y$ such that $y_{\text{left}}$ and $y_{\text{right}}$ are both leaves.
\end{definition}

\begin{definition}
Let $(T, \varphi)$ be an ${\bf a}$-Macaulay tree of $N$ for some ${\bf a}\in \mathbb{N}^n\backslash \{{\bf 0}_n\}$, $N \in \mathbb{P}$. The {\it condensation} of $(T, \varphi)$ is a condensed ${\bf a}$-Macaulay tree $(\widetilde{T}, \widetilde{\varphi})$ of $N$ that is constructed from $(T, \varphi)$ via the following algorithm:
\begin{center}
\begin{algorithmic}
\small
\State $(\widetilde{T}, \widetilde{\varphi}) \gets (T, \varphi)$.
\While {$\exists$ a cloning vertex of $(\widetilde{T}, \widetilde{\varphi})$}
    \State Choose any cloning vertex $y$ of $(\widetilde{T}, \widetilde{\varphi})$.
    \State $\widetilde{\varphi}(y) \gets \widetilde{\varphi}(y_{\text{left}})$.
    \State $\widetilde{\varphi} \gets \widetilde{\varphi}|_{\widetilde{V}\backslash \mathcal{D}_{\widetilde{T}}(y_{\text{right}})}$ (where $\widetilde{V}$ is the vertex set of $\widetilde{T}$).
    \State $\widetilde{T} \gets$ subgraph of $\widetilde{T}$ induced by $\widetilde{V}\backslash \mathcal{D}_{\widetilde{T}}(y_{\text{right}})$.
    \State $(\widetilde{T}, \widetilde{\varphi}) \gets (\widetilde{T}/\{y, y_{\text{left}}\}, \widetilde{\varphi}/\{y, y_{\text{left}}\})$.
\EndWhile
\end{algorithmic}
\end{center}
\end{definition}

In defining the condensation $(\widetilde{T}, \widetilde{\varphi})$ of $(T, \varphi)$, the number of vertices of $\widetilde{T}$ becomes strictly smaller in each iteration of the while loop, hence the while loop is not an infinite loop. It is also easy to see that the pair $(\widetilde{T}, \widetilde{\varphi})$ obtained from the above algorithm is uniquely determined by $(T, \varphi)$, independent of the choice of the cloning vertex in each iteration of the while loop, hence $(\widetilde{T}, \widetilde{\varphi})$ is well-defined, and in particular, the condensation of a condensed Macaulay tree is itself. The fact that $(\widetilde{T}, \widetilde{\varphi})$ is a condensed ${\bf a}$-Macaulay tree of $N$ is straightforward and left to the reader as an exercise.

For the rest of this subsection, let ${\bf a} \in \mathbb{P}^n$, let $(\Delta, \pi)$ be a pure color-shifted ${\bf a}$-balanced complex, and let $(T_0, \phi_0)$, $(T_1, \phi_1)$, $\dots$, $(T_k, \phi_k)$ be a shedding sequence of $(\Delta, \pi)$, where $(r_0, L_i, Y_i)$ denotes the $(\text{root}, 1, 3)$-triple of each $T_i$. Clearly $r_0$ is the common root of $T_0, T_1, \dots, T_k$, and observe that
\begin{equation}\label{eqn:YuL-containment}
Y_0 \subsetneq Y_1 \subsetneq \cdots \subsetneq Y_k; \ \ \ \ \ Y_0 \cup L_0 \subsetneq Y_1 \cup L_1 \subsetneq \cdots \subsetneq Y_k \cup L_k.
\end{equation}
For every $y\in Y_k$, let $\alpha(y)$ be the largest integer $<k$ such that $y\in L_{\alpha(y)}$, while for every $u\in L_k$, set $\alpha(u) = k$. Also, for each $i\in \{0,1,\dots, k\}$, define the vertex labeling $\varphi_i$ of $T_i$ as follows.
\begin{equation*}
\varphi_i(u) = \begin{cases}
{\bf a}, & \text{ if }u = r_0;\\
\phi_i(u), & \text{ if }u \in Y_i;\\
\phi_i^{[\mathcal{L}, \boldsymbol{\lambda}]}(u), & \text{ if }u \in L_i.
\end{cases}
\end{equation*}
By construction, every $u\in Y_i$ is labeled with an integer in $[n]$, while every $u\in L_i$ is labeled with an $n$-tuple in $\mathbb{P}^n$, hence the left-weight labeling of $(T_i, \varphi_i)$ is well-defined, and we denote this left-weight labeling by $\omega_i$.

\begin{lemma}\label{lemma:left-weight}
If $v\in Y_k$, then $\omega_{\alpha(v)}(v) = \omega_{\alpha(v)+1}(v)$.
\end{lemma}

\begin{proof}
Write $\phi_{\alpha(v)}(v) = ({\bf a}^{\prime}, \Delta^{\prime}, \pi^{\prime}, \boldsymbol{\lambda}^{\prime})$, where ${\bf a}^{\prime} = (a_1^{\prime}, \dots, a_n^{\prime})$ and $\pi^{\prime} = (V_1^{\prime}, \dots, V_n^{\prime})$ is an ordered partition of the vertex set $V^{\prime}$ of $\Delta^{\prime}$. Let $t = \phi_{\alpha(v)+1}(v)$, let $v^{\prime}_{t,\max}$ be the maximal element of $V_t^{\prime}$, and denote the vertex set of $\dl_{\Delta^{\prime}}(v^{\prime}_{t, \max})$ by $W$. Also, write $\pi^{\prime} \cap W = (V_1^{\prime\prime}, \dots, V_n^{\prime\prime})$. Since $\Delta^{\prime}$ is color-shifted, every vertex $v^{\prime\prime} < v^{\prime}_{t, \max}$ in $V^{\prime}_t$ must be in $W$, hence $|V^{\prime\prime}_t| = |V^{\prime}_t|-1$. Note that $|V^{\prime}_t| > a_t^{\prime} > 0$ yields $|V^{\prime\prime}_t| > 0$, while every vertex in $V^{\prime}\backslash V_t^{\prime}$ is in $W$, thus $\kappa(\pi^{\prime} \cap W, \boldsymbol{\lambda}^{\prime} - \boldsymbol{\delta}_{t,n}) = \boldsymbol{\lambda}^{\prime} - \boldsymbol{\delta}_{t,n}$. Now, $\varphi_{\alpha(v)+1}(v_{\text{left}}) = \phi_{\alpha(v)+1}^{[\mathcal{L}, \boldsymbol{\lambda}]}(v_{\text{left}}) = \kappa(\pi^{\prime} \cap W, \boldsymbol{\lambda}^{\prime} - \boldsymbol{\delta}_{t,n})$ by construction, and $v, v_{\text{left}}$ is the left relative sequence of $v$ in $T_{\alpha(v)+1}$, therefore $\omega_{\alpha(v)+1}(v) = \varphi_{\alpha(v)+1}(v_{\text{left}}) + \boldsymbol{\delta}_{t,n} = \boldsymbol{\lambda}^{\prime} = \phi_{\alpha(v)}^{[\mathcal{L}, \boldsymbol{\lambda}]}(v) = \varphi_{\alpha(v)}(v) = \omega_{\alpha(v)}(v)$.
\end{proof}

\begin{proposition}\label{prop:left-weight-shedding-sequence}
For every $i\in \{0,1,\dots, k\}$, we have the following:
\begin{enum*}
\item $\omega_i(v) = \omega_j(v)$ for all $v\in Y_i \cup L_i$ and all integers $j$ satisfying $i\leq j\leq k$.\label{condition-left-weight-labeling-invariant}
\item $\omega_i$ is proper.\label{condition-left-weight-labeling-proper}
\item If the children of some $y\in Y_i$ are both in $L_k$, then $\varphi_i(y_{\text{left}}) > \varphi_i(y_{\text{right}})$.\label{condition-left-weight-labeling-both-leaves}
\item If $\varphi_i(x) < n$ for some $x\in Y_i$, then $\omega_i^{[t]}(x^{\prime}) = \omega_i^{[t]}(x)$ for all $x^{\prime} \in \mathcal{D}_{T_i}(x)$ and all $t\in [n]\backslash [\varphi_i(x)]$.\label{condition-orderly}
\end{enum*}
\end{proposition}

\begin{proof}
First of all, \ref{condition-left-weight-labeling-invariant} follows from Lemma \ref{lemma:left-weight} by the algorithmic construction of a shedding sequence. Next, we prove \ref{condition-left-weight-labeling-proper} and \ref{condition-left-weight-labeling-both-leaves} concurrently by induction on $i$, with the base case $i=0$ being vacuously true since $Y_0 = \emptyset$.

Choose an arbitrary $y\in Y_i$ and let $q := \alpha(y) < i$. If $q < i-1$, then both $y_{\text{left}}$ and $y_{\text{right}}$ are vertices in $T_{q+1}$, and $\omega_{q+1}$ is proper by induction hypothesis, hence $\omega_{q+1}(y_{\text{left}}) \geq \omega_{q+1}(y_{\text{right}})$. Using statement \ref{condition-left-weight-labeling-invariant}, we thus get $\omega_i(y_{\text{left}}) \geq \omega_i(y_{\text{right}})$. Furthermore, if $y_{\text{left}}$ and $y_{\text{right}}$ are both in $L_k$, then the induction hypothesis also yields $\varphi_{q+1}(y_{\text{left}}) > \varphi_{q+1}(y_{\text{right}})$, which implies $\varphi_i(y_{\text{left}}) > \varphi_i(y_{\text{right}})$.

Suppose $q = i-1$. Let $\phi_{q}(y) = ({\bf a}^{\prime}, \Delta^{\prime}, \pi^{\prime}, \boldsymbol{\lambda}^{\prime})$, and write ${\bf a}^{\prime} = (a_1^{\prime}, \dots, a_n^{\prime})$, $\pi^{\prime} = (V_1^{\prime}, \dots, V_n^{\prime})$. Let $t = \phi_i(y)$, let $v^{\prime}_{t,\max}$ be the maximal element of $V_t^{\prime}$, and let $W, W^{\prime}$, be the vertex sets of $\dl_{\Delta^{\prime}}(v^{\prime}_{t, \max})$, $\lk_{\Delta^{\prime}}(v^{\prime}_{t, \max})$ respectively. The maximality of $q$ implies $\omega_i(y_{\text{left}}) = \varphi_i(y_{\text{left}}) = \kappa(\pi^{\prime} \cap W, \boldsymbol{\lambda}^{\prime} - \boldsymbol{\delta}_{t,n})$ and $\omega_i(y_{\text{right}}) = \varphi_i(y_{\text{right}}) = \kappa(\pi^{\prime} \cap W^{\prime}, \boldsymbol{\lambda}^{\prime} - \boldsymbol{\delta}_{t,n})$. Note that $\lk_{\Delta^{\prime}}(v^{\prime}_{t, \max}) \subseteq \dl_{\Delta^{\prime}}(v^{\prime}_{t, \max})$ implies $W^{\prime} \subseteq W$. Note also that $\dl_{\Delta^{\prime}}(v^{\prime}_{t, \max})$ is $\overline{{\bf a}^{\prime}}$-balanced, while $\lk_{\Delta^{\prime}}(v^{\prime}_{t, \max})$ is $\overline{({\bf a}^{\prime} - \boldsymbol{\delta}_{t,n})}$-balanced, so $|V_t^{\prime}| > a_t^{\prime} > 0$ implies the $\emptyset$ entries of $\pi^{\prime} \cap W$ and $\pi^{\prime} \cap W^{\prime}$ are identical, thus $\kappa(\pi^{\prime} \cap W, \boldsymbol{\lambda}^{\prime} - \boldsymbol{\delta}_{t,n}) \geq \kappa(\pi^{\prime} \cap W^{\prime}, \boldsymbol{\lambda}^{\prime} - \boldsymbol{\delta}_{t,n})$, i.e. $\omega_i(y_{\text{left}}) \geq \omega_i(y_{\text{right}})$.

Suppose further that $y_{\text{left}}$ and $y_{\text{right}}$ are both in $L_k$, and recall we already have $\varphi_i(y_{\text{left}}) = \omega_i(y_{\text{left}}) \geq \omega_i(y_{\text{right}}) = \varphi_i(y_{\text{right}})$. If $\varphi_i(y_{\text{left}}) = \varphi_i(y_{\text{right}})$, then $\pi^{\prime} \cap W = \pi^{\prime} \cap W^{\prime}$. Since $y_{\text{left}}, y_{\text{right}} \in L_k$ implies $\dl_{\Delta^{\prime}}(v^{\prime}_{t, \max})$ is an $\overline{{\bf a}^{\prime}}$-rib of a simplex and $\lk_{\Delta^{\prime}}(v^{\prime}_{t, \max})$ is an $\overline{({\bf a}^{\prime} - \boldsymbol{\delta}_{t,n})}$-rib of a simplex, it follows that $\Delta^{\prime}$ is an $\overline{{\bf a}^{\prime}}$-rib of a simplex, which contradicts the fact that $y$ is not a terminal vertex in $T_k$. Hence, we must have $\varphi_i(y_{\text{left}}) > \varphi_i(y_{\text{right}})$ in this case, therefore completing the induction step for statements \ref{condition-left-weight-labeling-proper} and \ref{condition-left-weight-labeling-both-leaves}.

Finally, we prove statement \ref{condition-orderly}. Fix some $i>0$, assume $\varphi_i(x) < n$ for some $x\in Y_i$, and write $m := \varphi_i(x)$. Let $\phi_{\alpha(x)}(x) = (\widetilde{{\bf a}}, \widetilde{\Delta}, \widetilde{\pi}, \widetilde{\boldsymbol{\lambda}})$, write $\widetilde{{\bf a}} = (\widetilde{a}_1, \dots, \widetilde{a}_n)$, $\widetilde{\pi} = (\widetilde{V}_1, \dots, \widetilde{V}_n)$, $\widetilde{\boldsymbol{\lambda}} = (\widetilde{\lambda}_1, \dots, \widetilde{\lambda}_n)$, and let $\widetilde{V} := \bigcup_{i\in [n]} \widetilde{V}_i$ be the vertex set of $\widetilde{\Delta}$. From the construction of the shedding sequence, $\widetilde{\Delta}$ is $t$-factorizable with respect to $(\widetilde{\pi}, \widetilde{{\bf a}})$ for all $t\in [n]\backslash [m]$, hence $\widetilde{\Delta} = \big\langle \binom{\widetilde{V}_n}{\widetilde{a}_n} \big\rangle * \big\langle \binom{\widetilde{V}_{n-1}}{\widetilde{a}_{n-1}} \big\rangle * \cdots * \big\langle \binom{\widetilde{V}_{m+1}}{\widetilde{a}_{m+1}} \big\rangle * \Sigma$ for some (non-empty) simplicial complex $\Sigma$ with vertex set $\widetilde{V}_1 \cup \dots \cup \widetilde{V}_m$, which yields $\dl_{\widetilde{\Delta}}(v) = \big\langle \binom{\widetilde{V}_n}{\widetilde{a}_n} \big\rangle * \big\langle \binom{\widetilde{V}_{n-1}}{\widetilde{a}_{n-1}} \big\rangle * \cdots * \big\langle \binom{\widetilde{V}_{m+1}}{\widetilde{a}_{m+1}} \big\rangle * \dl_{\Sigma}(v)$, and $\lk_{\widetilde{\Delta}}(v) = \big\langle \binom{\widetilde{V}_n}{\widetilde{a}_n} \big\rangle * \big\langle \binom{\widetilde{V}_{n-1}}{\widetilde{a}_{n-1}} \big\rangle * \cdots * \big\langle \binom{\widetilde{V}_{m+1}}{\widetilde{a}_{m+1}} \big\rangle * \lk_{\Sigma}(v)$, for all $v\in \widetilde{V}_1 \cup \dots \cup \widetilde{V}_m$. For every $y\in \mathcal{D}_{T_i}(x) \cap Y$, since $\varphi_i(y) \leq \varphi_i(x) = m$ implies the Macaulay shedding vertex used to construct the children of $y$ in $T_{\alpha(y)+1}$ is chosen from $\widetilde{V}_1 \cup \dots \cup \widetilde{V}_m$, it then follows inductively that for every $x^{\prime} \in \mathcal{D}_{T_i}(x)$, the simplicial complex $\phi_{\alpha(x^{\prime})}^{[\mathcal{L}, \Delta]}(x^{\prime})$ satisfies $\phi_{\alpha(x^{\prime})}^{[\mathcal{L}, \Delta]}(x^{\prime}) = \big\langle \binom{\widetilde{V}_n}{\widetilde{a}_n} \big\rangle * \big\langle \binom{\widetilde{V}_{n-1}}{\widetilde{a}_{n-1}} \big\rangle * \cdots * \big\langle \binom{\widetilde{V}_{m+1}}{\widetilde{a}_{m+1}} \big\rangle * \Sigma^{\prime}$ for some (non-empty) simplicial complex $\Sigma^{\prime}$ (dependent on the choice of $x^{\prime}$) with its vertex set contained in $\widetilde{V}_1 \cup \dots \cup \widetilde{V}_m$. Consequently, statement \ref{condition-left-weight-labeling-invariant} yields $\omega_i^{[t]}(x^{\prime}) = \omega_{\alpha(x^{\prime})}^{[t]}(x^{\prime}) = \phi_{\alpha(x^{\prime})}^{[\mathcal{L}, \boldsymbol{\lambda}]}(x^{\prime}) = \widetilde{\lambda}_t = \omega_i^{[t]}(x)$ for all $x^{\prime} \in \mathcal{D}_{T_i}(x)$ and all $t\in [n]\backslash [m]$.
\end{proof}

\begin{theorem}\label{thm:shedding-tree=>Macaulay-tree}
The pair $(T_i, \varphi_i)$ is an ${\bf a}$-Macaulay tree for every $i \in \{0, 1, \dots, k\}$. Furthermore, if $\Delta$ has $N$ facets, then $(T_k, \varphi_k)$ is a condensed ${\bf a}$-Macaulay tree of $N$.
\end{theorem}

\begin{proof}
Choose an arbitrary $i\in \{0,1,\dots, k\}$, and let $(T_i, \varphi|_{Y_i}, \nu_i)$ be the ${\bf a}$-splitting tree induced by $(T_i, \varphi|_{Y_i})$. Clearly $\varphi_i(r_0) = {\bf a}$, $\varphi_i(Y_i) \subseteq [n]$ and $\varphi_i(L_i) \subseteq \mathbb{P}^n$ by the definition of $\varphi_i$, while the construction of a shedding sequence yields $\varphi_i(y) \geq \varphi_i(y^{\prime})$ for every $y, y^{\prime} \in Y_i$ satisfying $y^{\prime} \in \mathcal{D}_{T_i}(y)$. It is also easy to show that $\varphi_i(u) = \phi_i^{[\mathcal{L}, \boldsymbol{\lambda}]}(u) \geq \phi_i^{[\mathcal{L}, {\bf a}]}(u) = \nu_i(u) > {\bf 0}_n$ for all $u\in L_i$. Proposition \ref{prop:left-weight-shedding-sequence} says $\omega_i$ is proper, and if $\varphi_i(x) < n$ for some $x\in Y_i$, then $\omega_i^{[t]}(x^{\prime}) = \omega_i^{[t]}(x)$ for all $x^{\prime} \in \mathcal{D}_{T_i}(x)$ and all $t\in [n]\backslash [\varphi_i(x)]$. This implies $\varphi_i^{[t]}(u) = \varphi_i^{[t]}(u^{\prime})$ for all $u, u^{\prime} \in L_i \cap \mathcal{D}_{T_i}(x)$ and all $t\in [n]\backslash [\varphi_i(x)]$. Furthermore, if $y\in Y_i$ satisfies $\nu_i^{[\varphi_i(y)]}(y) = 1$, then $\nu_i^{[\varphi_i(y)]}(y_{\text{right}}) = 0$, so by the definitions of $\phi_{\alpha(y)+1}^{[\mathcal{L}, \boldsymbol{\lambda}]}$ and $\kappa$, we get that $\omega_i^{[\varphi_i(y)]}(y_{\text{right}}) = \varphi_{\alpha(y)+1}^{[\varphi_i(y)]}(y_{\text{right}})$ is the $\text{$\varphi_i(y)$-th entry of }\phi_{\alpha(y)+1}^{[\mathcal{L}, \boldsymbol{\lambda}]}(y_{\text{right}})$, which equals $\varphi_{\alpha(y)}^{[\varphi_i(y)]}(y) - 1 = \omega_i^{[\varphi_i(y)]}(y) - 1$. Hence, $\omega_i^{[\varphi_i(y)]}(x) = \omega_i^{[\varphi_i(y)]}(y) - 1$ for all $x\in \mathcal{D}_{T_i}(y_{\text{right}})$, and we conclude that $(T_i, \varphi_i)$ is an ${\bf a}$-Macaulay tree.

Next, suppose $(T_k, \varphi_k)$ is not condensed, let $z$ be a cloning vertex of $(T_k, \varphi_k)$, and let $\varphi_k(z) = q$. Also, let $\phi_{\alpha(z)}(z) = ({\bf a}^{\prime}, \Delta^{\prime}, \pi^{\prime}, \boldsymbol{\lambda}^{\prime})$, and write ${\bf a}^{\prime} = (a_1^{\prime}, \dots, a_n^{\prime})$, $\pi^{\prime} = (V_1^{\prime}, \dots, V_n^{\prime})$, $\boldsymbol{\lambda}^{\prime} = (\lambda_1^{\prime}, \dots, \lambda_n^{\prime})$. If $z_{\text{left}}$ and $z_{\text{right}}$ are both in $L$, then Proposition \ref{prop:left-weight-shedding-sequence}\ref{condition-left-weight-labeling-both-leaves} says $\varphi_k(z_{\text{left}}) > \varphi_k(z_{\text{right}})$, which contradicts the assumption that $z$ is a cloning vertex. This forces $z_{\text{left}}, z_{\text{right}} \in Y$ and $\varphi_k(z_{\text{left}}) = \varphi_k(z_{\text{right}}) < q$. By assumption, the two subgraphs of $T_k$ induced by $\mathcal{D}_{T_k}(z_{\text{left}})$ and $\mathcal{D}_{T_k}(z_{\text{right}})$ are isomorphic as rooted trees, and they have the same corresponding $\varphi_k$-labels, hence Proposition \ref{prop:left-weight-shedding-sequence}\ref{condition-orderly} implies $\omega_k^{[q]}(x) = \omega_k^{[q]}(z_{\text{left}}) = \lambda_q^{\prime} - 1$ for all $x\in \big(\mathcal{D}_{T_k}(z)\big)\backslash \{z\}$. Now, let $v_{q,\lambda_q}^{\prime}$ be the maximal element in $V_q^{\prime}$. Then $\dl_{\Delta^{\prime}}(v_{q,\lambda_q^{\prime}}^{\prime}) = \big\langle \binom{V_n^{\prime}}{a_n^{\prime}}\big\rangle * \dots * \big\langle \binom{V_{q+1}^{\prime}}{a_{q+1}^{\prime}} \big\rangle * \big\langle \binom{V_q^{\prime}\backslash \{v_{q,\lambda_q^{\prime}}^{\prime}\}}{a_q^{\prime}} \big\rangle * \Sigma$, and $\lk_{\Delta^{\prime}}(v_{q,\lambda_q^{\prime}}^{\prime}) = \big\langle \binom{V_n^{\prime}}{a_n^{\prime}}\big\rangle * \dots * \big\langle \binom{V_{q+1}^{\prime}}{a_{q+1}^{\prime}} \big\rangle * \big\langle \binom{V_q^{\prime}\backslash \{v_{q,\lambda_q^{\prime}}^{\prime}\}}{a_q^{\prime}-1} \big\rangle * \Sigma$, for some common simplicial complex $\Sigma$ with vertex set $V_1^{\prime} \cup \dots \cup V_{q-1}^{\prime}$. This implies $\Delta^{\prime} = \Big\langle \binom{V_n^{\prime}}{a_n^{\prime}}\Big\rangle * \dots * \Big\langle \binom{V_{q+1}^{\prime}}{a_{q+1}^{\prime}} \Big\rangle * \Big\langle \binom{V_q^{\prime}}{a_q^{\prime}} \Big\rangle * \Sigma$, thus $\Delta^{\prime}$ is $q$-factorizable with respect to $(\pi^{\prime}, {\bf a}^{\prime})$. However, the construction of the shedding sequence would then force $\varphi(z) \neq q$, which is a contradiction, therefore $(T_k, \varphi_k)$ must be condensed. Finally, if $\Delta$ has $N$ facets, then Proposition \ref{prop:shedding-tree-fine-f-vector} yields $N = f_{{\bf a}}(\Delta) = \sum_{u\in L_k} \binom{\phi_k^{[\mathcal{L}, \boldsymbol{\lambda}]}(u)}{\phi_k^{[\mathcal{L}, {\bf a}]}(u)} = \sum_{u\in L_k} \binom{\varphi_k(u)}{\nu_k(u)}$.
\end{proof}

\begin{remark}
In view of Theorem \ref{thm:shedding-tree=>Macaulay-tree}, we say $(T_i, \varphi_i)$ is the Macaulay tree {\it induced by} $(T_i, \phi_i)$.
\end{remark}

For $m\in \mathbb{P}$, $t\in [n]$, define $[m]_t := \{(1, t), (2, t), \dots, (m,t)\} \subseteq \mathbb{P} \times [n]$. If ${\bf x} = (x_1, \dots, x_n) \in \mathbb{P}^n$ and ${\bf y} = (y_1, \dots, y_n) \in \mathbb{N}^n$ satisfy ${\bf x} \geq {\bf y}$, define the simplicial complex $\big\langle \begin{smallmatrix} {\bf x} \\ {\bf y} \end{smallmatrix} \big\rangle := \big\langle \binom{[x_1]_1}{y_1} \big\rangle * \big\langle \binom{[x_2]_2}{y_2} \big\rangle * \dots * \big\langle \binom{[x_n]_n}{y_n} \big\rangle$.

\begin{definition}\label{defn:Psi-complex}
Let ${\bf a} \in \mathbb{N}^n\backslash\{{\bf 0}_n\}$, let $(T, \varphi)$ be an ${\bf a}$-Macaulay tree, and denote the $(\text{root}, 1, 3)$-triple of $T$ by $(r_0, L, Y)$. Also, let $(T, \varphi|_Y, \nu)$ be the ${\bf a}$-splitting tree induced by $(T, \varphi|_Y)$, and let $\omega$ be the left-weight labeling of $(T, \varphi)$. Recall that $\omega^{[j]}$ denotes the $j$-th component function of $\omega$ for any $j\in [n]$. For each $u\in L$, define the set
\begin{equation*}
\psi_{(T, \varphi)}(u) := \Big\{ \big(\omega^{[\varphi(v)]}(v), \varphi(v)\big): v\in Y, \text{ both }v \text{ and its right child are ancestors of }u\Big\},
\end{equation*}
and define the simplicial complex $\Psi_{(T, \varphi)}(u) := \big\langle \begin{smallmatrix} \varphi(u)\\ \nu(u) \end{smallmatrix} \big\rangle * \langle \{\psi_{(T, \varphi)}(u)\} \rangle$, whose vertex set is a subset of $\mathbb{P} \times [n]$. The complexes $\big\langle \begin{smallmatrix} \varphi(u)\\ \nu(u) \end{smallmatrix} \big\rangle$, $\langle \{\psi_{(T, \varphi)}(u)\} \rangle$ have disjoint vertex sets, so $\Psi_{(T, \varphi)}(u)$ is well-defined. Let $\Delta_{(T, \varphi)}$ be the union of all simplicial complexes $\Psi_{(T, \varphi)}(u)$ over all possible terminal vertices $u\in L$. Let $(s_1, \dots, s_n) \in \mathbb{P}^n$ be the weight of $(T, \varphi)$ (i.e. the left-weight of the unique child of the root of $T$), and define the ordered partition $\pi_{(T, \varphi)} := ([s_1]_1, \dots, [s_n]_n)$.
\end{definition}

\begin{proposition}\label{prop:Macaulay-tree=>shedding-tree}
Let $n, N\in \mathbb{P}$, ${\bf a} \in \mathbb{P}^n$, and let $(T, \varphi)$ be an ${\bf a}$-Macaulay tree of $N$. Then $(\Delta_{(T, \varphi)}, \pi_{(T, \varphi)})$ is a pure ${\bf a}$-balanced complex with $N$ facets.
\end{proposition}

\begin{proof}
Follow the notation as above. Choose an arbitrary $u\in L$, and let $V(u)$ be the vertex set of $\Psi_{(T, \varphi)}(u)$. For each $t\in [n]$, define $b_t(u) := \{(i,j) \in \psi_{(T, \varphi)}(u) : j=t\}$ and $B_t(u) := \{(i,j) \in \Psi_{(T, \varphi)}(u) : j=t\}$. Observe that $(|b_1(u)|, \dots, |b_n(u)|) = {\bf a} - \nu(u)$ by the construction of an ${\bf a}$-splitting tree, while $\big\langle \begin{smallmatrix} \varphi(u)\\ \nu(u) \end{smallmatrix} \big\rangle$ is the $\nu(u)$-{\bone} of a simplex corresponding to the ordered partition $([\varphi^{[1]}(u)]_1, [\varphi^{[2]}(u)]_2, \dots, [\varphi^{[n]}(u)]_n)$. Since $B_i(u) \subseteq [s_i]_i$ for each $i\in [n]$ implies $\pi_{(T, \varphi)} \cap V(u) = (B_1(u), \dots, B_n(u))$, it follows that $\big(\Psi_{(T, \varphi)}(u), \pi_{(T, \varphi)} \cap V(u)\big)$ is an ${\bf a}$-balanced complex. Also, the $\nu(u)$-{\bone} $\big\langle \begin{smallmatrix} \varphi(u)\\ \nu(u) \end{smallmatrix} \big\rangle$ and the simplex $\langle \{\psi_{(T, \varphi)}(u)\} \rangle$ are both pure, so $\Psi_{(T, \varphi)}(u)$ is pure. Furthermore, $\psi_{(T, \varphi)}(u) = \psi_{(T, \varphi)}(u^{\prime})$ if and only if $u = u^{\prime}$, thus $\Delta_{(T, \varphi)} = \bigsqcup_{u\in L} \Psi_{(T, \varphi)}(u)$, which has $\sum_{u\in L} \binom{\varphi(u)}{\nu(u)} = N$ facets, by the definition of $(T, \varphi)$. Finally, let $r_1$ denote the unique child of $r_0$, let $v_1, \dots, v_m$ be the left relative sequence of $r_1$, and let $u_i \in L$ be the right-most relative of $v_i$ for each $i\in [m]$. Since $\omega(r_1) = (s_1, \dots, s_n)$, the construction of $\omega$ yields $\bigcup_{i\in [n]} [s_i]_i \subseteq \bigcup_{i\in [m]} V(u_i) \subseteq \bigcup_{i\in [n]} [s_i]_i$, therefore $(\Delta_{(T, \varphi)}, \pi_{(T, \varphi)})$ is a pure ${\bf a}$-balanced complex with $N$ facets.
\end{proof}

\begin{figure}[h!t]
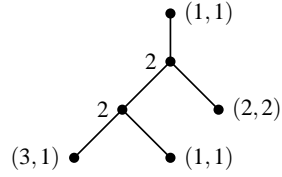

\centering
\scalebox{0.8}{\tikz[thick,scale=0.8]{
    \draw {
    (2,3) node[circle, draw, fill=black, inner sep=0pt, minimum width=4pt, label=right:{$(1,1)$}] {} -- (2,2) node[circle, draw, fill=black, inner sep=0pt, minimum width=4pt, label=left:{$2$}] {} -- (1,1) node[circle, draw, fill=black, inner sep=0pt, minimum width=4pt, label=left:{$2$}] {} -- (0,0) node[circle, draw, fill=black, inner sep=0pt, minimum width=4pt, label=left:{$(3,1)$}] {}
    (2,2) -- (3,1) node[circle, draw, fill=black, inner sep=0pt, minimum width=4pt, label=right:{$(2,2)$}] {}
    (1,1) -- (2,0) node[circle, draw, fill=black, inner sep=0pt, minimum width=4pt, label=right:{$(1,1)$}] {}
    };
}}
\caption{A $(1,1)$-Macaulay tree $(T, \varphi)$ where $(\Delta_{(T, \varphi)}, \pi_{(T, \varphi)})$ is not color-shifted.}
\label{Fig:balanced-complex-from-Mac-tree-not-necessarily-color-shifted}
\end{figure}
\begin{figure}[h!t]
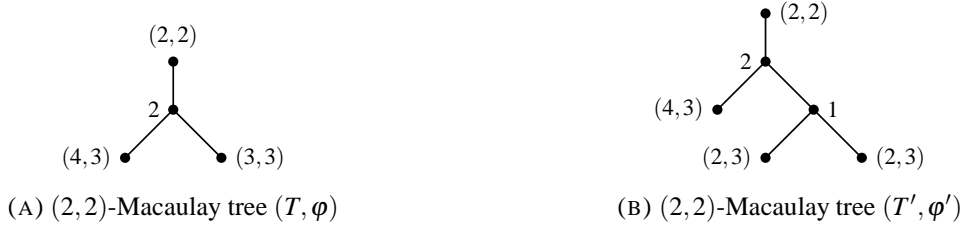

\centering
\begin{subfigure}[b]{0.45\textwidth}
\centering
\scalebox{0.8}{\tikz[thick,scale=0.8]{
    \draw {
    (2,3) node[circle, draw, fill=black, inner sep=0pt, minimum width=4pt, label=above:{$(2,2)$}] {} -- (2,2) node[circle, draw, fill=black, inner sep=0pt, minimum width=4pt, label=left:{$2$}] {} -- (1,1) node[circle, draw, fill=black, inner sep=0pt, minimum width=4pt, label=left:{$(4,3)$}] {}
    (2,2) -- (3,1) node[circle, draw, fill=black, inner sep=0pt, minimum width=4pt, label=right:{$(3,3)$}] {}
    };
}}
\caption{$(2,2)$-Macaulay tree $(T, \varphi)$}
\end{subfigure}
\begin{subfigure}[b]{0.45\textwidth}
\centering
\scalebox{0.8}{\tikz[thick,scale=0.8]{
    \draw {
    (2,3) node[circle, draw, fill=black, inner sep=0pt, minimum width=4pt, label=right:{$(2,2)$}] {} -- (2,2) node[circle, draw, fill=black, inner sep=0pt, minimum width=4pt, label=left:{$2$}] {} -- (3,1) node[circle, draw, fill=black, inner sep=0pt, minimum width=4pt, label=right:{$1$}] {} -- (4,0) node[circle, draw, fill=black, inner sep=0pt, minimum width=4pt, label=right:{$(2,3)$}] {}
    (2,2) -- (1,1) node[circle, draw, fill=black, inner sep=0pt, minimum width=4pt, label=left:{$(4,3)$}] {}
    (3,1) -- (2,0) node[circle, draw, fill=black, inner sep=0pt, minimum width=4pt, label=left:{$(2,3)$}] {}
    };
}}
\caption{$(2,2)$-Macaulay tree $(T^{\prime}, \varphi^{\prime})$}
\end{subfigure}
\caption{Distinct $(2,2)$-Macaulay trees $(T, \varphi)$ and $(T^{\prime}, \varphi^{\prime})$ such that $(\Delta_{(T, \varphi)}, \pi_{(T, \varphi)}) = (\Delta_{(T^{\prime}, \varphi^{\prime})}, \pi_{(T^{\prime}, \varphi^{\prime})})$.}
\label{Fig:distinct-Mac-trees-yield-same-complex}
\end{figure}
\begin{remark}\label{remark:balanced-complex-from-Mac-tree-not-necessarily-color-shifted}
The pure ${\bf a}$-balanced complex $(\Delta_{(T, \varphi)}, \pi_{(T, \varphi)})$ is not necessarily color-shifted. For example, the $(1,1)$-Macaulay tree $(T, \varphi)$ in Figure \ref{Fig:balanced-complex-from-Mac-tree-not-necessarily-color-shifted} yields
\begin{equation*}
\Delta_{(T, \varphi)} = \Big\langle\Big\{ \begin{matrix} \{(1,1), (1,2)\},& \{(2,1), (1,2)\},& \{(3,1), (1,2)\},\\
\{(1,1), (2,2)\},& \{(1,1), (3,2)\},& \{(2,1), (3,2)\} \end{matrix} \Big\}\Big\rangle,
\end{equation*}
yet $\{(2,1), (2,2)\} \not\in \Delta_{(T, \varphi)}$, which implies $(\Delta_{(T, \varphi)}, \pi_{(T, \varphi)})$ is not color-shifted. In this example, $(T, \varphi)$ is condensed, so $(T, \varphi)$ being condensed does not necessarily imply $(\Delta_{(T, \varphi)}, \pi_{(T, \varphi)})$ is color-shifted.
\end{remark}

\begin{remark}\label{remark:distinct-Mac-trees-yield-same-complex}
If $(\Delta, \pi)$ is a pure color-shifted ${\bf a}$-balanced complex such that $(\Delta, \pi) = (\Delta_{(T, \varphi)}, \pi_{(T, \varphi)})$ for some ${\bf a}$-Macaulay tree $(T, \varphi)$, then knowing $(\Delta, \pi)$ does not uniquely determine $(T, \varphi)$. For example, in Figure \ref{Fig:distinct-Mac-trees-yield-same-complex}, $(T, \varphi)$ and $(T^{\prime}, \varphi^{\prime})$ are two distinct $(2,2)$-Macaulay trees such that $(\Delta_{(T, \varphi)}, \pi_{(T, \varphi)}) = (\Delta_{(T^{\prime}, \varphi^{\prime})}, \pi_{(T^{\prime}, \varphi^{\prime})})$. However, the condensations of $(T, \varphi)$ and $(T^{\prime}, \varphi^{\prime})$ are identical, and by the definition of a shedding sequence, it is straightforward to show (e.g. by induction on the number of vertices of $\Delta$) that the condensation of $(T, \varphi)$ is the ${\bf a}$-Macaulay tree induced by the shedding tree of $(\Delta_{(T, \varphi)}, \pi_{(T, \varphi)})$.
\end{remark}

\subsection{Compressed-like Compatible Macaulay Trees}\label{subsec:CompressedLikeCompatibleMacTrees}
Let ${\bf a} \in \mathbb{P}^n$. Given any pure color-shifted ${\bf a}$-balanced complex $(\Delta, \pi)$, we can construct the ${\bf a}$-Macaulay tree induced by the shedding tree of $(\Delta, \pi)$, and we showed this Macaulay tree must be condensed. Conversely, given any ${\bf a}$-Macaulay tree $(T, \varphi)$ that is not necessarily condensed, we can construct the pure ${\bf a}$-balanced complex $(\Delta_{(T, \varphi)}, \pi_{(T, \varphi)})$. As shown in Remark \ref{remark:balanced-complex-from-Mac-tree-not-necessarily-color-shifted}, $(\Delta_{(T, \varphi)}, \pi_{(T, \varphi)})$ is not necessarily color-shifted, independent of whether $(T, \varphi)$ is condensed. In this subsection, we introduce the notions of `compressed-like' and `compatible' for ${\bf a}$-Macaulay trees, and we will establish the bijection
\begin{equation}
\left\{\begin{tabular}{c}\text{isomorphism classes of pure color-compressed}\\ \text{${\bf a}$-balanced complexes with $N$ facets} \end{tabular} \right\} \longleftrightarrow \left\{\begin{tabular}{c}\text{condensed compressed-like compatible}\\ \text{${\bf a}$-Macaulay trees of $N$} \end{tabular} \right\}
\end{equation}
with the maps $(\Delta, \pi) \longmapsto$ ${\bf a}$-Macaulay tree induced by the shedding tree of $(\Delta, \pi)$, and $(\Delta_{(T, \varphi)}, \pi_{(T, \varphi)}) \longmapsfrom (T, \varphi)$.

\begin{definition}\label{defn:compressed-like}
Let ${\bf a} \in \mathbb{N}^n\backslash\{{\bf 0}_n\}$ and let $(T, \varphi)$ be an ${\bf a}$-Macaulay tree. Denote the left-weight labeling of $(T, \varphi)$ by $\omega$, and let $(r_0, L, Y)$ be the $(\text{root}, 1, 3)$-triple of $T$. We say $(T, \varphi)$ is {\it compressed-like} if the following conditions hold for every $t\in [n]$ and every $y\in Y$ satisfying $\varphi(y) = t$:
\begin{enum*}
\item\label{defn:compressed-likeCondi} If $y$ is a descendant of $x_{\text{left}}$ for some $x \in Y$ satisfying $\varphi(x) = t$, and $x_0, x_1, \dots, x_{\ell}$ is the path (of length $\ell\geq 2$) from $x_0 = x$ to $x_{\ell} = y_{\text{right}}$, then $\omega^{[t]}(x_i) = \omega^{[t]}(x) - i$ for all $i\in [\ell]$.
\item\label{defn:compressed-likeCondii} Let $y_1, \dots, y_m$ be the right relative sequence of $y_{\text{left}}$ in $T$. If $k\in [m]$ is the smallest integer such that $y_k \in L$ or $\varphi(y_k) \neq t$, then $\omega^{[t^{\prime}]}(y_k) \geq \omega^{[t^{\prime}]}(y_{\text{right}})$ for all $t^{\prime} \in [n]\backslash \{t\}$.
\end{enum*}
\end{definition}

\begin{lemma}\label{lemma:compressed-like-n=1}
Let $k, N\in \mathbb{P}$, and let $(\Delta, V)$ be a pure compressed $(k-1)$-dimensional simplicial complex. Then the Macaulay tree induced by the shedding tree of $(\Delta, V)$ is compressed-like.
\end{lemma}

\begin{proof} Let $(T, \varphi)$ be the $k$-Macaulay tree induced by the shedding tree of $(\Delta, V)$, let $(r_0, L, Y)$ be the $(\text{root}, 1, 3)$-triple of $T$, and let $r_1$ be the unique child of $r_0$. Without loss of generality, assume $V = [n]$ and $n>k$. By definition, $\dl_{\Delta}(n)$ and $\lk_{\Delta}(n)$ are both pure and compressed. In particular, $\dl_{\Delta}(n) = \langle \binom{n-1}{k} \rangle$, hence the left child of $r_1$ is a terminal vertex. By induction on $n$, the left child of every $y\in Y$ is a terminal vertex, thus $(T, \varphi)$ is trivially compressed-like.
\end{proof}

\begin{proposition}\label{prop:condensed-compressed-case-n=1}
Let $k, N\in \mathbb{P}$. Then there exists a unique condensed compressed-like $k$-Macaulay tree $(T, \varphi)$ of $N$, and the left child of every trivalent vertex of $T$ is a terminal vertex.
\end{proposition}

\begin{proof}
First of all, Lemma \ref{lemma:compressed-like-n=1} and Theorem \ref{thm:shedding-tree=>Macaulay-tree} imply the existence of a condensed compressed-like $k$-Macaulay tree $(T, \varphi)$ of $N$. Let $(r_0, L, Y)$ be the $(\text{root}, 1, 3)$-triple of $T$, and note that $\varphi(y) = 1$ for all $y\in Y$. Let $(T, \varphi|_Y, \nu)$ be the $k$-splitting tree induced by $(T, \varphi|_Y)$, let $\omega$ be the left-weight labeling of $(T, \varphi)$, and let $u_1, \dots, u_m$ be all the terminal vertices (i.e. $|L| = m$), arranged in depth-first order. Clearly $|Y| = m-1$, and we let $y_1, \dots, y_{m-1}$ be all the trivalent vertices arranged in depth-first order.

Next, we show that the left child of every trivalent vertex in $T$ is a terminal vertex. Suppose not, and say $z := y_{\text{left}} \not\in L$ for some $y\in Y$. Let $z = z_1, \dots, z_{\ell}$ be the right relative sequence of $z$ in $T$ for some $\ell\geq 2$, and let $z^{\prime}$ be the left child of $z_{\ell-1}$. Note that $z_{\ell}$ is the right child of $z_{\ell-1}$, so $(T, \varphi)$ is compressed-like implies $\omega(z_{\ell}) = \omega(z_{\ell-1})-1$, yet the construction of $\omega$ yields $\omega(z^{\prime}) = \omega(z_{\ell-1})-1$, which contradicts the assumption that $(T, \varphi)$ is condensed. Therefore, the left child of every $y\in Y$ is in $L$ as claimed.

Consequently, $y_i$ is the parent of $u_i$ for every $i\in [m-1]$, while $y_{m-1}$ is the parent of $u_m$. If we define $y_m := u_m$, then the right relative sequence of $y_1$ is precisely $y_1, \dots, y_m$. Thus, $T$ is uniquely determined by $m$, and $\nu(u_i) = k+1-i$ for every $i\in [m]$. Since $\omega$ is proper implies $\omega(u_i) \geq \omega(y_{i+1})$ for all $i\in [m-1]$, we get $\varphi(u_i) = \omega(u_i) \geq \omega(y_{i+1}) > \omega(y_{i+1}) - 1 = \omega(u_{i+1}) = \varphi(u_{i+1})$ for all $i\in [m-2]$. Also, $(T, \varphi)$ is condensed, so $\varphi(u_{m-1}) = \omega(u_{m-1}) > \omega(u_m) = \varphi(u_m)$. This means $N = \sum_{i=1}^m \binom{\varphi(u_i)}{\nu(u_i)} = \sum_{i=1}^m \binom{\varphi(u_i)}{k+1-i}$, where $\varphi(u_1) > \dots > \varphi(u_m)\geq \nu(u_m) \geq 1$, therefore the uniqueness of $(T, \varphi)$ follows from the uniqueness of the $k$-th Macaulay representation of $N$.
\end{proof}

\begin{theorem}\label{thm:color-compressed=>compressed-like}
If $(\Delta, \pi)$ is a pure color-compressed ${\bf a}$-balanced complex for some ${\bf a} \in \mathbb{P}^n$, then the Macaulay tree induced by the shedding tree of $(\Delta, \pi)$ is compressed-like.
\end{theorem}

\begin{proof}
For convenience, identify each vertex $v_{i,j}$ of $\Delta$ with the pair $(j,i)$, so that each $V_i$ in $\pi = (V_1, \dots, V_n)$ is identified with $[\lambda_i]_i$. Choose any shedding sequence $(T_0, \phi_0), (T_1, \phi_1), \dots, (T_k, \phi_k)$ of $(\Delta, \pi)$, let $(r_0, L_i, Y_i)$ be the $(\text{root}, 1,3)$-triple of each $T_i$, let $(T, \varphi)$ be the Macaulay tree induced by $(T, \phi) := (T_k, \phi_k)$, and denote the left-weight labeling of $(T, \varphi)$ by $\omega$. For every $y\in Y_k$, let $\alpha(y)$ be the largest integer $< k$ such that $y \in L_{\alpha(y)}$, while for every $u\in L_k$, set $\alpha(u) = k$. 

Clearly $(T, \varphi)$ is compressed-like when $k=0$, so assume $k\geq 1$. Fix $t\in [n]$, choose some $x\in Y_k$ such that $\varphi(x) = t$, and let $x_1, \dots, x_{\ell}$ be the right relative sequence of $x_{\text{left}}$. If $x_{\text{left}} \not\in L_k$, then choose any $\ell^{\prime} \in [\ell-1]$ such that $\varphi(x_{\ell^{\prime}}) = t$, while if $x_{\text{left}} \in L_k$, then set $\ell^{\prime} = 0$. In either case, we claim that $\omega^{[t]}(x_i) = \omega^{[t]}(x) - i$ for all $i\in [\ell^{\prime}+1]$.

The claim is trivially true if $x_{\text{left}} \in L_k$, so assume $x_{\text{left}} \not\in L_k$. The construction of a shedding sequence yields $\varphi(x_j) = t$ for every $j\in [\ell^{\prime}]$. Let $\phi_{\alpha(x)}(x) = ({\bf a}^{\prime}, \Delta^{\prime}, \pi^{\prime}, \boldsymbol{\lambda}^{\prime})$, where ${\bf a}^{\prime} = (a_1^{\prime}, \dots, a_n^{\prime})$ and $\boldsymbol{\lambda}^{\prime} = (\lambda_1^{\prime}, \dots, \lambda_n^{\prime})$. Define $\Delta_i := \phi_{\alpha(x_i)}^{[\mathcal{L}, \Delta]}(x_i)$ for each $i\in [\ell]$, and note that $\Delta_1 = \dl_{\Delta^{\prime}}\big((\lambda_t^{\prime},t)\big)$. Since $(\Delta^{\prime}, \pi^{\prime})$ is color-compressed, every $a_t^{\prime}$-set in $\binom{[\lambda_t^{\prime}-1]_t}{a_t^{\prime}}$ is contained in (some facet of) $\Delta_1$, thus by induction on $i \in [\ell^{\prime}+1]$, every $(a_t^{\prime}+1-i)$-set in $\binom{[\lambda_t^{\prime}-i]_t}{a_t^{\prime}+1-i}$ is contained in (some facet of) $\Delta_i$. This implies $\omega^{[t]}(x_i) = \text{$t$-th entry of }\phi_{\alpha(x_i)}^{[\mathcal{L}, \boldsymbol{\lambda}]}(x_i) = \lambda_t^{\prime} - i = \omega^{[t]}(x) - i$ for all $i\in [\ell^{\prime}+1]$, so our claim is true in both cases. Condition \ref{defn:compressed-likeCondi} in Definition \ref{defn:compressed-like} holds by repeatedly using this claim and the definition of $\omega$.

Condition \ref{defn:compressed-likeCondii} in Definition \ref{defn:compressed-like} is vacuously true if $n=1$, so assume $n>1$. Fix some $t^{\prime} \in [n]\backslash\{t\}$, choose any $y\in Y$ such that $\varphi(y) = t$, let $y_1, \dots, y_m$ be the right relative sequence of $y_{\text{left}}$ in $T$, and let $q\in [m]$ be the smallest integer such that $y_q \in L$ or $\varphi(y_q) \neq t$. For this condition to hold, we need to show that $\omega^{[t^{\prime}]}(y_q) \geq \omega^{[t^{\prime}]}(y_{\text{right}})$. If $q=1$, then $\omega$ is proper implies $\omega^{[t^{\prime}]}(y_q) = \omega^{[t^{\prime}]}(y_{\text{left}}) \geq \omega^{[t^{\prime}]}(y_{\text{right}})$, and we are done, so assume $q\geq 2$. For the rest of this proof, let $\phi_{\alpha(y_{\text{right}})} = ({\bf a}^{\prime\prime}, \Delta^{\prime\prime}, \pi^{\prime\prime}, \boldsymbol{\lambda}^{\prime\prime})$ and let $\phi_{\alpha(y)} = (\widetilde{{\bf a}}, \widetilde{\Delta}, \widetilde{\pi}, \widetilde{\boldsymbol{\lambda}})$, where ${\bf a}^{\prime\prime} = (a_1^{\prime\prime}, \dots, a_n^{\prime\prime})$, $\widetilde{{\bf a}} = (\widetilde{a}_1, \dots, \widetilde{a}_n)$, $\pi^{\prime\prime} = (V_1^{\prime\prime}, \dots, V_n^{\prime\prime})$, $\widetilde{\pi} = (\widetilde{V}_1, \dots, \widetilde{V}_n)$, $\boldsymbol{\lambda}^{\prime\prime} = (\lambda_1^{\prime\prime}, \dots, \lambda_n^{\prime\prime})$, and $\widetilde{\boldsymbol{\lambda}} = (\widetilde{\lambda}_1, \dots, \widetilde{\lambda}_n)$.

If $V_{t^{\prime}}^{\prime\prime} = \emptyset$, then $a_{t^{\prime}}^{\prime\prime} = 0$. Since $\varphi(y) \neq t^{\prime}$ implies $\widetilde{a}_{t^{\prime}} = 0$, we have $\widetilde{V}_{t^{\prime}} = \emptyset$, so for all descendants $u$ of $y$ in $T$, the $t^{\prime}$-th entry in $\phi^{[\mathcal{L}, \pi]}_{\alpha(u)}(u)$ must be the empty set. By the construction of a shedding sequence, we get $\omega^{[t^{\prime}]}(u) = \text{$t^{\prime}$-th entry of }\phi_{\alpha(u)}^{[\mathcal{L}, \boldsymbol{\lambda}]}(u) = \widetilde{\lambda}_{t^{\prime}}$ for all descendants $u$ of $y$. In particular, $\omega^{[t^{\prime}]}(y_q) = \omega^{[t^{\prime}]}(y_{\text{right}})$, and we are done.

If instead $V_{t^{\prime}}^{\prime\prime} \neq \emptyset$, then $(\lambda^{\prime\prime}_{t^{\prime}}, t^{\prime})$ is the maximal element of $V^{\prime\prime}_{t^{\prime}}$, and we choose a facet $F$ of $\Delta^{\prime\prime}$ that contains this vertex $(\lambda^{\prime\prime}_{t^{\prime}}, t^{\prime})$. Note that $\Delta^{\prime\prime} = \lk_{\widetilde{\Delta}}((\omega^{[t]}(y),t))$ is $(\widetilde{{\bf a}}-\boldsymbol{\delta}_{t,n})$-balanced, which implies $F \cup \{(\omega^{[t]}(y),t)\}$ is a facet of $\widetilde{\Delta}$, thus the definition of $q$ and the construction of a shedding sequence together yield $|F \cap V_t| \geq q-2$. Let $F^{\prime}$ be the set of the $q-2$ largest elements in $F \cap V_t$, and let $u_i = (\omega^{[t]}(y_i),t)$ for each $i\in [q-1]$. Since $\widetilde{\Delta}$ is color-compressed, $(F\backslash F^{\prime}) \cup \{u_1, \dots, u_{q-1}\}$ must be a facet of $\widetilde{\Delta}$, thus $F\backslash F^{\prime}$ is a facet of $\phi_{\alpha(y_q)}^{[\mathcal{L}, \Delta]}(y_q)$. In particular, $(\lambda^{\prime\prime}_{t^{\prime}}, t^{\prime})$ is contained in $F\backslash F^{\prime}$, so $\phi_{\alpha(y_q)}^{[\mathcal{L}, \Delta]}(y_q)$ is color-compressed implies the $t^{\prime}$-entry of $\phi_{\alpha(y_q)}^{[\mathcal{L}, \pi]}(y_q)$ is a set containing $[\lambda_{t^{\prime}}^{\prime\prime}]_{t^{\prime}}$. Consequently, $\omega^{[t^{\prime}]}(y_q) = \text{$t^{\prime}$-th entry of }\phi_{\alpha(y_q)}^{[\mathcal{L}, \boldsymbol{\lambda}]}(y_q) \geq \lambda_{t^{\prime}}^{\prime\prime} = \omega^{[t^{\prime}]}(y_{\text{right}})$.
\end{proof}

For the rest of this subsection, let $(T, \varphi)$ be an ${\bf a}$-Macaulay tree for some ${\bf a} = (a_1, \dots, a_n) \in \mathbb{N}^n\backslash\{{\bf 0}_n\}$. Let $(r_0, L, Y)$ be the $(\text{root},1,3)$-triple of $T$, denote the left-weight labeling of $(T, \varphi)$ by $\omega$, and let $(T, \varphi|_Y, \nu)$ be the ${\bf a}$-splitting tree induced by $(T, \varphi|_Y)$. For any $x\in Y \cup L$, recall that $\mathcal{D}_T(x)$ is the set of all descendants of $x$ in $T$.

\begin{definition}
Suppose $t\in [n]$ and $x\in Y \cup L$. Let $x_1, \dots, x_{\ell}$ be the right relative sequence of $x$ in $T$, and let $k$ be the smallest integer in $[\ell]$ such that $x_k\in L$ or $\varphi(x_k) \neq t$. We then define the following sets:
\begin{align*}
\psi_{(T, \varphi)}(x) &:= \Big\{ \big(\omega^{[\varphi(v)]}(v), \varphi(v)\big): v\in Y, \text{ both }v \text{ and its right child are ancestors of }x\Big\};\\
\psi_{(T, \varphi)}^{(t)}(x) &:= \Big\{s\in \mathbb{P}: (s,t) \in \psi_{(T, \varphi)}(x)\Big\};\\
\hat{\psi}_{(T, \varphi)}^{(t)}(x) &:= \Big\{\omega^{[t]}(x_i): i\in [k] \Big\};\\
\hat{\xi}_{(T, \varphi)}^{(t)}(x) &:= \text{largest $\big(a_t-\big|\psi_{(T, \varphi)}^{(t)}(x)\big| - k\big)$-subset of $\big[\omega^{[t]}(x_k)-1\big]$ w.r.t. colex order};\\
\xi_{(T, \varphi)}^{(t)}(x) &:= \psi_{(T, \varphi)}^{(t)}(x) \cup \hat{\psi}_{(T, \varphi)}^{(t)}(x) \cup \hat{\xi}_{(T, \varphi)}^{(t)}(x).
\end{align*}
Observe that the sets $\psi_{(T, \varphi)}^{(t)}(x)$, $\hat{\psi}_{(T, \varphi)}^{(t)}(x)$, $\hat{\xi}_{(T, \varphi)}^{(t)}(x)$ are pairwise disjoint, so $\xi_{(T, \varphi)}^{(t)}(x)$ is an $a_t$-set, and we call this $a_t$-set $\xi_{(T, \varphi)}^{(t)}(x)$ the {\it $t$-signature} of $x$ in $(T, \varphi)$. In particular, if $x$ is a $t$-leading vertex, then $\psi_{(T, \varphi)}^{(t)}(x) = \emptyset$.
\end{definition}

\begin{definition}
Let $t\in \{0,1,\dots, n-1\}$, and let $u_1, \dots, u_m$ (for some $m\in \mathbb{P}$) be all the $t$-leading vertices in $\mathcal{L}_{(T, \varphi)}(t)$ arranged in depth-first order. For every $u\in \mathcal{L}_{(T, \varphi)}(t)$ such that the $(t+1)$-th entry of ${\bf a} - \nu(u)$ is positive, let $j$ be the unique integer in $[m]$ satisfying $u = u_j$, and define $\zeta^{t+1}_t(u) := u_{j-1}$. In particular, since $\nu(u_1) = {\bf a}$, we necessarily have $j>1$, hence $\zeta^{t+1}_t(u)$ is well-defined.
\end{definition}

Let $v$ be the unique common ancestor of $u_{j-1}$ and $u_j$ in $Y\cup L$ such that neither child of $v$ is a common ancestor of $u_{j-1}$ and $u_j$. The path from $r_0$ to the right-most relative of $v_{\text{left}}$ and the path from $r_0$ to the left-most relative of $v_{\text{right}}$ share the common vertex $v$, and these two paths must each contain some $t$-leading vertex. The definition of $t$-leading vertices implies $\varphi(v) > t$, so by the definition of the depth-first order, $u_{j-1}$ is contained in the right relative sequence of $v_{\text{left}}$ in $T$, and $u_j$ is contained in the left relative sequence of $v_{\text{right}}$ in $T$. Furthermore, since the $(t+1)$-th entry of ${\bf a} - \nu(u_j)$ is positive, the definition of a Macaulay tree forces $\varphi(v) = t+1$. Consequently, if $(T, \varphi)$ is compressed-like, then Condition \ref{defn:compressed-likeCondii} in Definition \ref{defn:compressed-like} yields $\xi_{(T, \varphi)}^{(t)}(u_j) \leq_{c\ell} \xi_{(T, \varphi)}^{(t)}(u_{j-1})$, or equivalently, $\xi_{(T, \varphi)}^{(t)}(u) \leq_{c\ell} \xi_{(T, \varphi)}^{(t)}(\zeta^{t+1}_t(u))$.

\begin{definition}
Let $i,j\in \mathbb{N}$ satisfy $j+1<i\leq n$. For every $u\in \mathcal{L}_{(T, \varphi)}(j+1)$ such that the $i$-th entry of ${\bf a} - \nu(u)$ is positive, assume $\zeta^i_{j+1}(u)$ has already been defined, and suppose $\xi_{(T, \varphi)}^{(j+1)}(u) \leq_{c\ell} \xi_{(T, \varphi)}^{(j+1)}(\zeta^i_{j+1}(u))$. Then for any $x\in \mathcal{L}_{(T, \varphi)}(j)$, let $\widetilde{x}$ be the unique ancestor of $x$ that is a $(j+1)$-leading vertex, set $\widetilde{y} := \zeta^i_{j+1}(\widetilde{x})$, and define $\zeta^i_j(x)$ as the smallest vertex in $U := \big\{z\in \mathcal{L}_{(T, \varphi)}(j) \cap \mathcal{D}_T(\widetilde{y}): \xi_{(T, \varphi)}^{(j+1)}(z) \geq_{c\ell} \xi_{(T, \varphi)}^{(j+1)}(x)\big\}$ with respect to the depth-first order. Note that if $y$ is the unique $j$-leading vertex in the right relative sequence of $\widetilde{y}$, then the definition of a $(j+1)$-signature yields $\xi_{(T, \varphi)}^{(j+1)}(x) \leq_{c\ell} \xi_{(T, \varphi)}^{(j+1)}(\widetilde{x}) \leq_{c\ell} \xi_{(T, \varphi)}^{(j+1)}(\widetilde{y}) = \xi_{(T, \varphi)}^{(j+1)}(y)$, thus $y \in U$ (i.e. $U$ is non-empty), so $\zeta^i_j(x)$ is well-defined.
\end{definition}

\begin{remark}\label{remark:algorithm-zeta}
Let $i,j\in \mathbb{N}$ satisfy $j<i\leq n$, and let $x\in \mathcal{L}_{(T, \varphi)}(j)$ such that the $i$-th entry of ${\bf a} - \nu(x)$ is positive. If $\zeta^i_j(x)$ is well-defined, then $\zeta^i_j(x)$ can be determined from $x$ via the following algorithm.
\begin{center}
\begin{algorithmic}
\small
\State $(x_0, x_1, \dots, x_{\ell}) \gets$ path from $r_0$ to $x$ (which has length $\ell$).
\State $q \gets$ largest integer in $[\ell-1]$ such that $\varphi(x_q) = i$ and $x_{q+1}$ is the right child of $x_q$.
\State $z \gets$ left child of $x_q$.
\State $(y_1, \dots, y_{\ell^{\prime}}) \gets$ right relative sequence of $z$ (which has $\ell^{\prime}$ terms).
\State $k \gets$ smallest integer in $[\ell^{\prime}]$ such that $y_k \in L$ or $\varphi(y_k) \neq i$.
\If {$y_k \in L$}
    \State $\zeta^i_j(x) \gets y_k$.
\Else
    \State $s \gets \varphi(y_k)$.
    \State $y\gets y_k$.
    \While {$s>t$}
        \State $U \gets \Big\{u\in \mathcal{L}_{(T, \varphi)}(s-1) \cap \mathcal{D}_T(y): \xi_{(T, \varphi)}^{(s)}(u) \geq_{c\ell} \xi_{(T, \varphi)}^{(s)}(x)\Big\}$.
        \State $y \gets$ smallest vertex in $U$ with respect to the depth-first order.
        \If {$y\in L$}
            \State $s \gets t$.
        \Else
            \State $s \gets \varphi(y)$.
        \EndIf
    \EndWhile
    \State $\zeta^i_j(x) \gets y$.
\EndIf
\end{algorithmic}
\end{center}
Notice that $\zeta^i_j(x)$ is well-defined only if the set $U$ in each iteration of the while loop is non-empty.
\end{remark}

\begin{proposition}\label{prop:compressed-like-properties}
Let $(T, \varphi)$ be compressed-like, let $s,t\in \mathbb{N}$ satisfy $t<s\leq n$, and let $x\in \mathcal{L}_{(T, \varphi)}(t)$ such that the $s$-th entry of ${\bf a} - \nu(x)$ is positive. If $\zeta^s_t(x)$ is well-defined, then we have the following:
\begin{enum*}
\item\label{cond:compressed-like-zeta-signature-q1} $\psi^{(s)}_{T, \varphi}(x) = \{q_1, \dots, q_m\}$ for some $m\in [a_s]$ such that $a_s -m + 1 < q_1 < \dots < q_m$.
\item\label{cond:compressed-like-zeta-explicit-signature} $\xi^{(s)}_{(T, \varphi)}(u) = \big\{q_1-i: i\in [a_s-m+1]\big\} \cup \{q_2, \dots, q_m\}$ for every $u\in \mathcal{D}_T\big(\zeta^s_t(x)\big)$.
\item\label{cond:compressed-like-zeta-signature>s} If $i$ is an integer satisfying $s<i\leq n$, then $\xi^{(i)}_{(T, \varphi)}(u) = \xi^{(i)}_{(T, \varphi)}(x)$ for every $u\in \mathcal{D}_T\big(\zeta^s_t(x)\big)$.
\end{enum*}
\end{proposition}

\begin{proof}
First of all, since the $s$-th entry of ${\bf a} - \nu(x)$ is positive, $\psi^{(s)}_{(T, \varphi)}(x)$ must be non-empty, so we can write $\psi^{(s)}_{T, \varphi}(x)$ as the set $\{q_1, \dots, q_m\}$ for some $m\in [a_s]$, such that $q_1 < \dots < q_m$. For each $i\in [m]$, let $y_i \in Y$ be the unique ancestor of $x$ such that $(\omega^{[\varphi(y_i)]}(y_i), \varphi(y_i)) = (q_i,s)$. Also, let $z$ be the left child of $y_1$. The subset consisting of the $m$ largest elements in $\xi^{(s)}_{(T, \varphi)}(x) \in \binom{\mathbb{P}}{a_s}$ is precisely $\psi^{(s)}_{(T, \varphi)}(x)$, hence $q_1 \geq a_s-m+1$. If $q_1 = a_s-m+1$, then the definition of the left-weight labeling $\omega$ yields $\omega^{[s]}(z) = a_s-m$, while the construction of the ${\bf a}$-splitting tree $(T, \varphi|_Y, \nu)$ yields $\nu^{[s]}(z) = a_s-m+1$. This means the left-most relative of $z$, which we denote by $z_{\text{left-most}}$, satisfies $\varphi(z_{\text{left-most}}) \not\geq \nu(z_{\text{left-most}})$, hence contradicting the definition of a Macaulay tree. Consequently, $q_1 > a_s-m+1$, i.e. statement \ref{cond:compressed-like-zeta-signature-q1} is true.

Let $z_1, \dots, z_{\ell}$ be the right relative sequence of $z$, and let $k$ be the smallest integer in $[\ell]$ such that $z_k\in L$ or $\varphi(z_k) \neq s$. Since $(T, \varphi)$ is compressed-like, Condition \ref{defn:compressed-likeCondi} in Definition \ref{defn:compressed-like} yields
\begin{equation}
\psi^{(s)}_{(T, \varphi)}\big(\zeta^s_t(x)\big) = \{q_2, \dots, q_m\} \cup \Big\{\omega^{[s]}(z_i): i\in [k-1]\Big\} = \{q_2, \dots, q_m\} \cup \Big\{q_1-i: i\in [k-1]\Big\}.
\end{equation}
Next, let $z^{\prime}$ be the left-most relative of $z_k$. The vertex $z^{\prime}$ and all vertices in $\mathcal{D}_T(\zeta^s_t(x)) \cap L$ are contained $\mathcal{D}_T(z_k) \cap L$, so it follows from the definition of a Macaulay tree that $\varphi^{[s]}(u) = \varphi^{[s]}(z^{\prime}) = \omega^{[s]}(z_k) = q_1-k$ for every $u\in \mathcal{D}_T(\zeta^s_t(x)) \cap L$. Since $t<s$ and $\zeta^s_t(x)$ is by construction a $t$-leading vertex of $(T, \varphi)$, we get $\xi^{(s)}_{(T, \varphi)}(u) = \big\{q_1-i: i\in [a_s-m+1]\big\} \cup \{q_2, \dots, q_m\}$ for every $u\in \mathcal{D}_T\big(\zeta^s_t(x)\big)$, thus proving statement \ref{cond:compressed-like-zeta-explicit-signature}.

Finally, if $s<i\leq n$ ($i\in \mathbb{P}$), then since both $\zeta^s_t(x)$ and $x$ are in $\mathcal{D}_T(y_1)$, it follows from the definition of a Macaulay tree and $\varphi(y_1) = s < i$ that $\psi^{(i)}_{(T, \varphi)}(u) = \psi^{(i)}_{(T, \varphi)}(x)$ and $\omega^{[i]}(u) = \omega^{[i]}(x)$ for all $u\in \mathcal{D}_T\big(\zeta^s_t(x)\big)$, which proves \ref{cond:compressed-like-zeta-signature>s}.
\end{proof}

\begin{definition}
Define $(T, \varphi)$ to be always {\it $n$-compatible}. Suppose $n>1$, $t\in [n-1]$, and $(T, \varphi)$ is compressed-like. Then we say $(T, \varphi)$ is {\it $t$-compatible} if (i): $(T, \varphi)$ is $s$-compatible for all $s\in [n]\backslash [t]$; and (ii): $\xi_{(T, \varphi)}^{(t)}(x) \leq_{c\ell} \xi_{(T, \varphi)}^{(t)}\big(\zeta^i_t(x)\big)$ for every $x\in \mathcal{L}_{(T, \varphi)}(t)$, $i\in [n]\backslash [t]$ such that the $i$-th entry of ${\bf a} - \nu(x)$ is positive. In particular, condition (i) ensures $\zeta^i_t(x)$ is well-defined. If $(T, \varphi)$ is $1$-compatible, which is identical to $(T, \varphi)$ being $j$-compatible for all $j\in [n]$, then we say $(T, \varphi)$ is {\it compatible}. Equivalently, $(T, \varphi)$ is {\it compatible} if $\xi_{(T, \varphi)}^{(t)}(x) \leq_{c\ell} \xi_{(T, \varphi)}^{(t)}\big(\zeta^i_t(x)\big)$ for every $i,t\in [n]$ satisfying $t<i$, and every $x\in \mathcal{L}_{(T, \varphi)}(t)$ such that the $i$-th entry of ${\bf a} - \nu(x)$ is positive.
\end{definition}

\begin{lemma}\label{lemma:compressed-like-t+1-compatible}
Let $0\leq t< s\leq n$ be integers, and let $x\in \mathcal{L}_{(T, \varphi)}(t)$ such that ${\bf a} - \nu(x)$ has a positive $s$-th entry. Suppose $(T, \varphi)$ is compressed-like and $(t+1)$-compatible. Then $\zeta^s_t(x)$ is well-defined, and for each integer $k$ such that $t<k<s$,
\begin{enum*}
\item\label{cond:psi-xi} $\psi^{(k)}_{(T, \varphi)}(u) \subseteq \xi^{(k)}_{(T, \varphi)}(x)$ for all $u\in \mathcal{D}_T\big(\zeta^s_t(x)\big)$; and
\item\label{cond:xi-xi} $\xi^{(k)}_{(T, \varphi)}(u) \geq_{c\ell} \xi^{(k)}_{(T, \varphi)}(x)$ for all $u\in \mathcal{D}_T\big(\zeta^s_t(x)\big)$.
\end{enum*}
\end{lemma}

\begin{proof}
First of all, $\zeta^s_t(x)$ is well-defined by the definition of $(t+1)$-compatible. Let $x_0, x_1, \dots, x_{\ell}$ be the path from $r_0$ to $x$ in $T$, and let $q$ be the largest integer in $[\ell-1]$ such that $\varphi(x_q) = s$ and $x_{q+1}$ is the right child of $x_q$. For each integer $i$ satisfying $t\leq i\leq n$, let $q_i$ be the unique integer in $[\ell]$ such that $x_{q_i}$ is an $i$-leading vertex. Also, denote the left child of $x_q$ by $z$, let $z_1, \dots, z_m$ be the path in $T$ from $z$ to $\zeta^s_t(x)$, and for each integer $i$ satisfying $t\leq i < s$, let $j_i$ be the unique integer in $[m]$ such that $z_{j_i}$ is an $i$-leading vertex. For example, $q_t = \ell$ and $j_t = m$, i.e. $x_{q_t} = x$ and $z_{j_t} = \zeta^s_t(x)$.

Suppose $k \in \mathbb{P}$ satisfies $t<k<s$. Then $(T, \varphi)$ is $(k+1)$-compatible implies $\zeta^s_k(x_{q_k})$ is well-defined, and Remark \ref{remark:algorithm-zeta} yields $\zeta^s_k(x_{q_k}) = z_{j_k}$. Note also that $\xi^{(k)}_{(T, \varphi)}(z_{j_k}) = \xi^{(k)}_{(T, \varphi)}\big(\zeta^s_k(x_{q_k})\big) \geq_{c\ell} \xi^{(k)}_{(T, \varphi)}(x_{q_k}) \geq_{c\ell} \xi^{(k)}_{(T, \varphi)}(x)$, since $(T, \varphi)$ is $k$-compatible. Now, define vertex $z^{\prime}$ algorithmically as follows: Starting with $z^{\prime} = z_{j_k}$, as long as $z^{\prime} \not\in \mathcal{L}_{(T, \varphi)}(k-1)$, replace $z^{\prime}$ with $z^{\prime}_{\text{right}}$ if $\omega^{[k]}(z^{\prime}) \in \xi^{(k)}_{(T, \varphi)}(x)$, and replace $z^{\prime}$ with $z^{\prime}_{\text{left}}$ otherwise. Repeat process until $z^{\prime} \in \mathcal{L}_{(T, \varphi)}(k-1)$; the resulting $z^{\prime}$ is the vertex we want. Using Remark \ref{remark:algorithm-zeta} and condition \ref{defn:compressed-likeCondi} of Definition \ref{defn:compressed-like}, we then get $z^{\prime} = z_{j_{k-1}}$, thus $\psi^{(k)}_{(T, \varphi)}(z_{j_{k-1}}) = \psi^{(k)}_{(T, \varphi)}(z^{\prime}) \subseteq \xi^{(k)}_{(T, \varphi)}(x)$. Since $z_{j_{k-1}}$ is a $(k-1)$-leading vertex, we have $\psi^{(k)}_{(T, \varphi)}(u) \subseteq \xi^{(k)}_{(T, \varphi)}(x)$ for all $u\in \mathcal{D}_T(z_{j_{k-1}})$, so \ref{cond:psi-xi} follows from the fact that $\zeta^s_t(x) \in \mathcal{D}_T(z_{j_{k-1}})$.

Next, we prove \ref{cond:xi-xi}. If $\big|\psi^{(k)}_{(T, \varphi)}\big(\zeta^s_t(x)\big)\big| = a_k$, then it follows from \ref{cond:psi-xi} that $\xi^{(k)}_{(T, \varphi)}(u) = \psi^{(k)}_{(T, \varphi)}(u) = \xi^{(k)}_{(T, \varphi)}(x)$ for all $u\in \mathcal{D}_T\big(\zeta^s_t(x)\big)$ and we are done. Assume $\big|\psi^{(k)}_{(T, \varphi)}\big(\zeta^s_t(x)\big)\big| < a_k$, and define $p:= \max\big(\xi^{(k)}_{(T, \varphi)}(x) \backslash \psi^{(k)}_{(T, \varphi)}\big(\zeta^s_t(x)\big)\big)$. The construction of $\zeta^s_t(x)$ yields $\xi^{(k)}_{(T, \varphi)}(z_{j_{k-1}}) \geq_{c\ell} \xi^{(k)}_{(T, \varphi)}(x)$, so in particular, $\omega^{[k]}(z_{j_{k-1}}) \geq p$. Since $z_{j_{k-1}}$ is a $(k-1)$-leading vertex, the definition of a Macaulay tree tells us that the left-most relative of $z_{j_{k-1}}$, which we denote by $z^{\prime\prime}$, satisfies $\omega^{[k]}(z^{\prime\prime}) = \omega^{[k]}(z_{j_{k-1}}) \geq p$. Consequently, $\omega^{[k]}(u) \geq p$ for all $u\in \mathcal{D}_T(z_{j_{k-1}})$, therefore the definition of a $k$-signature, together with statement \ref{cond:psi-xi}, yields $\xi^{(k)}_{(T, \varphi)}(u) \geq_{c\ell} \xi^{(k)}_{(T, \varphi)}(x)$ for all $u\in \mathcal{D}_T\big(\zeta^s_t(x)\big) \subseteq \mathcal{D}_T(z_{j_{k-1}})$.
\end{proof}

\begin{theorem}\label{thm:color-compressed=>compatible}
Let ${\bf a}\in \mathbb{P}^n$. If $(\Delta, \pi)$ is a pure color-compressed ${\bf a}$-balanced complex, then the ${\bf a}$-Macaulay tree induced by the shedding tree of $(\Delta, \pi)$ is compatible.
\end{theorem}

\begin{proof}
For convenience, identify each vertex $v_{i,j}$ of $\Delta$ with the pair $(j,i)$, so that each $V_i$ in $\pi = (V_1, \dots, V_n)$ is identified with $[\lambda_i]_i$. Let $(T_0, \phi_0), (T_1, \phi_1), \dots, (T_k, \phi_k)$ be a shedding sequence of $(\Delta, \pi)$. For each $i\in [k]$, let $(r_0, L_i, Y_i)$ denote the $(\text{root}, 1, 3)$-triple of $T_i$, let $(T_i, \varphi_i)$ be the Macaulay tree induced by $(T_i, \phi_i)$, and let $\omega_i$ be the left-weight labeling of $(T_i, \varphi_i)$. Let $(T_k, \varphi_k|_{Y_k}, \nu)$ be the ${\bf a}$-splitting tree induced by $(T_k, \varphi_k|_{Y_k})$. Also, for every $y\in Y_k$, let $\alpha(y)$ be the largest integer $<k$ such that $y\in L_{\alpha(y)}$, while for every $u\in L_k$, set $\alpha(u) = k$. Note that for $(T_k, \varphi_k)$ to be compatible, there is an implicit assumption of $(T_k, \varphi_k)$ being compressed-like, which we can assume by Theorem \ref{thm:color-compressed=>compressed-like}.

Suppose $(T_k, \varphi_k)$ is not compatible, and let $t\in [n-1]$ be maximal such that $(T_k, \varphi_k)$ is not $t$-compatible. Choose $x\in \mathcal{L}_{(T_k, \varphi_k)}(t)$, $s\in [n]\backslash [t]$ such that ${\bf a} - \nu(x)$ has a positive $s$-th entry, and $\xi_{(T_k, \varphi_k)}^{(t)}(x) \not\leq_{c\ell} \xi_{(T_k, \varphi_k)}^{(t)}(\zeta^s_t(x))$. In particular, $\zeta^s_t(x)$ is well-defined since $(T_k, \varphi_k)$ is $(t+1)$-compatible. Next, define $F^{\prime} := \bigcup_{j \in [n]\backslash [t]} \big\{(i, j) : i \in \xi_{(T_k, \varphi_k)}^{(j)}(x)\big\}$. By Proposition \ref{prop:left-weight-shedding-sequence}\ref{condition-orderly}, $\omega_k^{[j]}(u) = \omega_k^{[j]}(x)$ for all $j \in [n]\backslash [t]$ and $u\in \mathcal{D}_{T_k}(x)$, hence $F \cup F^{\prime} \in \Delta$ for all $u \in \mathcal{D}_{T_k}(x)$ and $F\in \phi_{\alpha(u)}^{[\mathcal{L}, \Delta]}(u)$ satisfying $F \subseteq \bigcup_{i\in [t]} V_i$. Note that $F_t := \big\{(i,t): i\in \xi_{(T_k, \varphi_k)}^{(t)}(x)\big\} \subseteq V_t$ is a face of $\phi_{\alpha(x)}^{[\mathcal{L}, \Delta]}(x)$, so $F^{\prime} \cup F_t \in \Delta$.

Recall the definition of the operation $C_j$ (for any $j\in [n]$) on colored multicomplexes defined in Section \ref{sec:ColorCompressionColoredComplexes}, and define $C_j$ on colored complexes analogously. Proposition \ref{prop:compressed-like-properties}\ref{cond:compressed-like-zeta-explicit-signature} yields $\xi_{(T_k, \varphi_k)}^{(s)}(\zeta^s_t(x)) <_{c\ell} \xi_{(T_k, \varphi_k)}^{(s)}(x)$, so since $(\Delta, \pi)$ is color-compressed implies $C_s(\Delta) = \Delta$, it then follows from $F^{\prime} \cup F_t \in \Delta$ that $\widetilde{F} \cup F_t \in \Delta$, where
\begin{equation*}
\widetilde{F} := \Big(F^{\prime}\backslash\Big\{(i,s): i\in \xi_{(T_k, \varphi_k)}^{(s)}(x)\Big\}\Big) \cup \Big\{(i,s): i\in \xi_{(T_k, \varphi_k)}^{(s)}\big(\zeta^s_t(x)\big)\Big\}.
\end{equation*}

Now, Lemma \ref{lemma:compressed-like-t+1-compatible}\ref{cond:psi-xi} says $\psi^{(i)}_{(T_k, \varphi_k)}\big(\zeta^s_t(x)\big) \subseteq \xi^{(i)}_{(T_k, \varphi_k)}(x)$ for all integers $t<i<s$ (if any), while Proposition \ref{prop:compressed-like-properties}\ref{cond:compressed-like-zeta-signature>s} says $\psi^{(i)}_{(T_k, \varphi_k)}\big(\zeta^s_t(x)\big) = \psi^{(i)}_{(T_k, \varphi_k)}(x)$ for all integers $s<i\leq n$ (if any). Consequently, $\psi_{(T_k, \varphi_k)}\big(\zeta^s_t(x)\big) \subseteq \widetilde{F}$, hence it follows from $\widetilde{F} \cup F_t \in \Delta$ that $\big(\widetilde{F}\backslash \psi_{(T_k, \varphi_k)}\big(\zeta^s_t(x)\big)\big) \cup F_t \in \phi_{\alpha(\zeta^s_t(x))}^{[\mathcal{L}, \Delta]}(\zeta^s_t(x))$. Looking at vertices in $V_t$, this means $\xi^{(t)}_{(T_k, \varphi_k)}(\zeta^s_t(x)) \geq_{c\ell} \xi^{(t)}_{(T_k, \varphi_k)}(x)$, which contradicts our original assumption that $\xi^{(t)}_{(T_k, \varphi_k)}(\zeta^s_t(x)) \not\geq_{c\ell} \xi^{(t)}_{(T_k, \varphi_k)}(x)$.
\end{proof}

\begin{lemma}\label{lemma:face-containment-criterion}
Let $(T, \varphi)$ have weight $(s_1, \dots, s_n) \in \mathbb{P}^n$. For each $t\in [n]$, let $G_t \in \binom{[s_t]}{a_t}$ and define $F_t := \{(i,t): i\in G_t\}$. If there exists some $u\in L$ such that $\psi^{(t)}_{(T, \varphi)}(u) \subseteq G_t$ for all $t\in [n]$, then $F_1 \cup \dots \cup F_n$ is contained in $\Psi_{(T, \varphi)}(u)$ if and only if $\xi^{(t)}_{(T, \varphi)}(u) \geq_{c\ell} G_t$ for all $t\in [n]$.
\end{lemma}

\begin{proof}
Suppose there exists some $u\in L$ as described. Definition \ref{defn:Psi-complex} yields $\Psi_{(T, \varphi)}(u) = \big\langle \begin{smallmatrix} \varphi(u)\\ \nu(u) \end{smallmatrix} \big\rangle * \langle \{\psi_{(T, \varphi)}(u)\} \rangle$, so $F_1 \cup \dots \cup F_n \in \Psi_{(T, \varphi)}(u)$ if and only if $F_t\backslash \{(i,t): i\in \psi^{(t)}_{(T, \varphi)}(u)\} \in \big\langle \binom{[\varphi^{[t]}(u)]_t}{\nu^{[t]}(u)}\big\rangle$ for all $t\in [n]$. Note that $|\psi^{(t)}_{(T, \varphi)}(u)| = a_t - \nu^{[t]}(u)$, thus it suffices to show that $\max\big(G_t\backslash \psi^{(t)}_{(T, \varphi)}(u)\big) \leq \varphi^{[t]}(u) = \omega^{[t]}(u)$ for all $t\in [n]$ satisfying $|\psi^{(t)}_{(T, \varphi)}(u)| < a_t$ (since the case $|\psi^{(t)}_{(T, \varphi)}(u)| = a_t$ is trivial), which is equivalent to $G_t \leq_{c\ell} \xi^{(t)}_{(T, \varphi)}(u)$.
\end{proof}

\begin{theorem}\label{thm:compressed-like-compatible=>color-compressed}
Let $(T, \varphi)$ be a compressed-like, compatible ${\bf a}$-Macaulay tree. Then $(\Delta_{(T, \varphi)}, \pi_{(T, \varphi)})$ is a pure color-compressed ${\bf a}$-balanced complex.
\end{theorem}

\begin{proof}
Denote the $(\text{root}, 1, 3)$-triple of $T$ by $(r_0, L, Y)$, let $(T, \varphi|_Y, \nu)$ be the ${\bf a}$-splitting tree induced by $(T, \varphi|_Y)$, and let $\omega$ be the left-weight labeling of $(T_, \varphi)$. We know from Proposition \ref{prop:Macaulay-tree=>shedding-tree} that $(\Delta_{(T, \varphi)}, \pi_{(T, \varphi)})$ is a pure ${\bf a}$-balanced complex, so we are left to show that $(\Delta_{(T, \varphi)}, \pi_{(T, \varphi)})$ is color-compressed.

Let $F$ be an arbitrary facet of $\Delta_{(T, \varphi)}$. By construction, there exists a unique terminal vertex in $L$, which we denote by $u_F$, such that $F\in \Psi_{(T, \varphi)}(u_F)$. Recall that $[s]_i := \{(1,i), (2,i), \dots, (s,i)\}$ for any $s\in \mathbb{P}$, $i\in [n]$, and note that $\pi_{(T, \varphi)} = ([s_1]_1, \dots, [s_n]_n)$, where $(s_1, \dots, s_n) \in \mathbb{P}^n$ is the weight of $(T, \varphi)$. For each $i\in [n]$, define $F_i := F \cap [s_i]_i \in \binom{[s_i]_i}{a_i}$, and define $\hat{F}_i := \{s\in \mathbb{P}: (s,i) \in F_i\} \subseteq [s_i]$. If $\hat{F}_i \neq [a_i]$, then let $\hat{F}_i^{\prime} < \hat{F}_i$ denote the immediate predecessor of $\hat{F}_i$ in $\binom{\mathbb{P}}{a_t}$ with respect to the colex order, while if $\hat{F}_i = [a_i]$, then let $\hat{F}_i^{\prime} := [a_i]$. In either case, define $F_i^{\prime} := \{(s,i) : s\in \hat{F}_i^{\prime}\}$ and $F^{\prime}_{\langle i \rangle} := (F\backslash F_i) \cup F_i^{\prime}$. Let $\mathcal{S}_{(T, \varphi)} := \big\{F\in \Delta_{(T, \varphi)}: \text{$F$ is a facet of $\Delta_{(T, \varphi)}$, and $\exists \ i\in [n]$ such that $F^{\prime}_{\langle i\rangle} \not\in \Delta_{(T, \varphi)}$}\big\}$. By the definition of color compression, $(\Delta_{(T, \varphi)}, \pi_{(T, \varphi)})$ is color-compressed if and only if $\mathcal{S}_{(T, \varphi)} = \emptyset$.

Suppose $\mathcal{S}_{(T, \varphi)} \neq \emptyset$. Choose any $F \in \mathcal{S}_{(T, \varphi)}$ such that $u_F$ is minimal with respect to the depth-first order, and let $t\in [n]$ satisfy $F^{\prime}_{\langle t\rangle} \not\in \Delta_{(T, \varphi)}$. If $\psi_{(T, \varphi)}^{(t)}(u_F) = \emptyset$, then $\hat{F}_t \subseteq [\varphi^{[t]}(u_F)]$ by definition, which implies $\hat{F}^{\prime}_t \subseteq [\varphi^{[t]}(u_F)]$, thus $F^{\prime}_{\langle t \rangle} \in \Psi_{(T, \varphi)}(u_F)$, which is a contradiction. Consequently, we can assume $\psi_{(T, \varphi)}^{(t)}(u_F) \neq \emptyset$, so write $\psi_{(T, \varphi)}^{(t)}(u_F) = \{q_1, \dots, q_m\}$ (where $1 \leq m \leq a_t$) such that $q_1 < \dots < q_m$. For each $i\in [m]$, let $y_i\in Y$ be the unique ancestor of $u_F$ such that $(\omega^{[\varphi(y_i)]}(y_i), \varphi(y_i)) = (q_i, t)$. If $m<a_t$ and $\hat{F}_t\backslash \psi_{(T, \varphi)}^{(t)}(u_F) \neq [a_t-m]$, then $\psi_{(T, \varphi)}^{(t)}(u_F) \subseteq \hat{F}_t^{\prime}$ by the definition of the colex order, which means $F^{\prime}_{\langle t\rangle} \in \Psi_{(T, \varphi)}(u_F)$, and again we have a contradiction, hence we can further assume that $\hat{F}_t = \{q_1, \dots, q_m\} \cup [a_t-m]$.

Note that $\psi^{(t)}_{(T, \varphi)}(u_F) \neq \emptyset$ implies the $t$-th entry of ${\bf a} - \nu(u_F)$ is positive, so $\zeta^t_0(u_F)$ is well-defined, and Proposition \ref{prop:compressed-like-properties}\ref{cond:compressed-like-zeta-signature-q1} tells us $q_1 > a_t-m+1$. The definition of colex order yields $\hat{F}_t^{\prime} = \{q_1-i: i\in [a_t-m+1]\} \cup \{q_2, \dots, q_m\}$, thus Proposition \ref{prop:compressed-like-properties}\ref{cond:compressed-like-zeta-explicit-signature} gives $\xi^{(t)}_{(T, \varphi)}\big(\zeta^t_0(u_F)\big) = \hat{F}_t^{\prime}$. Also, Proposition \ref{prop:compressed-like-properties}\ref{cond:compressed-like-zeta-signature>s} says $\xi^{(i)}_{(T, \varphi)}\big(\zeta^t_0(u_F)\big) = \xi^{(i)}_{(T, \varphi)}(u_F)$ for all integers $t<i\leq n$ (if any), while Lemma \ref{lemma:compressed-like-t+1-compatible} says $\psi^{(i)}_{(T, \varphi)}\big(\zeta^t_0(u_F)\big) \subseteq \xi^{(i)}_{(T, \varphi)}(u_F)$ and $\xi^{(i)}_{(T, \varphi)}\big(\zeta^t_0(u_F)\big) \geq_{c\ell} \xi^{(i)}_{(T, \varphi)}(u_F)$ for all $i\in [t-1]$. Now, define $G := F_t^{\prime} \cup \Big(\bigcup_{i\in [n]\backslash\{t\}} \big\{(s,i): s\in \xi^{(i)}_{(T, \varphi)}\big(\zeta^t_0(u_F)\big)\big\}\Big)$, and observe that $\big(\widehat{F_{\langle t\rangle}^{\prime}}\big)_i \leq_{c\ell} \hat{G}_i$ for all $i\in [n]$. However, Lemma \ref{lemma:face-containment-criterion} yields $G\in \Psi_{(T, \varphi)}\big(\zeta^t_0(u_F)\big)$, and $\zeta^t_0(u_F)$ is strictly smaller than $u_F$ in the depth-first order, so the minimality of $u_F$ implies $F_{\langle t\rangle}^{\prime} \in \Delta_{(T, \varphi)}$ (although not necessarily in $\Psi_{(T, \varphi)}\big(\zeta^t_0(u_F)\big)$), which is a contradiction.
\end{proof}

\begin{theorem}\label{thm:numerical-char-color-compressed-complex}
Let $(T, \varphi)$ be an ${\bf a}$-Macaulay tree of $N$ for some ${\bf a} \in \mathbb{P}^n$, $N\in \mathbb{P}$. Then $(T, \varphi)$ is the ${\bf a}$-Macaulay tree induced by the shedding tree of some pure color-compressed ${\bf a}$-balanced complex with $N$ facets if and only if $(T, \varphi)$ is condensed, compressed-like, and compatible.
\end{theorem}

\begin{proof}
If there exists a pure color-compressed ${\bf a}$-balanced complex $(\Delta, \pi)$ with $N$ facets, such that $(T, \varphi)$ is the ${\bf a}$-Macaulay tree induced by the shedding tree of $(\Delta, \pi)$, then Theorem \ref{thm:shedding-tree=>Macaulay-tree} says $(T, \varphi)$ is condensed, Theorem \ref{thm:color-compressed=>compressed-like} yields $(T, \varphi)$ is compressed-like, and Theorem \ref{thm:color-compressed=>compatible} tells us $(T, \varphi)$ is compatible. Conversely, if instead $(T, \varphi)$ is condensed, compressed-like, and compatible, then Proposition \ref{prop:Macaulay-tree=>shedding-tree} gives the pure ${\bf a}$-balanced complex $(\Delta_{(T, \varphi)}, \pi_{(T, \varphi)})$ with $N$ facets, Theorem \ref{thm:compressed-like-compatible=>color-compressed} tells us $(\Delta_{(T, \varphi)}, \pi_{(T, \varphi)})$ is color-compressed, while by Remark \ref{remark:distinct-Mac-trees-yield-same-complex}, the ${\bf a}$-Macaulay tree induced by the shedding tree of $(\Delta_{(T, \varphi)}, \pi_{(T, \varphi)})$ is precisely $(T, \varphi)$.
\end{proof}

\subsection{Generalized ${\bf a}$-Macaulay Representations}\label{subsec:GenMacReps}
\begin{definition} Let $N\in \mathbb{P}$, ${\bf a} \in \mathbb{N}^n\backslash \{{\bf 0}_n\}$. A {\it generalized ${\bf a}$-Macaulay representation of $N$} is a condensed compressed-like compatible ${\bf a}$-Macaulay tree of $N$. Explicitly, it is a pair $(T, \varphi)$ such that
\begin{enum*}
\item $T$ is a trivalent planted binary tree with $(\text{root},1,3)$-triple $(r_0, L, Y)$.
\item\label{MacRepDefnMacTree} $\varphi$ is a vertex labeling of $T$ satisfying all of the following conditions.
\begin{enum*}
\item $\varphi(r_0) = {\bf a}$, $\varphi(Y) \subseteq [n]$ and $\varphi(L) \subseteq \mathbb{P}^n$.
\item If $y, y^{\prime} \in Y$ satisfies $y^{\prime} \in \mathcal{D}_T(y)$, i.e. $y^{\prime}$ is a descendant of $y$, then $\varphi(y) \geq \varphi(y^{\prime})$.
\item If $y\in Y$ satisfies $\varphi(y) < n$, then $\varphi^{[t]}(u) = \varphi^{[t]}(u^{\prime})$ for all $u, u^{\prime} \in L \cap \mathcal{D}_T(y)$ and all $t\in [n]\backslash [\varphi(y)]$.
\item The left-weight labeling $\omega$ of $(T, \varphi)$ satisfies $\omega(y_{\text{left}}) \geq \omega(y_{\text{right}})$ for every $y\in Y$.
\item The ${\bf a}$-splitting tree $(T, \varphi|_Y, \nu)$ induced by $(T, \varphi|_Y)$ satisfies $\varphi(u) \geq \nu(u) > {\bf 0}_n$ for all $u\in L$.
\item If $y\in Y$ satisfies $\nu^{[\varphi(y)]}(y) = 1$, then $\omega^{[\varphi(y)]}(x) = \omega^{[\varphi(y)]}(y) - 1$ for all $x\in \mathcal{D}_T(y_{\text{right}})$.
\item $N$ can be written as the sum $N = \sum_{u\in L} \binom{\varphi(u)}{\nu(u)}$.
\end{enum*}
\item\label{MacRepDefnCondensed} There does not exist any $y\in Y$ satisfying both of the following two conditions:
\begin{enum*}
\item The two subgraphs of $T$ induced by $\mathcal{D}_T(y_{\text{left}})$ and $\mathcal{D}_T(y_{\text{right}})$ are isomorphic as rooted trees, and they have the same corresponding $\varphi$-labels, i.e. if $\beta: \mathcal{D}_T(y_{\text{left}}) \to \mathcal{D}_T(y_{\text{right}})$ is the corresponding isomorphism, then $\varphi(\beta(x)) = \varphi(x)$ for all $x\in \mathcal{D}_T(y_{\text{left}})$.
\item Both $y_{\text{left}}$ and $y_{\text{right}}$ are $(\varphi(y)-1)$-leading vertices of $(T, \varphi)$.
\end{enum*}
\item\label{MacRepDefnCompressedLike} For every $t\in [n]$ and every $y\in Y$ such that $\varphi(y) = t$, the following conditions hold:
\begin{enum*}
\item If $y$ is a descendant of $x_{\text{left}}$ for some $x \in Y$ satisfying $\varphi(x) = t$, and $x_0, x_1, \dots, x_{\ell}$ is the path (of length $\ell\geq 2$) from $x_0 = x$ to $x_{\ell} = y_{\text{right}}$, then $\omega^{[t]}(x_i) = \omega^{[t]}(x) - i$ for all $i\in [\ell]$.
\item Let $y_1, \dots, y_m$ be the right relative sequence of $y_{\text{left}}$ in $T$. If $k\in [m]$ is the smallest integer such that $y_k \in L$ or $\varphi(y_k) \neq t$, then $\omega^{[t^{\prime}]}(y_k) \geq \omega^{[t^{\prime}]}(y_{\text{right}})$ for all $t^{\prime} \in [n]\backslash \{t\}$.
\end{enum*}
\item\label{MacRepDefnCompatible} $\xi_{(T, \varphi)}^{(t)}(x) \leq_{c\ell} \xi_{(T, \varphi)}^{(t)}\big(\zeta^i_t(x)\big)$ for every $i,t\in [n]$ satisfying $t<i$ and every $t$-leading vertex $x$ of $(T, \varphi)$ such that the $i$-th entry of ${\bf a} - \nu(x)$ is positive.
\end{enum*}
\end{definition}
\begin{remark}
Conditions \ref{MacRepDefnMacTree}, \ref{MacRepDefnCondensed} and \ref{MacRepDefnCompressedLike} can be easily checked by hand (at least when $T$ is not too large), while condition \ref{MacRepDefnCompatible} is somewhat tedious to check. For large $n$, we suggest checking these conditions algorithmically. All conditions involve only comparing the labels and left-weights of certain pairs of vertices in $T$, hence generalized Macaulay representations are purely numerical notions and do not depend on the existence of any colored complexes.
\end{remark}
For $N=0$, the {\it generalized ${\bf a}$-Macaulay representation of $0$} is defined to be the pair $\alpha_{\emptyset} := (T_{\emptyset}, \varphi_{\emptyset})$, where $T_{\emptyset}$ is the null graph (i.e. the unique graph with zero vertices), and $\varphi_{\emptyset}: \emptyset \to \emptyset$ is the (unique) function whose domain and codomain are empty sets. A {\it generalized ${\bf a}$-Macaulay representation} is a generalized ${\bf a}$-Macaulay representation of $N$ for some $N\in \mathbb{N}$, and a {\it generalized Macaulay representation} is a generalized ${\bf a}$-Macaulay representation for some ${\bf a} \in \mathbb{N}^n\backslash \{{\bf 0}_n\}$. The $n$-tuple ${\bf a}$ is called the {\it type} of this generalized Macaulay representation. A generalized Macaulay representation $\alpha$ is called {\it non-trivial} if $\alpha \neq \alpha_{\emptyset}$, and called {\it trivial} if $\alpha = \alpha_{\emptyset}$.

In general, a generalized ${\bf a}$-Macaulay representation of $N$ is not uniquely determined by ${\bf a}$ and $N$. However, for $n=1$, every compressed-like $k$-Macaulay tree is compatible (for any $k\in \mathbb{P}$), so Proposition \ref{prop:condensed-compressed-case-n=1} says the generalized $k$-Macaulay representation of $N$ is unique given $k,N\in \mathbb{P}$. For arbitrary ${\bf a} \in \mathbb{P}^n$, Theorem \ref{thm:numerical-char-color-compressed-complex} gives the bijection
\begin{equation}
\left\{\begin{tabular}{c}\text{isomorphism classes of pure color-compressed}\\ \text{${\bf a}$-balanced complexes with $N$ facets} \end{tabular} \right\} \longleftrightarrow \left\{\begin{tabular}{c}\text{generalized ${\bf a}$-Macaulay}\\ \text{representations of $N$}\end{tabular} \right\}.
\end{equation}

\begin{definition}
Let ${\bf a} \in \mathbb{P}^n$, let $\alpha = (T, \varphi)$ be an ${\bf a}$-Macaulay tree of $N \in \mathbb{P}$, let $(r_0, L, Y)$ be the $(\text{root},1,3)$-triple of $T$, and let $(T, \varphi|_Y, \nu)$ be the ${\bf a}$-splitting tree induced by $(T, \varphi|_Y)$. Then for any ${\bf x}\in \mathbb{Z}^n$, define
\[\partial_{[{\bf x}]}(\alpha) := \sum_{u\in L} \binom{\varphi(u)}{\nu(u) + {\bf x}}.\]
\end{definition}

\begin{theorem}\label{thm:numerical-char-pure-color-shifted-complex}
Let ${\bf a}\in \mathbb{P}^n$, and let $f = \{f_{{\bf b}}\}_{{\bf b} \leq {\bf a}}$ be an array of integers. Then the following are equivalent:
\begin{enum*}
\item $f$ is the fine $f$-vector of a pure color-compressed ${\bf a}$-balanced complex.
\item $f_{{\bf a}} > 0$, and there is a generalized ${\bf a}$-Macaulay representation $\alpha$ of $f_{{\bf a}}$ such that $f_{{\bf b}} = \partial_{[{\bf b} - {\bf a}]}(\alpha)$ for all ${\bf b} \leq {\bf a}$.
\end{enum*}
\end{theorem}

\begin{proof}
This immediately follows from Theorem \ref{thm:numerical-char-color-compressed-complex} and Proposition \ref{prop:shedding-tree-fine-f-vector}.
\end{proof}

\begin{theorem}\label{thm:numerical-char-completely-balanced-CM-complex}
Let $d\in \mathbb{P}$, let $f = \{f_S\}_{S \subseteq [d]}$ be an array of integers, and define $\phi: 2^{[d]} \to \{0,1\}^d$ by $S \mapsto (s_1, \dots, s_d)$, where $s_i = 1$ if and only if $i\in S$. Then the following are equivalent:
\begin{enum*}
\item $f$ is the flag $f$-vector of a $(d-1)$-dimensional completely balanced CM complex.
\item $f_{[d]} > 0$, and there is a generalized ${\bf 1}_d$-Macaulay representation $\alpha$ of $f_{[d]}$ such that $f_S = \partial_{[\phi(S) - {\bf 1}_d]}(\alpha)$ for all $S \subseteq [d]$.
\end{enum*}
\end{theorem}

\begin{proof}
As mentioned in Remark \ref{remark:color-shifted=color-compressed-in-1n-case}, the notion of `color-compressed' is equivalent to the notion of `color-shifted' in the specific case of type ${\bf 1}_d$, hence the assertion follows from Corollary \ref{cor:BalancedCMListOfEquiv} and Theorem \ref{thm:numerical-char-pure-color-shifted-complex}.
\end{proof}

\begin{remark}
In view of Theorem \ref{thm:BFSthm}, there is also a different but equivalent numerical characterization of the flag $h$-vectors of completely balanced CM complexes; see Corollary \ref{cor:equiv-flag-h-vector-char}.
\end{remark}

\begin{example}\label{example:non-trivial-non-eg-completely-balanced-CM}
By the Frankl-F\"{u}redi-Kalai theorem~\cite{book:BilleraBjornerHandbookDiscCompGeom}, the vector $(7, 11, 5)$ is the $f$-vector of a completely balanced CM complex, since the corresponding $h$-vector $(1,4,0,0)$ is the $f$-vector (including the $f_{-1}$ term) of a ${\bf 1}_3$-colored complex. One possible refinement of $(7, 11, 5)$ is the array $F = \{F_S\}_{S\subseteq [3]}$, given by
\begin{equation*}
F_{\{1,2,3\}} = 5,\ F_{\{1,2\}} = 5,\ F_{\{1, 3\}} = 3,\ F_{\{2,3\}} = 3,\ F_{\{1\}} = 3,\ F_{\{2\}} = 2,\ F_{\{3\}} = 2,\ F_{\emptyset} = 1,
\end{equation*}
and we would like to know if $F$ is the flag $f$-vector of a completely balanced CM complex. Figure \ref{Fig:non-trivial-eg} gives a list of all $24$ generalized ${\bf 1}_3$-Macaulay representations of $5$. Using Theorem \ref{thm:numerical-char-completely-balanced-CM-complex}, we then conclude that up to permutations of the $3$ colors such that $f_{\{1,2\}} \geq f_{\{1,3\}} \geq f_{\{2,3\}}$, the possible flag $f$-vectors $\{f_S\}_{S \subseteq [3]}$ of a $2$-dimensional completely balanced CM complex satisfying $f_{[3]} = 5$ (and $f_{\emptyset} = 1$) are given by
\begin{equation*}
\begin{bmatrix}f_{\{1,2\}} & f_{\{1,3\}} & f_{\{2,3\}}\\ f_{\{1\}} & f_{\{2\}} & f_{\{3\}} \end{bmatrix} \in \left\{ \begin{bmatrix}5 & 5 & 1\\ 5 & 1 & 1 \end{bmatrix}, \begin{bmatrix}5 & 4 & 2\\ 4 & 2 & 1 \end{bmatrix}, \begin{bmatrix}5 & 3 & 3\\ 3 & 3 & 1 \end{bmatrix}, \begin{bmatrix}5 & 3 & 2\\ 3 & 2 & 1 \end{bmatrix}, \begin{bmatrix}4 & 4 & 3\\ 3 & 2 & 2 \end{bmatrix}, \begin{bmatrix}4 & 3 & 3\\ 2 & 2 & 2 \end{bmatrix} \right\}.
\end{equation*}
Therefore $F$ is not the flag $f$-vector of a completely balanced CM complex.
\end{example}

\begin{figure}[hb!t]
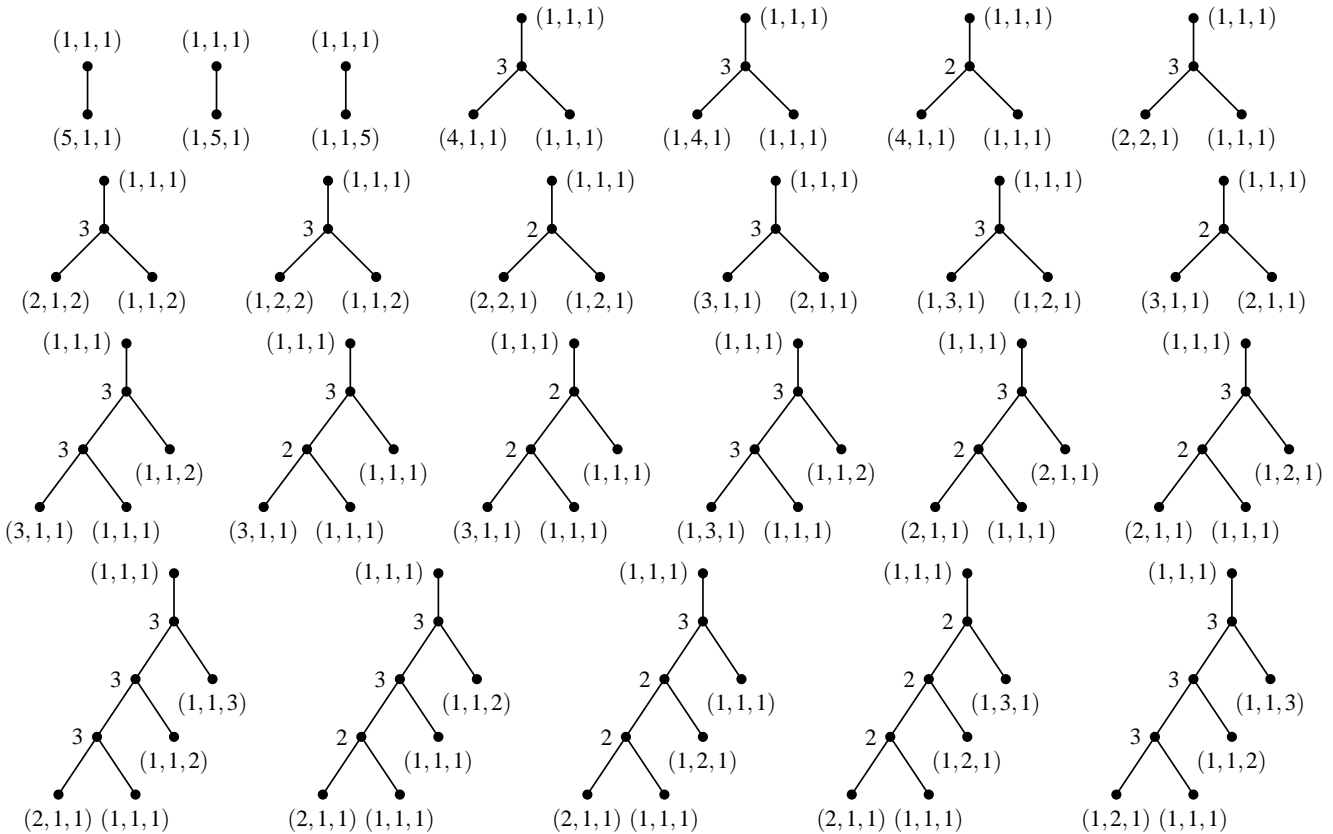

\centering
\begin{subfigure}[b]{0.09\textwidth}
\centering
\scalebox{0.8}{\tikz[thick,scale=0.8]{
    \draw {
    (2,3) node[circle, draw, fill=black, inner sep=0pt, minimum width=4pt, label=above:{$(1,1,1)$}] {} -- (2,2) node[circle, draw, fill=black, inner sep=0pt, minimum width=4pt, label=below:{$(5,1,1)$}] {}
    };
}}
\end{subfigure}
\begin{subfigure}[b]{0.09\textwidth}
\centering
\scalebox{0.8}{\tikz[thick,scale=0.8]{
    \draw {
    (2,3) node[circle, draw, fill=black, inner sep=0pt, minimum width=4pt, label=above:{$(1,1,1)$}] {} -- (2,2) node[circle, draw, fill=black, inner sep=0pt, minimum width=4pt, label=below:{$(1,5,1)$}] {}
    };
}}
\end{subfigure}
\begin{subfigure}[b]{0.09\textwidth}
\centering
\scalebox{0.8}{\tikz[thick,scale=0.8]{
    \draw {
    (2,3) node[circle, draw, fill=black, inner sep=0pt, minimum width=4pt, label=above:{$(1,1,1)$}] {} -- (2,2) node[circle, draw, fill=black, inner sep=0pt, minimum width=4pt, label=below:{$(1,1,5)$}] {}
    };
}}
\end{subfigure}
\begin{subfigure}[b]{0.16\textwidth}
\centering
\scalebox{0.8}{\tikz[thick,scale=0.8]{
    \draw {
    (2,3) node[circle, draw, fill=black, inner sep=0pt, minimum width=4pt, label=right:{$(1,1,1)$}] {} -- (2,2) node[circle, draw, fill=black, inner sep=0pt, minimum width=4pt, label=left:{$3$}] {} -- (3,1) node[circle, draw, fill=black, inner sep=0pt, minimum width=4pt, label=below:{$(1,1,1)$}] {}
    (2,2) -- (1,1) node[circle, draw, fill=black, inner sep=0pt, minimum width=4pt, label=below:{$(4,1,1)$}] {}
    };
}}
\end{subfigure}
\begin{subfigure}[b]{0.16\textwidth}
\centering
\scalebox{0.8}{\tikz[thick,scale=0.8]{
    \draw {
    (2,3) node[circle, draw, fill=black, inner sep=0pt, minimum width=4pt, label=right:{$(1,1,1)$}] {} -- (2,2) node[circle, draw, fill=black, inner sep=0pt, minimum width=4pt, label=left:{$3$}] {} -- (3,1) node[circle, draw, fill=black, inner sep=0pt, minimum width=4pt, label=below:{$(1,1,1)$}] {}
    (2,2) -- (1,1) node[circle, draw, fill=black, inner sep=0pt, minimum width=4pt, label=below:{$(1,4,1)$}] {}
    };
}}
\end{subfigure}
\begin{subfigure}[b]{0.16\textwidth}
\centering
\scalebox{0.8}{\tikz[thick,scale=0.8]{
    \draw {
    (2,3) node[circle, draw, fill=black, inner sep=0pt, minimum width=4pt, label=right:{$(1,1,1)$}] {} -- (2,2) node[circle, draw, fill=black, inner sep=0pt, minimum width=4pt, label=left:{$2$}] {} -- (3,1) node[circle, draw, fill=black, inner sep=0pt, minimum width=4pt, label=below:{$(1,1,1)$}] {}
    (2,2) -- (1,1) node[circle, draw, fill=black, inner sep=0pt, minimum width=4pt, label=below:{$(4,1,1)$}] {}
    };
}}
\end{subfigure}
\begin{subfigure}[b]{0.16\textwidth}
\centering
\scalebox{0.8}{\tikz[thick,scale=0.8]{
    \draw {
    (2,3) node[circle, draw, fill=black, inner sep=0pt, minimum width=4pt, label=right:{$(1,1,1)$}] {} -- (2,2) node[circle, draw, fill=black, inner sep=0pt, minimum width=4pt, label=left:{$3$}] {} -- (3,1) node[circle, draw, fill=black, inner sep=0pt, minimum width=4pt, label=below:{$(1,1,1)$}] {}
    (2,2) -- (1,1) node[circle, draw, fill=black, inner sep=0pt, minimum width=4pt, label=below:{$(2,2,1)$}] {}
    };
}}
\end{subfigure}

\begin{subfigure}[b]{0.16\textwidth}
\centering
\scalebox{0.8}{\tikz[thick,scale=0.8]{
    \draw {
    (2,3) node[circle, draw, fill=black, inner sep=0pt, minimum width=4pt, label=right:{$(1,1,1)$}] {} -- (2,2) node[circle, draw, fill=black, inner sep=0pt, minimum width=4pt, label=left:{$3$}] {} -- (3,1) node[circle, draw, fill=black, inner sep=0pt, minimum width=4pt, label=below:{$(1,1,2)$}] {}
    (2,2) -- (1,1) node[circle, draw, fill=black, inner sep=0pt, minimum width=4pt, label=below:{$(2,1,2)$}] {}
    };
}}
\end{subfigure}
\begin{subfigure}[b]{0.16\textwidth}
\centering
\scalebox{0.8}{\tikz[thick,scale=0.8]{
    \draw {
    (2,3) node[circle, draw, fill=black, inner sep=0pt, minimum width=4pt, label=right:{$(1,1,1)$}] {} -- (2,2) node[circle, draw, fill=black, inner sep=0pt, minimum width=4pt, label=left:{$3$}] {} -- (3,1) node[circle, draw, fill=black, inner sep=0pt, minimum width=4pt, label=below:{$(1,1,2)$}] {}
    (2,2) -- (1,1) node[circle, draw, fill=black, inner sep=0pt, minimum width=4pt, label=below:{$(1,2,2)$}] {}
    };
}}
\end{subfigure}
\begin{subfigure}[b]{0.16\textwidth}
\centering
\scalebox{0.8}{\tikz[thick,scale=0.8]{
    \draw {
    (2,3) node[circle, draw, fill=black, inner sep=0pt, minimum width=4pt, label=right:{$(1,1,1)$}] {} -- (2,2) node[circle, draw, fill=black, inner sep=0pt, minimum width=4pt, label=left:{$2$}] {} -- (3,1) node[circle, draw, fill=black, inner sep=0pt, minimum width=4pt, label=below:{$(1,2,1)$}] {}
    (2,2) -- (1,1) node[circle, draw, fill=black, inner sep=0pt, minimum width=4pt, label=below:{$(2,2,1)$}] {}
    };
}}
\end{subfigure}
\begin{subfigure}[b]{0.16\textwidth}
\centering
\scalebox{0.8}{\tikz[thick,scale=0.8]{
    \draw {
    (2,3) node[circle, draw, fill=black, inner sep=0pt, minimum width=4pt, label=right:{$(1,1,1)$}] {} -- (2,2) node[circle, draw, fill=black, inner sep=0pt, minimum width=4pt, label=left:{$3$}] {} -- (3,1) node[circle, draw, fill=black, inner sep=0pt, minimum width=4pt, label=below:{$(2,1,1)$}] {}
    (2,2) -- (1,1) node[circle, draw, fill=black, inner sep=0pt, minimum width=4pt, label=below:{$(3,1,1)$}] {}
    };
}}
\end{subfigure}
\begin{subfigure}[b]{0.16\textwidth}
\centering
\scalebox{0.8}{\tikz[thick,scale=0.8]{
    \draw {
    (2,3) node[circle, draw, fill=black, inner sep=0pt, minimum width=4pt, label=right:{$(1,1,1)$}] {} -- (2,2) node[circle, draw, fill=black, inner sep=0pt, minimum width=4pt, label=left:{$3$}] {} -- (3,1) node[circle, draw, fill=black, inner sep=0pt, minimum width=4pt, label=below:{$(1,2,1)$}] {}
    (2,2) -- (1,1) node[circle, draw, fill=black, inner sep=0pt, minimum width=4pt, label=below:{$(1,3,1)$}] {}
    };
}}
\end{subfigure}
\begin{subfigure}[b]{0.16\textwidth}
\centering
\scalebox{0.8}{\tikz[thick,scale=0.8]{
    \draw {
    (2,3) node[circle, draw, fill=black, inner sep=0pt, minimum width=4pt, label=right:{$(1,1,1)$}] {} -- (2,2) node[circle, draw, fill=black, inner sep=0pt, minimum width=4pt, label=left:{$2$}] {} -- (3,1) node[circle, draw, fill=black, inner sep=0pt, minimum width=4pt, label=below:{$(2,1,1)$}] {}
    (2,2) -- (1,1) node[circle, draw, fill=black, inner sep=0pt, minimum width=4pt, label=below:{$(3,1,1)$}] {}
    };
}}
\end{subfigure}

\begin{subfigure}[b]{0.16\textwidth}
\centering
\scalebox{0.8}{\tikz[thick,scale=0.8]{
    \draw {
    (1.8,3.4) node[circle, draw, fill=black, inner sep=0pt, minimum width=4pt, label=left:{$(1,1,1)$}] {} -- (1.8,2.4) node[circle, draw, fill=black, inner sep=0pt, minimum width=4pt, label=left:{$3$}] {} -- (0.9,1.2) node[circle, draw, fill=black, inner sep=0pt, minimum width=4pt, label=left:{$3$}] {} -- (0,0) node[circle, draw, fill=black, inner sep=0pt, minimum width=4pt, label=below:{$(3,1,1)$}] {}
    (1.8,2.4) -- (2.7,1.2) node[circle, draw, fill=black, inner sep=0pt, minimum width=4pt, label=below:{$(1,1,2)$}] {}
    (0.9,1.2) -- (1.8,0) node[circle, draw, fill=black, inner sep=0pt, minimum width=4pt, label=below:{$(1,1,1)$}] {}
    };
}}
\end{subfigure}
\begin{subfigure}[b]{0.16\textwidth}
\centering
\scalebox{0.8}{\tikz[thick,scale=0.8]{
    \draw {
    (1.8,3.4) node[circle, draw, fill=black, inner sep=0pt, minimum width=4pt, label=left:{$(1,1,1)$}] {} -- (1.8,2.4) node[circle, draw, fill=black, inner sep=0pt, minimum width=4pt, label=left:{$3$}] {} -- (0.9,1.2) node[circle, draw, fill=black, inner sep=0pt, minimum width=4pt, label=left:{$2$}] {} -- (0,0) node[circle, draw, fill=black, inner sep=0pt, minimum width=4pt, label=below:{$(3,1,1)$}] {}
    (1.8,2.4) -- (2.7,1.2) node[circle, draw, fill=black, inner sep=0pt, minimum width=4pt, label=below:{$(1,1,1)$}] {}
    (0.9,1.2) -- (1.8,0) node[circle, draw, fill=black, inner sep=0pt, minimum width=4pt, label=below:{$(1,1,1)$}] {}
    };
}}
\end{subfigure}
\begin{subfigure}[b]{0.16\textwidth}
\centering
\scalebox{0.8}{\tikz[thick,scale=0.8]{
    \draw {
    (1.8,3.4) node[circle, draw, fill=black, inner sep=0pt, minimum width=4pt, label=left:{$(1,1,1)$}] {} -- (1.8,2.4) node[circle, draw, fill=black, inner sep=0pt, minimum width=4pt, label=left:{$2$}] {} -- (0.9,1.2) node[circle, draw, fill=black, inner sep=0pt, minimum width=4pt, label=left:{$2$}] {} -- (0,0) node[circle, draw, fill=black, inner sep=0pt, minimum width=4pt, label=below:{$(3,1,1)$}] {}
    (1.8,2.4) -- (2.7,1.2) node[circle, draw, fill=black, inner sep=0pt, minimum width=4pt, label=below:{$(1,1,1)$}] {}
    (0.9,1.2) -- (1.8,0) node[circle, draw, fill=black, inner sep=0pt, minimum width=4pt, label=below:{$(1,1,1)$}] {}
    };
}}
\end{subfigure}
\begin{subfigure}[b]{0.16\textwidth}
\centering
\scalebox{0.8}{\tikz[thick,scale=0.8]{
    \draw {
    (1.8,3.4) node[circle, draw, fill=black, inner sep=0pt, minimum width=4pt, label=left:{$(1,1,1)$}] {} -- (1.8,2.4) node[circle, draw, fill=black, inner sep=0pt, minimum width=4pt, label=left:{$3$}] {} -- (0.9,1.2) node[circle, draw, fill=black, inner sep=0pt, minimum width=4pt, label=left:{$3$}] {} -- (0,0) node[circle, draw, fill=black, inner sep=0pt, minimum width=4pt, label=below:{$(1,3,1)$}] {}
    (1.8,2.4) -- (2.7,1.2) node[circle, draw, fill=black, inner sep=0pt, minimum width=4pt, label=below:{$(1,1,2)$}] {}
    (0.9,1.2) -- (1.8,0) node[circle, draw, fill=black, inner sep=0pt, minimum width=4pt, label=below:{$(1,1,1)$}] {}
    };
}}
\end{subfigure}
\begin{subfigure}[b]{0.16\textwidth}
\centering
\scalebox{0.8}{\tikz[thick,scale=0.8]{
    \draw {
    (1.8,3.4) node[circle, draw, fill=black, inner sep=0pt, minimum width=4pt, label=left:{$(1,1,1)$}] {} -- (1.8,2.4) node[circle, draw, fill=black, inner sep=0pt, minimum width=4pt, label=left:{$3$}] {} -- (0.9,1.2) node[circle, draw, fill=black, inner sep=0pt, minimum width=4pt, label=left:{$2$}] {} -- (0,0) node[circle, draw, fill=black, inner sep=0pt, minimum width=4pt, label=below:{$(2,1,1)$}] {}
    (1.8,2.4) -- (2.7,1.2) node[circle, draw, fill=black, inner sep=0pt, minimum width=4pt, label=below:{$(2,1,1)$}] {}
    (0.9,1.2) -- (1.8,0) node[circle, draw, fill=black, inner sep=0pt, minimum width=4pt, label=below:{$(1,1,1)$}] {}
    };
}}
\end{subfigure}
\begin{subfigure}[b]{0.16\textwidth}
\centering
\scalebox{0.8}{\tikz[thick,scale=0.8]{
    \draw {
    (1.8,3.4) node[circle, draw, fill=black, inner sep=0pt, minimum width=4pt, label=left:{$(1,1,1)$}] {} -- (1.8,2.4) node[circle, draw, fill=black, inner sep=0pt, minimum width=4pt, label=left:{$3$}] {} -- (0.9,1.2) node[circle, draw, fill=black, inner sep=0pt, minimum width=4pt, label=left:{$2$}] {} -- (0,0) node[circle, draw, fill=black, inner sep=0pt, minimum width=4pt, label=below:{$(2,1,1)$}] {}
    (1.8,2.4) -- (2.7,1.2) node[circle, draw, fill=black, inner sep=0pt, minimum width=4pt, label=below:{$(1,2,1)$}] {}
    (0.9,1.2) -- (1.8,0) node[circle, draw, fill=black, inner sep=0pt, minimum width=4pt, label=below:{$(1,1,1)$}] {}
    };
}}
\end{subfigure}

\begin{subfigure}[b]{0.19\textwidth}
\centering
\scalebox{0.8}{\tikz[thick,scale=0.8]{
    \draw {
    (1.6,3.4) node[circle, draw, fill=black, inner sep=0pt, minimum width=4pt, label=left:{$(1,1,1)$}] {} -- (1.6,2.4) node[circle, draw, fill=black, inner sep=0pt, minimum width=4pt, label=left:{$3$}] {} -- (0.8,1.2) node[circle, draw, fill=black, inner sep=0pt, minimum width=4pt, label=left:{$3$}] {} -- (0,0) node[circle, draw, fill=black, inner sep=0pt, minimum width=4pt, label=left:{$3$}] {} -- (-0.8,-1.2) node[circle, draw, fill=black, inner sep=0pt, minimum width=4pt, label=below:{$(2,1,1)$}] {}
    (1.6,2.4) -- (2.4,1.2) node[circle, draw, fill=black, inner sep=0pt, minimum width=4pt, label=below:{$(1,1,3)$}] {}
    (0.8,1.2) -- (1.6,0) node[circle, draw, fill=black, inner sep=0pt, minimum width=4pt, label=below:{$(1,1,2)$}] {}
    (0,0) -- (0.8,-1.2) node[circle, draw, fill=black, inner sep=0pt, minimum width=4pt, label=below:{$(1,1,1)$}] {}
    };
}}
\end{subfigure}
\begin{subfigure}[b]{0.19\textwidth}
\centering
\scalebox{0.8}{\tikz[thick,scale=0.8]{
    \draw {
    (1.6,3.4) node[circle, draw, fill=black, inner sep=0pt, minimum width=4pt, label=left:{$(1,1,1)$}] {} -- (1.6,2.4) node[circle, draw, fill=black, inner sep=0pt, minimum width=4pt, label=left:{$3$}] {} -- (0.8,1.2) node[circle, draw, fill=black, inner sep=0pt, minimum width=4pt, label=left:{$3$}] {} -- (0,0) node[circle, draw, fill=black, inner sep=0pt, minimum width=4pt, label=left:{$2$}] {} -- (-0.8,-1.2) node[circle, draw, fill=black, inner sep=0pt, minimum width=4pt, label=below:{$(2,1,1)$}] {}
    (1.6,2.4) -- (2.4,1.2) node[circle, draw, fill=black, inner sep=0pt, minimum width=4pt, label=below:{$(1,1,2)$}] {}
    (0.8,1.2) -- (1.6,0) node[circle, draw, fill=black, inner sep=0pt, minimum width=4pt, label=below:{$(1,1,1)$}] {}
    (0,0) -- (0.8,-1.2) node[circle, draw, fill=black, inner sep=0pt, minimum width=4pt, label=below:{$(1,1,1)$}] {}
    };
}}
\end{subfigure}
\begin{subfigure}[b]{0.19\textwidth}
\centering
\scalebox{0.8}{\tikz[thick,scale=0.8]{
    \draw {
    (1.6,3.4) node[circle, draw, fill=black, inner sep=0pt, minimum width=4pt, label=left:{$(1,1,1)$}] {} -- (1.6,2.4) node[circle, draw, fill=black, inner sep=0pt, minimum width=4pt, label=left:{$3$}] {} -- (0.8,1.2) node[circle, draw, fill=black, inner sep=0pt, minimum width=4pt, label=left:{$2$}] {} -- (0,0) node[circle, draw, fill=black, inner sep=0pt, minimum width=4pt, label=left:{$2$}] {} -- (-0.8,-1.2) node[circle, draw, fill=black, inner sep=0pt, minimum width=4pt, label=below:{$(2,1,1)$}] {}
    (1.6,2.4) -- (2.4,1.2) node[circle, draw, fill=black, inner sep=0pt, minimum width=4pt, label=below:{$(1,1,1)$}] {}
    (0.8,1.2) -- (1.6,0) node[circle, draw, fill=black, inner sep=0pt, minimum width=4pt, label=below:{$(1,2,1)$}] {}
    (0,0) -- (0.8,-1.2) node[circle, draw, fill=black, inner sep=0pt, minimum width=4pt, label=below:{$(1,1,1)$}] {}
    };
}}
\end{subfigure}
\begin{subfigure}[b]{0.19\textwidth}
\centering
\scalebox{0.8}{\tikz[thick,scale=0.8]{
    \draw {
    (1.6,3.4) node[circle, draw, fill=black, inner sep=0pt, minimum width=4pt, label=left:{$(1,1,1)$}] {} -- (1.6,2.4) node[circle, draw, fill=black, inner sep=0pt, minimum width=4pt, label=left:{$2$}] {} -- (0.8,1.2) node[circle, draw, fill=black, inner sep=0pt, minimum width=4pt, label=left:{$2$}] {} -- (0,0) node[circle, draw, fill=black, inner sep=0pt, minimum width=4pt, label=left:{$2$}] {} -- (-0.8,-1.2) node[circle, draw, fill=black, inner sep=0pt, minimum width=4pt, label=below:{$(2,1,1)$}] {}
    (1.6,2.4) -- (2.4,1.2) node[circle, draw, fill=black, inner sep=0pt, minimum width=4pt, label=below:{$(1,3,1)$}] {}
    (0.8,1.2) -- (1.6,0) node[circle, draw, fill=black, inner sep=0pt, minimum width=4pt, label=below:{$(1,2,1)$}] {}
    (0,0) -- (0.8,-1.2) node[circle, draw, fill=black, inner sep=0pt, minimum width=4pt, label=below:{$(1,1,1)$}] {}
    };
}}
\end{subfigure}
\begin{subfigure}[b]{0.19\textwidth}
\centering
\scalebox{0.8}{\tikz[thick,scale=0.8]{
    \draw {
    (1.6,3.4) node[circle, draw, fill=black, inner sep=0pt, minimum width=4pt, label=left:{$(1,1,1)$}] {} -- (1.6,2.4) node[circle, draw, fill=black, inner sep=0pt, minimum width=4pt, label=left:{$3$}] {} -- (0.8,1.2) node[circle, draw, fill=black, inner sep=0pt, minimum width=4pt, label=left:{$3$}] {} -- (0,0) node[circle, draw, fill=black, inner sep=0pt, minimum width=4pt, label=left:{$3$}] {} -- (-0.8,-1.2) node[circle, draw, fill=black, inner sep=0pt, minimum width=4pt, label=below:{$(1,2,1)$}] {}
    (1.6,2.4) -- (2.4,1.2) node[circle, draw, fill=black, inner sep=0pt, minimum width=4pt, label=below:{$(1,1,3)$}] {}
    (0.8,1.2) -- (1.6,0) node[circle, draw, fill=black, inner sep=0pt, minimum width=4pt, label=below:{$(1,1,2)$}] {}
    (0,0) -- (0.8,-1.2) node[circle, draw, fill=black, inner sep=0pt, minimum width=4pt, label=below:{$(1,1,1)$}] {}
    };
}}
\end{subfigure}
\caption{All $24$ generalized ${\bf 1}_3$-Macaulay representations of $5$.}
\label{Fig:non-trivial-eg}
\end{figure}

\begin{definition}\label{defn-twin}
Let ${\bf a}, {\bf a}^{\prime} \in \mathbb{N}^n\backslash\{{\bf 0}_n\}$ satisfy ${\bf a}^{\prime} \leq {\bf a}$, and let $\alpha = (T, \varphi)$ be an ${\bf a}$-Macaulay tree. Denote the $(\text{root},1,3)$-triple of $T$ by $(r_0, L, Y)$, and let $(T, \varphi|_Y, \nu)$ be the ${\bf a}$-splitting tree induced by $(T, \varphi|_Y)$. Then, the {\it ${\bf a}^{\prime}$-twin} of $(T, \varphi)$ is the pair $(T^{\prime}, \varphi^{\prime})$ constructed from $(T, \varphi)$ via the following algorithm:
\begin{center}
\begin{algorithmic}
\small
\State $S \gets \{x\in (Y \cup L \cup \{r_0\}): \nu(x) > {\bf a} - {\bf a}^{\prime}\}$.
\State $S_0 \gets \{y\in Y: \nu(y) > {\bf a} - {\bf a}^{\prime}, \nu(y_{\text{right}}) \not> {\bf a} - {\bf a}^{\prime}\}$.
\State $T^{\prime} \gets$ subgraph of $T$ induced by $S$.
\State $\varphi^{\prime} \gets \varphi|_S$.
\State $\varphi^{\prime}(r_0) \gets {\bf a}^{\prime}$.
\While {$S_0 = \emptyset$}
    \State $z \gets$ largest vertex in $S_0$ with respect to the depth-first order.
    \State $z_{\text{child}} \gets$ unique child of $z$ in $T^{\prime}$.
    \State $S_0 \gets S_0\backslash \{z\}$.
    \State $L^{\prime} \gets$ set of leaves of $T^{\prime}$.
    \If {$z_{\text{child}} \in L^{\prime}$}
        \State $\varphi^{\prime}(z) \gets \varphi^{\prime}(z_{\text{child}}) + \boldsymbol{\delta}_{\varphi^{\prime}(z),n}$.
        \State $\varphi^{\prime}(z_{\text{child}}) = \varphi^{\prime}(z)$.
    \Else
        \State $t^{\prime} \gets \varphi^{\prime}(z)$.
        \State $\varphi^{\prime}(z) \gets \varphi^{\prime}(z_{\text{child}})$.
        \For {$x\in \mathcal{D}_{T^{\prime}}(z_{\text{child}}) \cap L^{\prime}$}
            \State $(\varphi^{\prime})^{[t^{\prime}]}(x) \gets (\varphi^{\prime})^{[t^{\prime}]}(x) + 1$.
        \EndFor
    \EndIf
    \State $(T^{\prime}, \varphi^{\prime}) \gets (T^{\prime}/\{z, z_{\text{child}}\}, \varphi^{\prime}/\{z, z_{\text{child}}\})$.
\EndWhile
\end{algorithmic}
\end{center}
\end{definition}

\begin{remark}
It is easy to verify that $(T^{\prime}, \varphi^{\prime})$ in Definition \ref{defn-twin} is an ${\bf a}^{\prime}$-Macaulay tree of $\partial_{[{\bf a}^{\prime} - {\bf a}]}(\alpha)$.
\end{remark}

For the rest of this subsection, let ${\bf a}, {\bf a}^{\prime} \in \mathbb{N}^n\backslash \{{\bf 0}_n\}$ satisfy ${\bf a}^{\prime} \leq {\bf a}$, let $\alpha = (T, \varphi)$ be an ${\bf a}$-Macaulay tree, let $\alpha^{\prime} = (T^{\prime}, \varphi^{\prime})$ be an ${\bf a}^{\prime}$-Macaulay tree, and let $\alpha^{\prime\prime} = (T^{\prime\prime}, \varphi^{\prime\prime})$ be the condensation of the ${\bf a}^{\prime}$-twin of $\alpha$. Denote the $(\text{root},1,3)$-triples of $T^{\prime}$ and $T^{\prime\prime}$ by $(r_0^{\prime}, L^{\prime}, Y^{\prime})$ and $(r_0^{\prime\prime}, L^{\prime\prime}, Y^{\prime\prime})$ respectively. Next, construct a graph $\widetilde{T}$ as follows: Take the disjoint union of $T^{\prime}$ and $T^{\prime\prime}$, identify the roots $r_0^{\prime}$, $r_0^{\prime\prime}$ as a single vertex $\hat{r}$, then attach a leaf $\widetilde{r}$ to $\hat{r}$. Note that $\widetilde{T}$ is a trivalent planted binary tree with root $\widetilde{r}$, and treat $T^{\prime}$ and $T^{\prime\prime}$ as subgraphs of $\widetilde{T}$. Now, define the vertex labeling $\widetilde{\varphi}$ of $\widetilde{T}$ as follows: $\widetilde{\varphi}(\widetilde{r}) = ({\bf a}^{\prime},2) \in \mathbb{P}^{n+1}$; $\widetilde{\varphi}(\hat{r}) = n+1$; $\widetilde{\varphi}(x) = \varphi^{\prime}(x) \in \mathbb{P}$ if $x\in Y^{\prime}$; $\widetilde{\varphi}(x) = (\varphi^{\prime}(x),2) \in \mathbb{P}^{n+1}$ if $x\in L^{\prime}$; $\widetilde{\varphi}(x) = \varphi^{\prime\prime}(x) \in \mathbb{P}$ if $x\in Y^{\prime\prime}$; and $\widetilde{\varphi}(x) = (\varphi^{\prime\prime}(x),1) \in \mathbb{P}^{n+1}$ if $x\in L^{\prime\prime}$. We check that $(\widetilde{T}, \widetilde{\varphi})$ is an $({\bf a}^{\prime},2)$-Macaulay tree, which we denote by $\alpha \wedge \alpha^{\prime}$. Furthermore, $\alpha \wedge \alpha^{\prime}$ is compressed-like if both $\alpha$ and $\alpha^{\prime}$ are compressed-like, and $\alpha \wedge \alpha^{\prime}$ is condensed if $\alpha^{\prime}$ is condensed.

\begin{figure}[h!t]
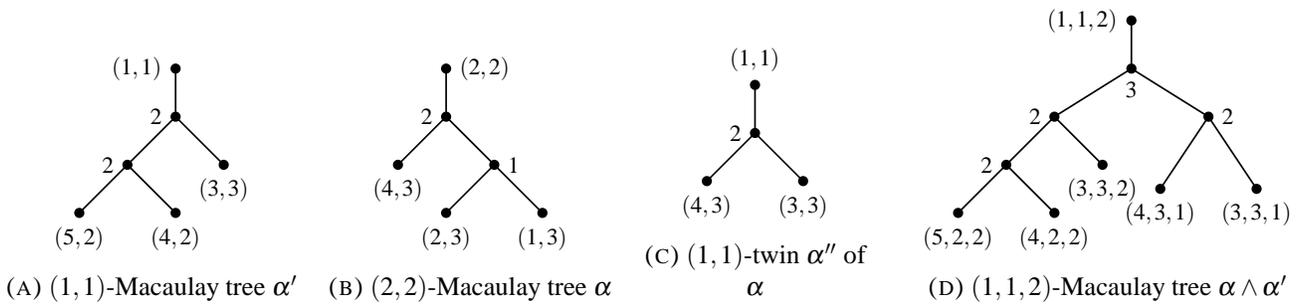

\centering
\begin{subfigure}[b]{0.23\textwidth}
\centering
\scalebox{0.8}{\tikz[thick,scale=0.8]{
    \draw {
    (2,3) node[circle, draw, fill=black, inner sep=0pt, minimum width=4pt, label=left:{$(1,1)$}] {} -- (2,2) node[circle, draw, fill=black, inner sep=0pt, minimum width=4pt, label=left:{$2$}] {} -- (1,1) node[circle, draw, fill=black, inner sep=0pt, minimum width=4pt, label=left:{$2$}] {} -- (0,0) node[circle, draw, fill=black, inner sep=0pt, minimum width=4pt, label=below:{$(5,2)$}] {}
    (2,2) -- (3,1) node[circle, draw, fill=black, inner sep=0pt, minimum width=4pt, label=below:{$(3,3)$}] {}
    (1,1) -- (2,0) node[circle, draw, fill=black, inner sep=0pt, minimum width=4pt, label=below:{$(4,2)$}] {}
    };
}}
\caption{\centering $(1,1)$-Macaulay tree $\alpha^{\prime}$}
\end{subfigure}
\begin{subfigure}[b]{0.23\textwidth}
\centering
\scalebox{0.8}{\tikz[thick,scale=0.8]{
    \draw {
    (2,3) node[circle, draw, fill=black, inner sep=0pt, minimum width=4pt, label=right:{$(2,2)$}] {} -- (2,2) node[circle, draw, fill=black, inner sep=0pt, minimum width=4pt, label=left:{$2$}] {} -- (3,1) node[circle, draw, fill=black, inner sep=0pt, minimum width=4pt, label=right:{$1$}] {} -- (4,0) node[circle, draw, fill=black, inner sep=0pt, minimum width=4pt, label=below:{$(1,3)$}] {}
    (2,2) -- (1,1) node[circle, draw, fill=black, inner sep=0pt, minimum width=4pt, label=below:{$(4,3)$}] {}
    (3,1) -- (2,0) node[circle, draw, fill=black, inner sep=0pt, minimum width=4pt, label=below:{$(2,3)$}] {}
    };
}}
\caption{\centering $(2,2)$-Macaulay tree $\alpha$}
\end{subfigure}
\begin{subfigure}[b]{0.18\textwidth}
\centering
\scalebox{0.8}{\tikz[thick,scale=0.8]{
    \draw {
    (2,3) node[circle, draw, fill=black, inner sep=0pt, minimum width=4pt, label=above:{$(1,1)$}] {} -- (2,2) node[circle, draw, fill=black, inner sep=0pt, minimum width=4pt, label=left:{$2$}] {} -- (3,1) node[circle, draw, fill=black, inner sep=0pt, minimum width=4pt, label=below:{$(3,3)$}] {}
    (2,2) -- (1,1) node[circle, draw, fill=black, inner sep=0pt, minimum width=4pt, label=below:{$(4,3)$}] {}
    };
}}
\caption{\centering $(1,1)$-twin $\alpha^{\prime\prime}$ of $\alpha$}
\end{subfigure}
\begin{subfigure}[b]{0.33\textwidth}
\centering
\scalebox{0.8}{\tikz[thick,scale=0.8]{
    \draw {
    (3.6,4) node[circle, draw, fill=black, inner sep=0pt, minimum width=4pt, label=left:{$(1,1,2)$}] {} -- (3.6,3) node[circle, draw, fill=black, inner sep=0pt, minimum width=4pt, label=below:{$3$}] {} -- (2,2) node[circle, draw, fill=black, inner sep=0pt, minimum width=4pt, label=left:{$2$}] {} -- (1,1) node[circle, draw, fill=black, inner sep=0pt, minimum width=4pt, label=left:{$2$}] {} -- (0,0) node[circle, draw, fill=black, inner sep=0pt, minimum width=4pt, label=below:{$(5,2,2)$}] {}
    (2,2) -- (3,1) node[circle, draw, fill=black, inner sep=0pt, minimum width=4pt, label=below:{$(3,3,2)$}] {}
    (1,1) -- (2,0) node[circle, draw, fill=black, inner sep=0pt, minimum width=4pt, label=below:{$(4,2,2)$}] {}
    (3.6,3) -- (5.2,2) node[circle, draw, fill=black, inner sep=0pt, minimum width=4pt, label=right:{$2$}] {} -- (6.2,0.5) node[circle, draw, fill=black, inner sep=0pt, minimum width=4pt, label=below:{$(3,3,1)$}] {}
    (5.2,2) -- (4.2,0.5) node[circle, draw, fill=black, inner sep=0pt, minimum width=4pt, label=below:{$(4,3,1)$}] {}
    };
}}
\caption{$(1,1,2)$-Macaulay tree $\alpha \wedge \alpha^{\prime}$}
\end{subfigure}
\caption{Macaulay trees $\alpha$, $\alpha^{\prime}$, $\alpha^{\prime\prime}$, and $\alpha \wedge \alpha^{\prime}$.}
\label{Fig:wedge product}
\end{figure}

\begin{example}
Figure \ref{Fig:wedge product} shows an example of a $(1,1)$-Macaulay tree $\alpha^{\prime}$, a $(2,2)$-Macaulay tree $\alpha$, the $(1,1)$-twin $\alpha^{\prime\prime}$ of $\alpha$ (which is its condensation), and the corresponding $(1,1,2)$-Macaulay tree $\alpha \wedge \alpha^{\prime}$.
\end{example}

\begin{definition}
Let $\alpha, \alpha^{\prime}$ be non-trivial generalized Macaulay representations. If $\alpha \wedge \alpha^{\prime}$ is compatible, or equivalently, if $\alpha \wedge \alpha^{\prime}$ is a generalized Macaulay representation, then we write $\alpha \preceq \alpha^{\prime}$. Also, write $\alpha_{\emptyset} \preceq \alpha^{\prime\prime}$ for all generalized Macaulay representations $\alpha^{\prime\prime}$.
\end{definition}

\begin{proposition}\label{prop:preceq-partial-order}
$\preceq$ is a partial order on generalized Macaulay representations.
\end{proposition}

\begin{proof}
Given $n$-tuples $\rho = ([s_1]_1, \dots, [s_n]_n)$ and $\rho^{\prime} = ([s_1^{\prime}]_1, \dots, [s_n^{\prime}]_n)$ for some positive integers $s_1, \dots, s_n, s_1^{\prime}, \dots, s_n^{\prime}$, write $\rho \leq \rho^{\prime}$ to mean $(s_1, \dots, s_n) \leq (s_1^{\prime}, \dots, s_n^{\prime})$. Let ${\bf a}, {\bf b}, {\bf c}\in \mathbb{N}^n\backslash\{{\bf 0}_n\}$ satisfy ${\bf c} \leq {\bf b} \leq {\bf a}$, and let $\alpha, \beta, \gamma$ be arbitrary generalized Macaulay representations of types ${\bf a}, {\bf b}, {\bf c}$ respectively. Without loss of generality, assume $\alpha, \beta, \gamma$ are all non-trivial. Clearly $\alpha \preceq \alpha$, since $\alpha$ is compatible implies $\alpha \wedge \alpha$ is compatible.

Next, suppose $\alpha \preceq \beta$ and $\beta \preceq \alpha$. This means $\alpha \wedge \beta$ and $\beta \wedge \alpha$ are well-defined, which implies ${\bf a} = {\bf b}$, hence the condensation of the ${\bf a}$-twin of $\beta$ is $\beta$, and the condensation of the ${\bf b}$-twin of $\alpha$ is $\alpha$. Since $\alpha \wedge \beta$ is compressed-like and compatible, Theorem \ref{thm:compressed-like-compatible=>color-compressed} yields $(\Delta_{\alpha \wedge \beta}, \pi_{\alpha \wedge \beta})$ is color-compressed, hence by the construction of $\alpha \wedge \beta$, we get $\Delta_{\alpha} \subseteq \Delta_{\beta}$ and $\pi_{\alpha} \leq \pi_{\beta}$. A symmetric argument also gives $\Delta_{\beta} \subseteq \Delta_{\alpha}$ and $\pi_{\beta} \leq \pi_{\alpha}$, thus $(\Delta_{\alpha}, \pi_{\alpha}) = (\Delta_{\beta}, \pi_{\beta})$. Now $\alpha$ and $\beta$ are condensed, so Theorem \ref{thm:numerical-char-color-compressed-complex} yields $\alpha = \beta$.

Finally, suppose instead that $\alpha \preceq \beta$ and $\beta \preceq \gamma$. Let $\alpha^{\prime}$ be the ${\bf b}$-twin of $\alpha$, let $\widetilde{\alpha}$ be the ${\bf c}$-twin of $\alpha$, and let $\widetilde{\beta}$ be the ${\bf c}$-twin of $\beta$. Since $\alpha \wedge \beta$ and $\beta \wedge \gamma$ are compressed-like and compatible, Theorem \ref{thm:compressed-like-compatible=>color-compressed} says $(\Delta_{\alpha \wedge \beta}, \pi_{\alpha \wedge \beta})$ and $(\Delta_{\beta \wedge \gamma}, \pi_{\beta \wedge \gamma})$ are both color-compressed. By the construction of $\alpha \wedge \beta$, we get $\Delta_{\alpha^{\prime}} \subseteq \Delta_{\beta}$, $\Delta_{\widetilde{\alpha}} \subseteq \Delta_{\widetilde{\beta}}$, $\pi_{\alpha^{\prime}} \leq \pi_{\beta}$, and $\pi_{\widetilde{\alpha}} \leq \pi_{\widetilde{\beta}}$. Similarly, the construction of $\beta \wedge \gamma$ yields $\Delta_{\widetilde{\beta}} \subseteq \Delta_{\gamma}$ and $\pi_{\widetilde{\beta}} \leq \pi_{\gamma}$, thus $\Delta_{\widetilde{\alpha}} \subseteq \Delta_{\gamma}$ and $\pi_{\widetilde{\alpha}} \leq \pi_{\gamma}$. Consequently, the construction of $\alpha \wedge \gamma$ yields $(\Delta_{\alpha \wedge \gamma}, \pi_{\alpha \wedge \gamma})$ is a pure color-compressed $({\bf c},2)$-balanced complex, so since $\alpha \wedge \gamma$ is condensed, Remark \ref{remark:distinct-Mac-trees-yield-same-complex} tells us the Macaulay tree induced by the shedding tree of $(\Delta_{\alpha \wedge \gamma}, \pi_{\alpha \wedge \gamma})$ is precisely $\alpha \wedge \gamma$, and Theorem \ref{thm:numerical-char-color-compressed-complex} then yields $\alpha \wedge \gamma$ is compatible, i.e. $\alpha \preceq \gamma$.
\end{proof}

\begin{proposition}\label{prop:case-n=1-preceq-equiv}
Let $k, N, N^{\prime} \in \mathbb{P}$, let $\alpha$ be the generalized $(k+1)$-Macaulay representation of $N$, and let $\alpha^{\prime}$ be the generalized $k$-Macaulay representation of $N^{\prime}$. Then $\alpha \preceq \alpha^{\prime}$ if and only if $\partial_{[-1]}(\alpha) \leq N^{\prime}$.
\end{proposition}

\begin{proof}
Let the condensation of the $k$-twin of $\alpha$ be $\alpha^{\prime\prime}$, which by construction is a generalized $k$-Macaulay representation of $\partial_{[-1]}(\alpha)$. By definition, $\alpha \preceq \alpha^{\prime}$ if and only if $\alpha \wedge \alpha^{\prime}$ is compatible, or equivalently, $\Delta_{\alpha^{\prime\prime}} \subseteq \Delta_{\alpha^{\prime}}$. Now, $\alpha^{\prime}, \alpha^{\prime\prime}$ are both generalized $k$-Macaulay representations, so by Theorem \ref{thm:numerical-char-color-compressed-complex}, $\Delta_{\alpha^{\prime}}$, $\Delta_{\alpha^{\prime\prime}}$ are both pure compressed $(k-1)$-dimensional simplicial complexes and hence each completely determined by its number of facets. This implies $\Delta_{\alpha^{\prime\prime}} \subseteq \Delta_{\alpha^{\prime}}$ if and only if $f_{k-1}(\Delta_{\alpha^{\prime\prime}}) \leq f_{k-1}(\Delta_{\alpha^{\prime}})$, which by Proposition \ref{prop:Macaulay-tree=>shedding-tree} is equivalent to $\partial_{[-1]}(\alpha) \leq N^{\prime}$.
\end{proof}

Finally, we characterize the fine $f$-vectors of colored complexes.

\begin{theorem}\label{thm:main-theorem-colored-complexes}
Let ${\bf a} \in \mathbb{P}^n$, and let $f = \{f_{\bf b}\}_{{\bf b} \leq {\bf a}}$ be an array of integers. Then the following are equivalent:
\begin{enum*}
\item $f$ is the fine $f$-vector of an ${\bf a}$-colored complex. \label{main-theorem:condition1}
\item $f$ is the fine $f$-vector of a color-compressed ${\bf a}$-colored complex. \label{main-theorem:condition2}
\item $f$ is the fine $f$-vector of a color-shifted ${\bf a}$-colored complex. \label{main-theorem:condition3}
\item $f_{{\bf 0}_n} = 1$, $f_{{\bf a}^{\prime}} = 0$ for all ${\bf a}^{\prime} \not\geq {\bf 0}_n$, $f_{\boldsymbol{\delta}_{i,n}} > 0$ for all $i\in [n]$, and there is an array $\{\alpha_{\bf b}\}_{{\bf 0}_n < {\bf b} \leq {\bf a}}$ such that each $\alpha_{{\bf b}}$ is a generalized ${\bf b}$-Macaulay representation of $f_{\bf b}$, and $\alpha_{\bf b} \preceq \alpha_{{\bf b}^{\prime}}$, for all ${\bf b}, {\bf b}^{\prime}$ satisfying ${\bf 0}_n < {\bf b}^{\prime} \lessdot {\bf b} \leq {\bf a}$.\label{main-theorem:condition4}
\end{enum*}

\end{theorem}

\begin{proof}
The equivalences \ref{main-theorem:condition1} $\Leftrightarrow$ \ref{main-theorem:condition2} $\Leftrightarrow$ \ref{main-theorem:condition3} follow immediately from Corollary \ref{cor:ColorCompressionPreservesFineFVector} and Remark \ref{remark:color-shifted=color-compressed-in-1n-case}, thus we are left to show that \ref{main-theorem:condition1} $\Leftrightarrow$ \ref{main-theorem:condition4}. Let $(\Delta, \pi)$ be an arbitrary colored-compressed ${\bf a}$-colored complex, and assume without loss of generality that $\pi = ([\lambda_1]_1, \dots, [\lambda_n]_n)$, where $\lambda_1, \dots, \lambda_n \in \mathbb{P}$. Clearly $f_{{\bf 0}_n}(\Delta) = 1$ and $f_{{\bf a}^{\prime}}(\Delta) = 0$ for all ${\bf a}^{\prime} \not\geq {\bf 0}_n$. By definition, $f_{\boldsymbol{\delta}_{i,n}}(\Delta) = \lambda_i > 0$ for each $i\in [n]$. For any ${\bf b}\in \mathbb{N}^n$ satisfying ${\bf 0}_n < {\bf b} \leq {\bf a}$, define $\Delta_{{\bf b}} := \big\langle \mathcal{F}_{\bf b}(\Delta) \big\rangle$, let $U_{{\bf b}}$ denote the vertex set of $\Delta_{{\bf b}}$, and let $\pi_{{\bf b}} := \pi \cap U_{\bf b}$. By default, set $U_{\bf b} = \emptyset$ if $\mathcal{F}_{\bf b}(\Delta) = \emptyset$. Notice that if $\mathcal{F}_{\bf b}(\Delta) \neq \emptyset$, then $(\Delta_{{\bf b}}, \overline{\pi_{{\bf b}}})$ is a pure color-compressed $\overline{{\bf b}}$-balanced complex, thus by Theorem \ref{thm:numerical-char-color-compressed-complex}, the $\overline{{\bf b}}$-Macaulay tree induced by the shedding tree of $(\Delta_{{\bf b}}, \overline{\pi_{{\bf b}}})$ is a generalized $\overline{{\bf b}}$-Macaulay representation, which we denote by $\overline{\alpha}_{{\bf b}} = (T_{\bf b}, \overline{\varphi_{\bf b}})$. We can then extend each such $\overline{\alpha_{{\bf b}}}$ to a generalized ${\bf b}$-Macaulay representation $\alpha_{{\bf b}} = (T_{\bf b}, \varphi_{\bf b})$ as follows.

Given ${\bf b} = (b_1, \dots, b_n)$, let $m$ be the number of non-zero entries in ${\bf b}$, and let $I := \{i_1, \dots, i_m\} \subseteq [n]$ be the set of indices with $i_1 < \dots < i_m$, such that $b_{i_1}, \dots, b_{i_m}$ are all the non-zero entries of ${\bf b}$. Let $(r_0, L, Y)$ be the $(\text{root},1,3)$-triple of $T_{\bf b}$, and define the vertex labeling $\varphi_{\bf b}$ of $T_{\bf b}$ by (i): $\varphi_{\bf b}(r_0) = {\bf b}$; (ii): $\varphi_{\bf b}(y) = i_{(\overline{\varphi_{\bf b}}(y))} \in I$ for all $y\in Y$; (iii): $(\varphi_{\bf b})^{[i_t]}(u) = (\overline{\varphi_{\bf b}})^{[t]}(u)$ for all $t\in [m]$, $u\in L$; and (iv): $(\varphi_{\bf b})^{[j]}(u) = \lambda_j$ for all $j\in [n]\backslash I$, $u\in L$. We can check that $\alpha_{\bf b}$ is a ${\bf b}$-Macaulay tree of $f_{\bf b}(\Delta)$, and the fact that $\overline{\alpha}_{\bf b}$ is a generalized Macaulay representation easily implies $\alpha_{\bf b}$ is also a generalized Macaulay representation. As for the case $\mathcal{F}_{\bf b}(\Delta) = \emptyset$, define $\alpha_{\bf b} := \alpha_{\emptyset}$.

Let ${\bf b}, {\bf b}^{\prime} \in \mathbb{N}^n$ satisfy ${\bf 0}_n < {\bf b}^{\prime} \lessdot {\bf b} \leq {\bf a}$. We want to show that $\alpha_{\bf b} \preceq \alpha_{{\bf b}^{\prime}}$. Clearly, if $\mathcal{F}_{\bf b}(\Delta) = \emptyset$, then $\alpha_{\bf b} = \alpha_{\emptyset} \preceq \alpha_{{\bf b}^{\prime}}$. Since $\mathcal{F}_{{\bf b}^{\prime}}(\Delta) = \emptyset$ implies $\mathcal{F}_{\bf b}(\Delta) = \emptyset$, we can assume $\mathcal{F}_{\bf b}(\Delta)$ and $\mathcal{F}_{{\bf b}^{\prime}}(\Delta)$ are non-empty, i.e. $\alpha_{\bf b}$ and $\alpha_{{\bf b}^{\prime}}$ are both non-trivial. Let $\widetilde{\alpha}_{{\bf b}^{\prime}, {\bf b}}$ denote the ${\bf b}^{\prime}$-twin of $\alpha_{\bf b}$. Theorem \ref{thm:numerical-char-color-compressed-complex} tells us $\Delta_{{\bf b}^{\prime}} = \Delta_{\alpha_{{\bf b}^{\prime}}}$ and $\big\langle \mathcal{F}_{{\bf b}^{\prime}}(\Delta_{{\bf b}}) \big\rangle = \Delta_{\widetilde{\alpha}_{{\bf b}^{\prime}, {\bf b}}}$, while it follows from definition that $\big\langle \mathcal{F}_{{\bf b}^{\prime}}(\Delta_{{\bf b}}) \big\rangle \subseteq \Delta_{{\bf b}^{\prime}}$. Thus, by the construction of $\alpha_{\bf b} \wedge \alpha_{{\bf b}^{\prime}}$, we get that $(\Delta_{\alpha_{\bf b} \wedge \alpha_{{\bf b}^{\prime}}}, \pi_{\alpha_{\bf b} \wedge \alpha_{{\bf b}^{\prime}}})$ is a pure color-compressed $({\bf b}^{\prime},2)$-balanced complex, therefore Theorem \ref{thm:numerical-char-color-compressed-complex} tells us $\alpha_{\bf b} \wedge \alpha_{{\bf b}^{\prime}}$ is compatible, i.e. $\alpha_{\bf b} \preceq \alpha_{{\bf b}^{\prime}}$.

Conversely, suppose instead statement \ref{main-theorem:condition4} of the theorem is true. By assumption, $\alpha_{\bf b} \preceq \alpha_{{\bf b}^{\prime}}$ for every ${\bf b}, {\bf b}^{\prime} \in \mathbb{N}^n$ satisfying ${\bf 0}_n < {\bf b}^{\prime} \lessdot {\bf b} \leq {\bf a}$, so since $\preceq$ is a partial order (Proposition \ref{prop:preceq-partial-order}), it follows that $\alpha_{\bf b} \preceq \alpha_{\bf c}$ for every ${\bf b}, {\bf c} \in \mathbb{N}^n$ satisfying ${\bf 0}_n < {\bf c} \leq {\bf b} \leq {\bf a}$. Define $\pi_{\Gamma} := ([f_{\boldsymbol{\delta}_{1,n}}]_1, \dots, [f_{\boldsymbol{\delta}_{n,n}}]_n)$, and let $\Gamma$ be the union of all simplicial complexes $\Delta_{\alpha_{\bf b}}$ over all $n$-tuples ${\bf b}$ satisfying ${\bf 0}_n < {\bf b} \leq {\bf a}$, $f_{\bf b}>0$. For each such ${\bf b}$, the ${\bf b}$-balanced complex $(\Delta_{\alpha_{\bf b}}, \overline{\pi_{{\bf b}}})$ is color-compressed, hence $(\Gamma, \pi_{\Gamma})$ is a color-compressed ${\bf a}$-colored complex.

Now, choose arbitrary ${\bf b}, {\bf c}\in \mathbb{N}^n$ such that ${\bf 0}_n < {\bf c} \leq {\bf b} \leq {\bf a}$. Note that $\alpha_{\emptyset} \preceq \alpha_{\bf b}$, while $\alpha_{\bf b} = \alpha_{\emptyset}$ if and only if $f_{\bf b} = 0$, so $f_{\bf b} = 0$ if and only if $f_{\bf b}(\Gamma) = 0$. Consequently, if $f_{\bf b} > 0$, then $f_{\bf c} > 0$ and $\mathcal{F}_{\bf c}(\Delta_{\alpha_{\bf b}}) \subseteq \Delta_{\alpha_{\bf c}}$, which imply $\mathcal{F}_{\bf c}(\Gamma) = \mathcal{F}_{\bf c}(\Delta_{\alpha_{\bf c}})$. Therefore by Proposition \ref{prop:Macaulay-tree=>shedding-tree}, $\alpha_{\bf b}$ is a generalized ${\bf b}$-Macaulay representation of $f_{\bf b}(\Gamma)$ for every ${\bf b} \in \mathbb{N}^n$ satisfying ${\bf 0}_n < {\bf b} \leq {\bf a}$, i.e. $f$ is the fine $f$-vector of $(\Gamma, \pi_{\Gamma})$.
\end{proof}

In the special case $n=1$, an ${\bf a}$-colored complex is precisely a $(k-1)$-dimensional simplicial complex for some $k\in [a_1]$, so in view of Theorem \ref{thm:main-theorem-colored-complexes} (cf. Proposition \ref{prop:case-n=1-preceq-equiv}), we can restate the Kruskal-Katona theorem as follows:

\begin{theorem}
Let $d\in \mathbb{P}$, and let $f = (f_0, \dots, f_{d-1}) \in \mathbb{Z}^d$. Then the following are equivalent:
\begin{enum*}
\item $f$ is the $f$-vector of a $(d-1)$-dimensional simplicial complex.
\item $f$ is the $f$-vector of a compressed $(d-1)$-dimensional simplicial complex.
\item $f$ is the $f$-vector of a shifted $(d-1)$-dimensional simplicial complex.
\item $f \in \mathbb{P}^d$ and $\alpha_d \preceq \dots \preceq \alpha_1$, where each $\alpha_k$ is the (unique) generalized $k$-Macaulay representation of $f_{k-1}$.
\end{enum*}
\end{theorem}

As for the case ${\bf a} = {\bf 1}_n$, Theorem \ref{thm:main-theorem-colored-complexes} yields the following corollary.

\begin{corollary}\label{cor:numer-char-colored-complex-type1n}
Let $n\in \mathbb{P}$, let $f = \{f_S\}_{S \subseteq [n]}$ be an array of integers, and define $\phi: 2^{[n]} \to \{0,1\}^n$ by $S \mapsto (s_1, \dots, s_n)$, where $s_i = 1$ if and only if $i\in S$. Then the following are equivalent:
\begin{enum*}
\item $f$ is the flag $f$-vector of a ${\bf 1}_n$-colored complex.
\item $f_{\emptyset} = 1$, $f_{\{i\}} > 0$ for all $i\in [n]$, and there is an array $\{\alpha_S\}_{S\subseteq [n], S\neq \emptyset}$ such that each $\alpha_{S}$ is a generalized $\phi(S)$-Macaulay representation of $f_S$, and $\alpha_{S} \preceq \alpha_{S^{\prime}}$ for all non-empty $S, S^{\prime} \subseteq [n]$ satisfying $|S| = |S^{\prime}| + 1$.
\end{enum*}
\end{corollary}

By Theorem \ref{thm:BFSthm}, the flag $h$-vector of a $(d-1)$-dimensional completely balanced CM complex is the flag $f$-vector of a ${\bf 1}_d$-colored complex, hence by Corollary \ref{cor:numer-char-colored-complex-type1n}, we get the following different but equivalent numerical characterization of completely balanced CM complexes.

\begin{corollary}\label{cor:equiv-flag-h-vector-char}
Let $d\in \mathbb{P}$, let $h = \{h_S\}_{S \subseteq [d]}$ be an array of integers, and define $\phi: 2^{[d]} \to \{0,1\}^d$ by $S \mapsto (s_1, \dots, s_d)$, where $s_i = 1$ if and only if $i\in S$. Then the following are equivalent:
\begin{enum*}
\item $h$ is the flag $h$-vector of a $(d-1)$-dimensional completely balanced CM complex.
\item $h_{\emptyset} = 1$, $h_{\{i\}} > 0$ for all $i\in [d]$, and there is an array $\{\alpha_S\}_{S\subseteq [d], S\neq \emptyset}$ such that each $\alpha_{S}$ is a generalized $\phi(S)$-Macaulay representation of $h_S$, and $\alpha_{S} \preceq \alpha_{S^{\prime}}$ for all non-empty $S, S^{\prime} \subseteq [d]$ satisfying $|S| = |S^{\prime}| + 1$.
\end{enum*}
\end{corollary}

\section*{Acknowledgements}
The author thanks Edward Swartz for his enthusiasm and patience in reading several drafts of this paper, as well as for giving many helpful comments and suggestions. The author also thanks Louis Billera for his insightful comments that helped improve the paper.

\bibliographystyle{plain}
\bibliography{References}

\end{document}